\pdfoutput=1
\documentclass{ucthesis}
\usepackage{amsmath}
\usepackage{amsthm}
\usepackage{graphicx}
\usepackage{bbm,amssymb}
\usepackage{bbold}
 \usepackage{makeidx}
 \usepackage{index}
 \usepackage{hyperref}
 \newindex{default}{idx}{ind}{Index of Notation}

\newcommand{\floor}[1]{\lfloor {#1} \rfloor}
\newcommand{\ceil}[1]{\lceil {#1} \rceil}

\newtheorem{theorem}{Theorem}[section]
\newtheorem{prop}[theorem]{Proposition}
\newtheorem{lemma}[theorem]{Lemma}
\newtheorem{corollary}[theorem]{Corollary}

\theoremstyle{remark}
\newtheorem*{remark}{Remark}

\theoremstyle{definition}
\newtheorem*{definition}{Definition}
\newtheorem*{question}{Question}
\newtheorem*{conjecture}{Conjecture}

\def\liminf{\mathop{\rm lim\,inf}\limits}
\def\Leb{\mathcal{L}}
\def\normal{{\bf n}}
\def\Points{{::}}
\def\div{{\text{div}\hspace{0.5ex}}}

\title{Limit Theorems for Internal Aggregation Models}
\author{Lionel Timothy Levine}
\degreeyear{2007}
\degreesemester{Fall}
\degree{Doctor of Philosophy}
\chair{~\\ Professor Yuval Peres}
\othermembers{Professor Lawrence C. Evans\\ Professor Elchanan Mossel}
\prevdegrees{A.B. Harvard University, 2002}
\field{Mathematics}
\campus{Berkeley}

\DeclareSymbolFont{AMSb}{U}{msb}{m}{n}
\DeclareMathSymbol{\C}{\mathbin}{AMSb}{"43} 
\DeclareMathSymbol{\EE}{\mathbin}{AMSb}{"45} 
\DeclareMathSymbol{\N}{\mathbin}{AMSb}{"4E} 
\DeclareMathSymbol{\PP}{\mathbin}{AMSb}{"50} 
\DeclareMathSymbol{\Q}{\mathbin}{AMSb}{"51} 
\DeclareMathSymbol{\R}{\mathbin}{AMSb}{"52} 
\DeclareMathSymbol{\Z}{\mathbin}{AMSb}{"5A}

\begin{document}

\maketitle
\approvalpage
\copyrightpage

\begin{abstract}
We study the scaling limits of three different aggregation models on $\Z^d$: internal DLA, in which particles perform random walks until reaching an unoccupied site; the rotor-router model, in which particles perform deterministic analogues of random walks; and the divisible sandpile, in which each site distributes its excess mass equally among its neighbors.  As the lattice spacing tends to zero, all three models are found to have the same scaling limit, which we describe as the solution to a certain PDE free boundary problem in $\R^d$.  In particular, internal DLA has a deterministic scaling limit.  We find that the scaling limits are quadrature domains, which have arisen independently in many fields such as potential theory and fluid dynamics.  Our results apply both to the case of multiple point sources and to the Diaconis-Fulton smash sum of domains.

In the special case when all particles start at a single site, we show that the scaling limit is a Euclidean ball in $\R^d$ and give quantitative bounds on the rate of convergence to a ball.  For the divisible sandpile, the error in the radius is bounded by a constant independent of the total starting mass.  For the rotor-router model in $\Z^d$, the inner error grows at most logarithmically in the radius $r$, while the outer error is at most order $r^{1-1/d} \log r$.  We also improve on the previously best known bounds of Le Borgne and Rossin in $\Z^2$ and Fey and Redig in higher dimensions for the shape of the classical abelian sandpile model. 

Lastly, we study the sandpile group of a regular tree whose leaves are collapsed to a single sink vertex, and determine the decomposition of the full sandpile group as a product of cyclic groups.  For the regular ternary tree of height $n$, for example, the sandpile group is isomorphic to $(\Z_3)^{2^{n-3}} \oplus (\Z_7)^{2^{n-4}} \oplus \ldots \oplus \Z_{2^{n-1}-1} \oplus \Z_{2^n-1}$.  We use this result to prove that rotor-router aggregation on the regular tree yields a perfect ball.
\end{abstract}

\begin{frontmatter}
\begin{dedication}
~ \\ ~ \\ ~ \\ ~ \\ ~ \\
\begin{center}
To my parents, Lance and Terri.
\end{center}
\end{dedication}

\begin{acknowledgements}
This work would not have been possible without the incredible support and insight of my advisor, Yuval Peres.  Yuval taught me not only a lot of mathematics, but also the tools, techniques, instincts and heuristics essential to the working mathematician.  Somehow, he even managed to teach me a little bit about life as well.  My entire philosophy and approach, and my sense of what is important in mathematics, are colored by his ideas.

I would like to thank Jim Propp for introducing me to this beautiful area of mathematics, and for many inspiring conversations over the years.  Thanks also to Scott Armstrong, Henry Cohn, Darren Crowdy, Craig Evans, Anne Fey, Chris Hillar, Wilfried Huss, Itamar Landau, Karola M\'{e}sz\'{a}ros, Chandra Nair, David Pokorny, Oded Schramm, Scott Sheffield, Misha Sodin, Kate Stange, Richard Stanley, Parran Vanniasegaram, Grace Wang, and David Wilson for many helpful conversations.
Richard Liang taught me how to write image files using C, so that I could write programs to generate many of the figures.  In addition, Itamar Landau and Yelena Shvets helped create several of the figures.

I also thank the NSF for supporting me with a Graduate Research Fellowship during much of the time period when this work was carried out.

Finally, I would like to thank my family and friends for supporting me and believing in me through the best and worst of times.  Your love and support mean more to me than I can possibly express in words.
\end{acknowledgements}

\tableofcontents
\printindex
\end{frontmatter}

\chapter{Introduction}

\section{Three Models with the Same Scaling Limit}
\label{scalinglimitintro}

Given finite sets $A, B \subset \Z^d$, Diaconis and Fulton \cite{DF} defined the \emph{smash sum} $A \oplus B$ as a certain random set whose cardinality is the sum of the cardinalities of $A$ and $B$.  Write $A \cap B = \{x_1, \ldots, x_k\}$.  To construct the smash sum, begin with the union $C_0 = A \cup B$ and for each $j=1,\ldots,k$ let
	\[ C_j = C_{j-1} \cup \{y_j\} \]
where $y_j$ is the endpoint of a simple random walk started at $x_j$ and stopped on exiting~$C_{j-1}$.  Then define $A\oplus B = C_k$.  The key observation of \cite{DF} is that the law of $A\oplus B$ does not depend on the ordering of the points $x_j$.  The sum of two squares in $\Z^2$ overlapping in a smaller square is pictured in Figure~\ref{smashsumfigure}.

\begin{figure}
\centering
\includegraphics[scale=1.15]{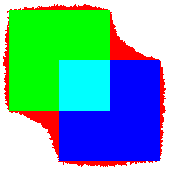}
\includegraphics[scale=1.15]{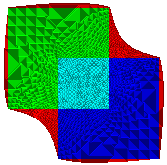}
\includegraphics[scale=1.15]{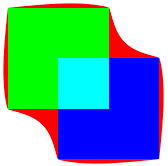}
\caption{Smash sum of two squares in $\Z^2$ overlapping in a smaller square, for internal DLA (top left), the rotor-router model (top right), and the divisible sandpile.
}
\label{smashsumfigure}
\end{figure}

In Theorem~\ref{DFsum}, below, we prove that as the lattice spacing goes to zero, the smash sum $A \oplus B$ has a deterministic scaling limit in $\R^d$.  Before stating our main results, we describe some related models and describe our technique for identifying their common scaling limit, which comes from the theory of free boundary problems in PDE.

The Diaconis-Fulton smash sum generalizes the model of \emph{internal diffusion-limited aggregation} (``internal DLA'') studied in \cite{LBG}, and in fact was part of the original motivation for that paper.  In classical internal DLA, we start with $n$ particles at the origin $o \in \Z^d$
\index{$o$, origin in $\Z^d$}
and let each perform a simple random walk until it reaches an unoccupied site.  The resulting random set of $n$ occupied sites in $\Z^d$ can be described as the $n$-fold smash sum of $\{o\}$ with itself.  We will use the term internal DLA to refer to particles which perform simple random walks in $\Z^d$ until reaching an unoccupied site, starting from an arbitrary initial configuration.  In this broader sense of the term, both the Diaconis-Fulton sum and the model studied in \cite{LBG} are particular cases of internal DLA. 

In defining the smash sum $A\oplus B$, various alternatives to random walk are possible.  \emph{Rotor-router walk} is a deterministic analogue of random walk, first studied by Priezzhev et al.\ \cite{PDDK} under the name ``Eulerian walkers.''  At each site in $\Z^2$ is a {\it rotor} pointing north, south, east or west.  A particle performs a nearest-neighbor walk on the lattice according to the following rule: during each time step, the rotor at the particle's current location is rotated clockwise by $90$ degrees, and the particle takes a step in the direction of the newly rotated rotor.  In higher dimensions, the model can be defined analogously by repeatedly cycling the rotors through an ordering of the $2d$ cardinal directions in $\Z^d$.  The sum of two squares in $\Z^2$ using rotor-router walk is pictured in Figure~\ref{smashsumfigure}; all rotors began pointing west.  The shading in the figure indicates the final rotor directions, with four different shades corresponding to the four possible directions.

The {\it divisible sandpile} model uses continuous amounts of mass in place of discrete particles.  A lattice site is {\it full} if it has mass at least $1$.  Any full site can {\it topple} by keeping mass $1$ for itself and distributing the excess mass equally among its neighbors.  At each time step, we choose a full site and topple it.  As time goes to infinity, provided each full site is eventually toppled, the mass approaches a limiting distribution in which each site has mass $\leq 1$.
Note that individual topplings do not commute.  However, the divisible sandpile is ``abelian'' in the sense that any sequence of topplings produces the same limiting mass distribution; this is proved in Lemma~\ref{abelianproperty}.  Figure~\ref{smashsumfigure} shows the limiting domain of occupied sites resulting from starting mass $1$ on each of two squares in $\Z^2$, and mass $2$ on the smaller square where they intersect.

Figure~\ref{smashsumfigure} raises a few natural questions: as the underlying lattice spacing becomes finer and finer, will the smash sum $A\oplus B$ tend to some limiting shape in $\R^d$, and if so, what is this shape?  Will it be the same limiting shape for all three models?  To see how we might identify the limiting shape, consider the divisible sandpile {\it odometer function}
	\begin{equation} \label{odomdefintro} u(x) = \text{total mass emitted from } x. \end{equation}
Since each neighbor $y\sim x$ emits an equal amount of mass to each of its $2d$ neighbors, the total  mass received by $x$ from its neighbors is $\frac{1}{2d} \sum_{y \sim x} u(y)$, hence
	\begin{equation} \label{netchangeinmass} \Delta u(x) = \nu(x) - \sigma(x) \end{equation}
where $\sigma(x)$ and $\nu(x)$ are the initial and final amounts of mass at $x$, respectively.  Here $\Delta$ is the discrete Laplacian in $\Z^d$, defined by 
	\begin{equation} \label{discretelaplaciandef} \Delta u(x) = \frac{1}{2d} \sum_{y \sim x} u(y) - u(x).
\index{$\Delta$, discrete Laplacian}
\end{equation}

Equation (\ref{netchangeinmass}) suggests the following approach to finding the limiting shape.  We first construct a function on $\Z^d$ whose Laplacian is $\sigma-1$; an example is the function
	\begin{equation} \label{theobstacleintro} \gamma(x) = - |x|^2 - \sum_{y \in \Z^d} g_1(x,y) \sigma(y) \end{equation}
where in dimension $d\geq 3$ the Green's function $g_1(x,y)$ is the expected number of times a simple random walk started at $x$ visits $y$ (in dimension $d=2$ we use the recurrent potential kernel in place of the Green's function).  Here $|x|$ denotes the Euclidean norm $(x_1^2+\ldots+x_d^2)^{1/2}$.
By (\ref{netchangeinmass}), since $\nu \leq 1$ the sum $u+\gamma$ is a superharmonic function on $\Z^d$; that is, $\Delta(u+\gamma) \leq 0$.  Moreover if $f \geq \gamma$ is any superharmonic function lying above $\gamma$, then $f - \gamma - u$ is superharmonic on the domain $D = \{x\in \Z^d | \nu(x)=1 \}$ of fully occupied sites, and nonnegative outside $D$, hence nonnegative everywhere.  Thus we have proved the following lemma.

\begin{lemma}
\label{discretemajorantintro}
Let $\sigma$ be a nonnegative function on $\Z^d$ with finite support.  Then the odometer function (\ref{odomdefintro}) for the divisible sandpile started with mass $\sigma(x)$ at each site $x$ is given by
	\[ u = s - \gamma \]
where $\gamma$ is given by (\ref{theobstacleintro}), and
	\[ s(x) = \inf \{f(x) | f \text{ is superharmonic on $\Z^d$ and } f\geq \gamma \} \]
is the {\it least superharmonic majorant} of $\gamma$.
\end{lemma}

\begin{figure}
\centering
\includegraphics[scale=.4]{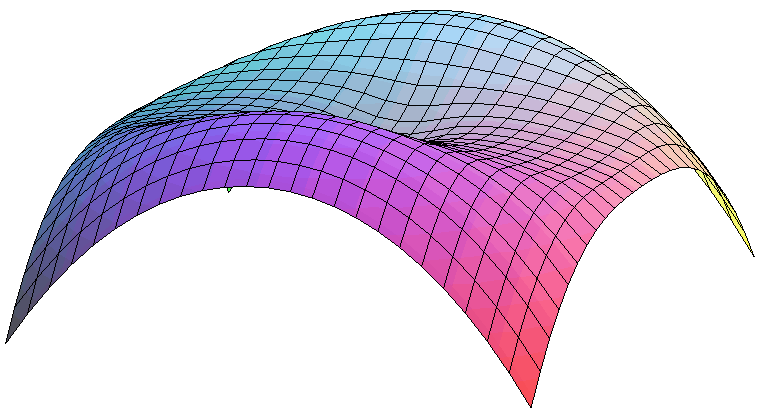} \\
\includegraphics[scale=.4]{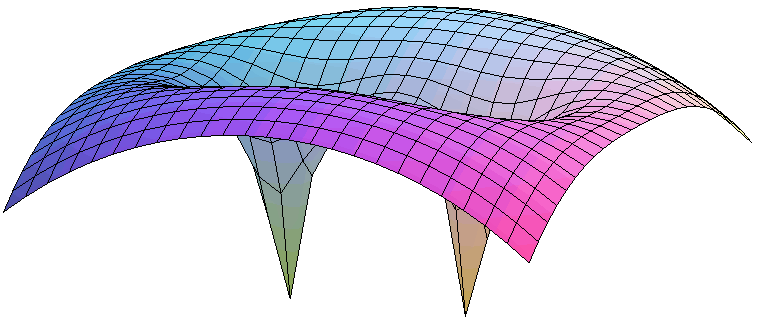}
\caption{The obstacles $\gamma$ corresponding to starting mass $1$ on each of two overlapping disks (top) and mass $100$ on each of two nonoverlapping disks.}
\label{twosourceobstacle}
\end{figure}

Lemma~\ref{discretemajorantintro} allows us to formulate the problem in a way which translates naturally to the continuum.  Given a function $\sigma$ on $\R^d$ representing the initial mass density, by analogy with (\ref{theobstacleintro}) we define the {\it obstacle}
	\[ \gamma(x) = -|x|^2 - \int_{\R^d} g(x,y) \sigma(y) dy \]
where $g(x,y)$ is the Green's function on $\R^d$ proportional to $|x-y|^{2-d}$ in dimensions $d \geq 3$ and to $-\log |x-y|$ in dimension two.  We then let
	\[ s(x) = \inf \{f(x) | f \text{ is continuous, superharmonic and }  f \geq \gamma \}. \]
The {\it odometer function} for $\sigma$ is then given by $u = s - \gamma$, and the final domain of occupied sites is given by
	\begin{equation} \label{thenoncoincidencesetintro} D = \{x \in \R^d| s(x)>\gamma(x) \}. \end{equation}
This domain $D$ is called the {\it noncoincidence set for the obstacle problem} with obstacle $\gamma$; 
for an in-depth discussion of the obstacle problem, see \cite{Friedman}.

If $A,B$ are bounded open sets in $\R^d$, we define the \emph{smash sum} of $A$ and $B$ as
\index{$A \oplus B$, smash sum}
	\begin{equation} \label{smashsumdef} A\oplus B = A\cup B \cup D \end{equation}
where $D$ is given by (\ref{thenoncoincidencesetintro}) with $\sigma = 1_A + 1_B$.  In the two-dimensional setting, an alternative definition of the smash sum in terms of quadrature identities is mentioned in \cite{Gustafsson88}.

In this thesis we prove, among other things, that if any of our three aggregation models is run on finer and finer lattices with initial mass densities converging in an appropriate sense to $\sigma$, the resulting domains of occupied sites will converge in an appropriate sense to the domain $D$ given by (\ref{thenoncoincidencesetintro}).  We will always work in dimension $d \geq 2$; for a discussion of the rotor-router model in one dimension, see \cite{undergradthesis}.  

Let us define the appropriate notion of convergence of domains, which amounts essentially to convergence in the Hausdorff metric.
Fix a sequence $\delta_n \downarrow 0$ representing the lattice spacing. \index{$\delta_n$, lattice spacing} Given domains $A_n \subset \delta_n \Z^d$ and $D \subset \R^d$, write $A_n \to D$ if for any $\epsilon>0$
	\begin{equation} \label{convergenceinthehausdorffmetric} D_\epsilon \cap \delta_n \Z^d \subset 	A_n \subset D^{\epsilon} 
	\end{equation}
for all sufficiently large $n$.  Here 
	\begin{equation} \label{innerepsilonneighborhood} D_\epsilon = \{x \in D \,|\, B(x,\epsilon) \subset D \} \index{$D_\epsilon$, inner $\epsilon$-neighborhood} \end{equation}
and
	\[ D^\epsilon = \{x \in \R^d \,|\, B(x,\epsilon) \cap D \neq \emptyset \} \index{$D^\epsilon$, outer $\epsilon$-neighborhood}	\]
are the inner and outer $\epsilon$-neighborhoods of $D$.  For $x \in \delta_n \Z^d$ we write $x^\Box = \left[x+\frac{\delta_n}{2}, x-\frac{\delta_n}{2}\right]^d$.  For $t \in \R$ write $\lfloor t \rceil$ for the closest integer to $t$.  

Throughout this thesis, to avoid trivialities we work in dimension $d \geq 2$.  Our first main result is the following.

\begin{theorem}
\label{intromain}
Let $\Omega \subset \R^d$ be a bounded open set, and let $\sigma : \R^d \rightarrow \Z_{\geq 0}$ be a bounded function which is continuous almost everywhere, satisfying $\{\sigma \geq 1\} = \bar{\Omega}$.  Let $D_n, R_n, I_n$ be the domains of occupied sites formed from the divisible sandpile, rotor-router model, and internal DLA, respectively, in the lattice $\delta_n \Z^d$ started from source density 
	\[ \sigma_n(x) = \left\lfloor \delta_n^{-d} \int_{x^\Box} \sigma(y) dy \right\rceil. \]
Then as $n \to \infty$
	\[ D_n, R_n \to D \cup \Omega; \]
and if $\delta_n \leq 1 / \log n$, then with probability one
	\[ I_n \to D \cup \Omega \]
where $D$ is given by (\ref{thenoncoincidencesetintro}), and the convergence is in the sense of~(\ref{convergenceinthehausdorffmetric}).
\end{theorem}

\begin{remark}
When forming the rotor-router domains $R_n$, the initial rotors in each lattice $\delta_n \Z^d$ may be chosen arbitrarily.
\end{remark}

We prove a somewhat more general form of Theorem~\ref{intromain} which allows for some flexibility in how the discrete density $\sigma_n$ is constructed from $\sigma$.  In particular, taking $\sigma = 1_{\bar{A}} + 1_{\bar{B}}$ we obtain the following theorem, which explains the similarity of the three smash sums pictured in Figure~\ref{smashsumfigure}.

\begin{theorem}
\label{DFsum}
Let $A,B \subset \R^d$ be bounded open sets whose boundaries have measure zero.  Let $D_n, R_n, I_n$ be the smash sum of $A \cap \delta_n \Z^d$ and $B \cap \delta_n \Z^d$, formed using divisible sandpile, rotor-router and internal DLA dynamics, respectively.  Then as $n \to \infty$
	\[ D_n, R_n \to A \oplus B; \]
and if $\delta_n \leq 1 / \log n$, then with probability one
	\[ I_n \to A \oplus B \]
where $A \oplus B$ is given by (\ref{smashsumdef}), and the convergence is in the sense of~(\ref{convergenceinthehausdorffmetric}).
\end{theorem}


For the divisible sandpile, Theorem~\ref{intromain} can be generalized by dropping the requirement that $\sigma$ be integer valued; see Theorem~\ref{domainconvergence} for the precise statement.  
Taking $\sigma$ real-valued is more problematic in the case of the rotor-router model and internal DLA, since these models work with discrete particles.  Still, one might wonder if, for example, given a domain $A \subset \R^d$, starting each even site in $A \cap \delta_n \Z^d$ with one particle and each odd site with two particles, the resulting domains $R_n, I_n$ would converge to the noncoincidence set $D$ for density $\sigma = \frac32 1_A$.  This is in fact the case: if $\sigma_n$ is a density on $\delta_n \Z^d$, as long as a certain ``smoothing'' of $\sigma_n$ converges to $\sigma$, the rotor-router and internal DLA domains started from source density $\sigma_n$ will converge to $D$.  See Theorems~\ref{rotordomainconvergence} and~\ref{IDLAconvergence} for the precise statements.

\section{Single Point Sources}

One interesting case not covered by Theorems~\ref{intromain} and~\ref{DFsum} is the case of point sources.  Lawler, Bramson and Griffeath \cite{LBG} showed that the scaling limit of internal DLA in $\Z^d$ with a single point source of particles is a Euclidean ball.  In chapter~\ref{pointsource}, we prove analogous results for rotor-router aggregation and the divisible sandpile, and give quantitative bounds on the rate of convergence to a ball.  Let $A_n$ be the domain of $n$ sites in $\Z^d$ formed from rotor-router aggregation starting from a point source of $n$ particles at the origin.  Thus $A_n$ is defined inductively by the rule
	\[ A_n = A_{n-1} \cup \{x_n\} \]
where $x_n$ is the endpoint of a rotor-router walk started at the origin in $\Z^d$ and stopped on first exiting $A_{n-1}$.  
For example, in $\Z^2$, if all rotors initially point north, the sequence will begin $A_1 = \{(0,0)\}$, $A_2 = \{(0,0),(1,0)\}$, $A_3 = \{(0,0),(1,0),(0,-1)\}$.  The region $A_{1,000,000}$ is pictured in Figure~\ref{rotor1m}.

\begin{figure}
\centering
\includegraphics[scale=.28]{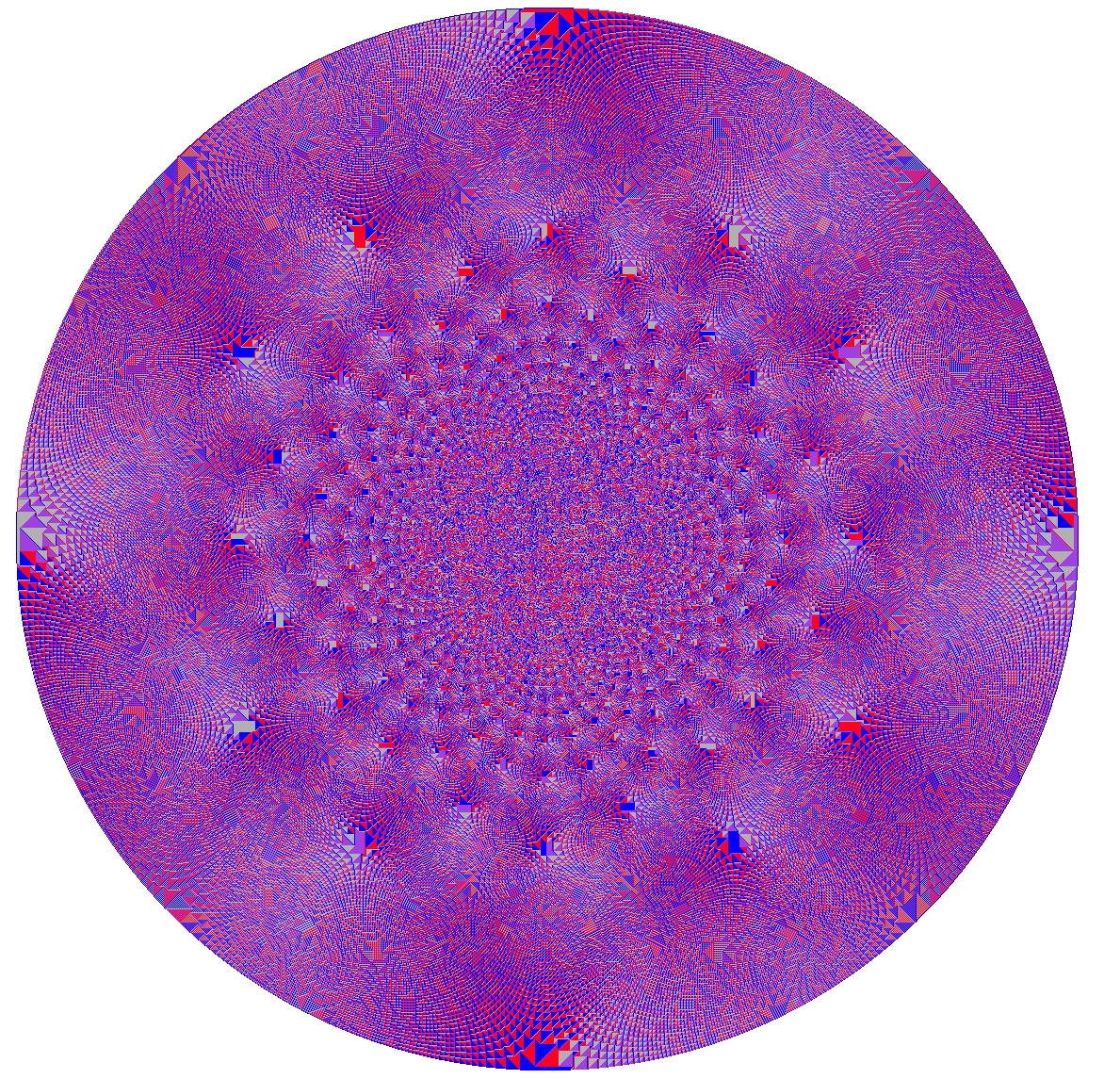}
\caption{Rotor-router aggregate of one million particles in $\Z^2$.  Each site is colored according to the direction of its rotor.}
\label{rotor1m}
\end{figure}

Jim Propp observed from simulations in two dimensions that the regions $A_n$ are extraordinarily close to circular, and asked why this was so \cite{Kleber,Propp}.  Despite the impressive empirical evidence for circularity, the best result known until now \cite{LP} says only that if $A_n$ is rescaled to have unit volume, the volume of the symmetric difference of $A_n$ with a ball of unit volume tends to zero as a power of $n$, as $n \uparrow \infty$.  The main outline of the argument is summarized in \cite{intelligencer}.  Fey and Redig  \cite{FR} also show that $A_n$ contains a diamond.  In particular, these results do not rule out the possibility of ``holes'' in $A_n$ far from the boundary or of long tendrils extending far beyond the boundary of the ball, provided the volume of these features is negligible compared to $n$.

Our main result on the shape of rotor-router aggregation with a single point source is the following, which rules out the possibility of holes far from the boundary or of long tendrils.  For $r\geq 0$ let
	\[ B_r = \{ x \in \Z^d ~:~ |x| < r \}. \]
\index{$B_r$, discrete ball}

\begin{theorem}
\label{rotorcircintro}
Let $A_n$ be the region formed by rotor-router aggregation in $\Z^d$ starting from $n$ particles at the origin and any initial rotor state.  There exist constants $c,c'$ depending only on $d$, such that
	\[ B_{r-c\log r} \subset A_n \subset B_{r (1+c'r^{-1/d}\log r)} \]
where $r=(n/\omega_d)^{1/d}$, and $\omega_d$ is the volume of the unit ball in $\R^d$.
\end{theorem}

We remark that the same result holds when the rotors evolve according to stacks of bounded discrepancy; see the remark following Lemma~\ref{odomflow}.

By way of comparison with Theorem~\ref{rotorcircintro}, if $I_n$ is the internal DLA region formed from $n$ particles started at the origin, the best known bounds \cite{Lawler95} are (up to logarithmic factors)
	\[ B_{r-r^{1/3}} \subset I_n \subset B_{r+r^{1/3}} \]
for all sufficiently large $n$, with probability one.

Our next result treats the divisible sandpile with all mass initially concentrated at a point source.  The resulting domain of fully occupied sites is extremely close to a ball; in fact, the error in the radius is bounded independent of the total mass.

\begin{theorem}
\label{divsandcircintro}
For $m \geq 0$ let $D_m \subset \Z^d$ be the domain of fully occupied sites for the divisible sandpile formed from a pile of mass $m$ at the origin.  There exist constants $c,c'$ depending only on $d$, such that
	\[ B_{r-c} \subset D_m \subset B_{r+c'}, \]
where $r = (m/\omega_d)^{1/d}$ and $\omega_d$ is the volume of the unit ball in $\R^d$.
\end{theorem} 

The divisible sandpile is similar to the ``oil game'' studied by Van den Heuvel \cite{VdH}.  In the terminology of \cite{FR}, it also corresponds to the $h \rightarrow -\infty$ limit of the classical abelian sandpile (defined below), that is, the abelian sandpile started from the initial condition in which every site has a very deep ``hole.''

\begin{figure}
\centering
\includegraphics[scale=.45]{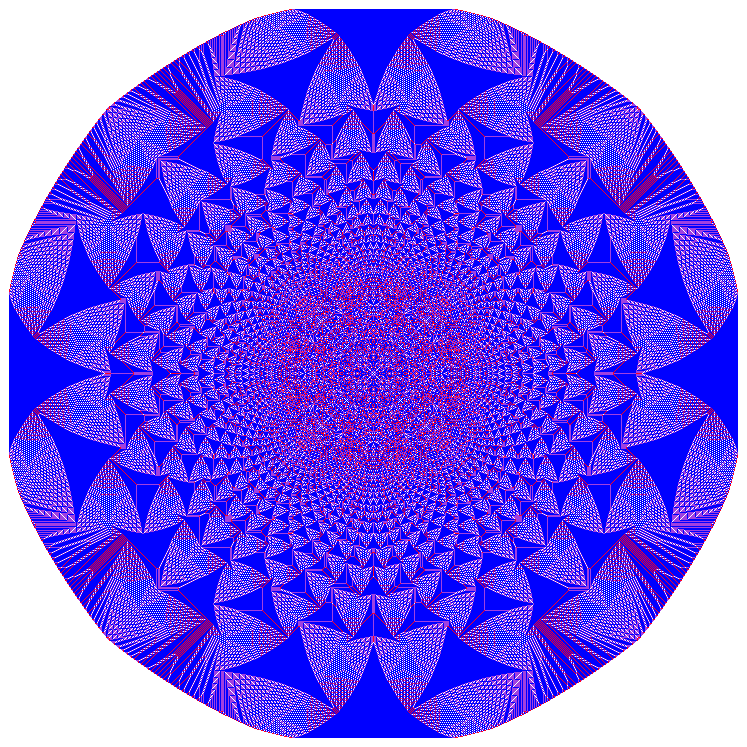}
\caption{Classical abelian sandpile aggregate of one million particles in $\Z^2$.  Colors represent the number of grains at each site.}
\label{sandpile1m}
\end{figure}

In the classical \emph{abelian sandpile model} \cite{BTW}, each site in $\Z^d$ has an integer number of grains of sand; if a site has at least $2d$ grains, it {\it topples}, sending one grain to each neighbor.  If $n$ grains of sand are started at the origin in $\Z^d$, write $S_n$ for the set of sites that are visited during the toppling process; in particular, although a site may be empty in the final state, we include it in $S_n$ if it was occupied at any time during the evolution to the final state.

\begin{figure}
\centering
\includegraphics[scale=.3]{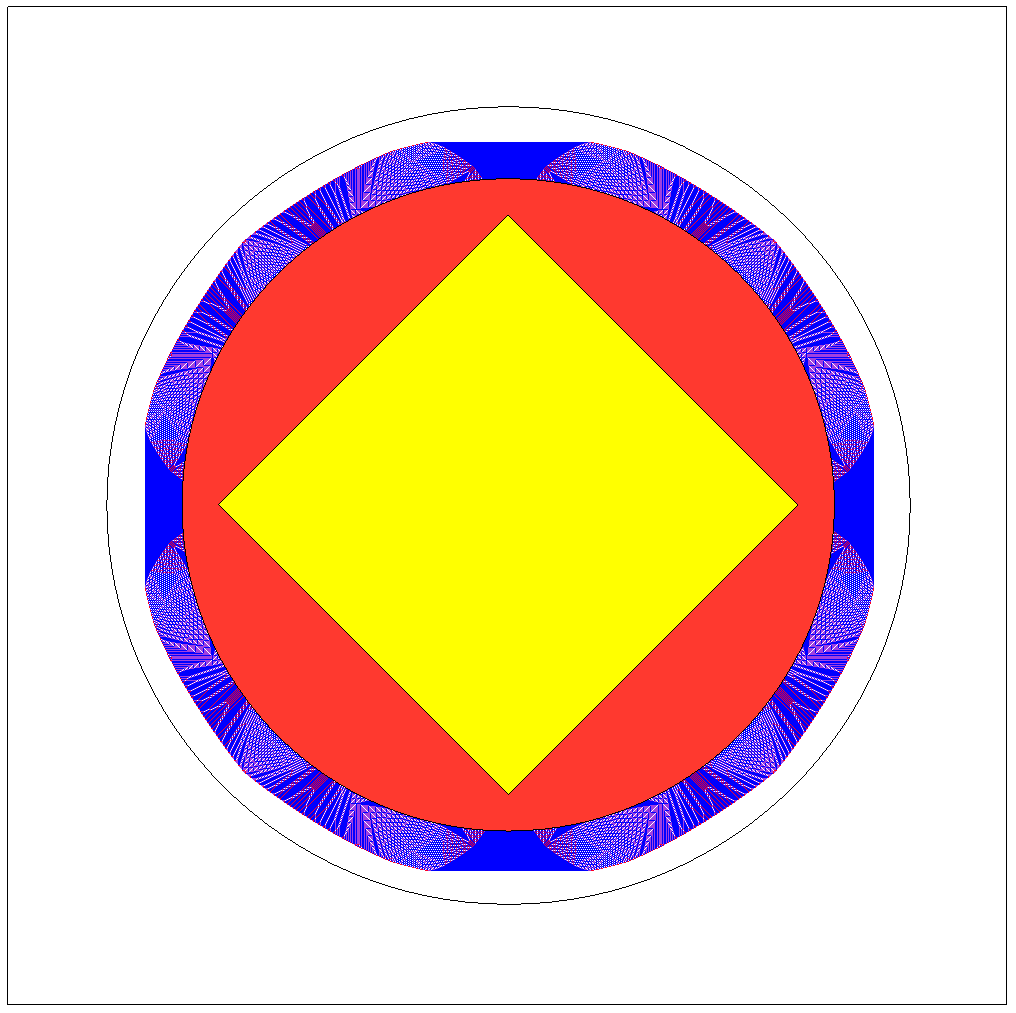}
\caption{Known bounds on the shape of the classical abelian sandpile in $\Z^2$.  The inner diamond and outer square are due to Le Borgne and Rossin \cite{LBR}; the inner and outer circles are those in Theorem~\ref{sandpilecircintro}.}
\label{sandpilebounds}
\end{figure}

Until now the best known constraints on the shape of $S_n$ in two dimensions were due to Le Borgne and Rossin \cite{LBR}, who proved that
	\[ \{x \in \Z^2 \,|\, x_1+x_2 \leq \sqrt{n/12}-1 \} \subset S_n \subset \{x \in \Z^2 \,|\, x_1,x_2 \leq \sqrt{n}/2 \}. \]
Fey and Redig~\cite{FR} proved analogous bounds in higher dimensions, and extended these bounds to arbitrary values of the height parameter $h$.  This parameter is discussed in section~\ref{sandpilepointsource}.

The methods used to prove the near-perfect circularity of the divisible sandpile shape in Theorem~\ref{divsandcircintro} can be used to give constraints on the shape of the classical abelian sandpile, improving on the bounds of~\cite{FR} and~\cite{LBR}.

\begin{theorem}
\label{sandpilecircintro}
Let $S_n$ be the set of sites that are visited by the classical abelian sandpile model in $\Z^d$, starting from $n$ particles at the origin.  Write $n = \omega_d r^d$.  Then for any $\epsilon>0$ we have
	\[ B_{c_1r - c_2} \subset S_n \subset B_{c'_1r + c'_2} \]
where
	\[ c_1 = (2d-1)^{-1/d}, \qquad c'_1 = (d-\epsilon)^{-1/d}. \]
The constant $c_2$ depends only on $d$, while $c'_2$ depends only on $d$ and $\epsilon$.
\end{theorem}

Note that Theorem~\ref{sandpilecircintro} does not settle the question of the asymptotic shape of $S_n$, and Figure~\ref{sandpile1m} indicates that the limiting shape in two dimensions may be a polygon rather than a disk.  Even the existence of an asymptotic shape is not known, however.

\section{Multiple Point Sources}

\begin{figure}
\centering
\includegraphics[scale=.198]{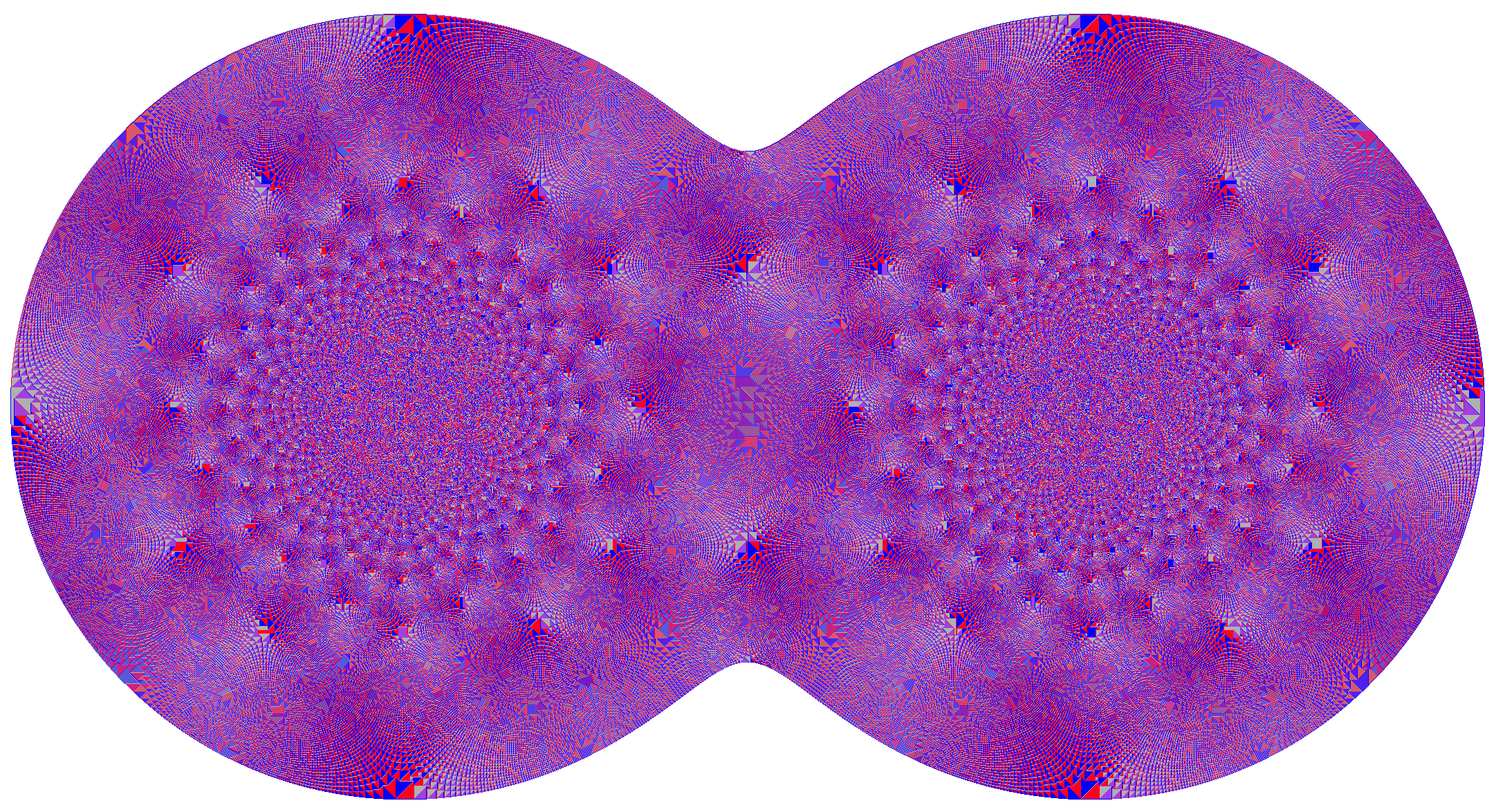}
\caption{The rotor-router model in $\Z^2$ started from two point sources on the $x$-axis.  The boundary of the limiting shape is an algebraic curve of degree~$4$; see equation (\ref{twosourcequartic}).}
\label{rotortwosource}
\end{figure}

Using our results for single point sources together with the construction of the smash sum (\ref{smashsumdef}), we can understand the limiting shape of our aggregation models started from multiple point sources.  The answer turns out to be a smash sum of balls centered at the sources.   For $x \in \R^d$ write $x^\Points$ for the closest lattice point in $\delta_n \Z^d$, breaking ties to the right.  Our shape theorem for multiple point sources, which is deduced from Theorems~\ref{DFsum}, ~\ref{rotorcircintro} and~\ref{divsandcircintro} along with the main result of \cite{LBG}, is the following.

\begin{theorem}
\label{multiplepointsources}  
Fix $x_1, \ldots, x_k \in \R^d$ and $\lambda_1, \ldots, \lambda_k > 0$.  Let $B_i$ be the ball of volume $\lambda_i$ centered at $x_i$.  Fix a sequence $\delta_n \downarrow 0$, and for $x \in \delta_n \Z^d$ let
	\[ \sigma_n(x) = \left\lfloor \delta_n^{-d} \sum_{i=1}^k \lambda_i 1_{\{x = x_i^\Points\}} \right\rfloor. \]
Let $D_n,R_n,I_n$ be the domains of occupied sites in $\delta_n \Z^d$ formed from the divisible sandpile, rotor-router model, and internal DLA, respectively, started from source density $\sigma_n$.  Then as $n \to \infty$
	\begin{equation} \label{smashsumofballs} D_n, R_n \to B_1 \oplus \ldots \oplus B_k; \end{equation}
and if $\delta_n \leq 1 / n$, then with probability one
	\[ I_n \to B_1 \oplus \ldots \oplus B_k \]
where $\oplus$ denotes the smash sum (\ref{smashsumdef}), and the convergence is in the sense of~(\ref{convergenceinthehausdorffmetric}).
\end{theorem}

Implicit in equation (\ref{smashsumofballs}) is the associativity of the smash sum operation, which is not readily apparent from the definition (\ref{smashsumdef}).  For a proof of associativity, see Lemma~\ref{associativity}.  For related results in dimension two, see \cite[Prop.~3.10]{Sakai82} and \cite[section\ 2.4]{VE}.

We remark that a similar model of internal DLA with multiple point sources was studied by Gravner and Quastel \cite{GQ}, who also obtained a variational solution.  In their model, instead of starting with a fixed number of particles, each source $x_i$ emits particles according to a Poisson process.  The shape theorems of \cite{GQ} concern convergence in the sense of volume, which is a weaker form of convergence than (\ref{convergenceinthehausdorffmetric}).

\section{Quadrature Domains}

By analogy with the discrete case, we would expect that volumes add under the smash sum operation; that is,
	\[ \Leb(A \oplus B) = \Leb(A) + \Leb(B) \]
where $\Leb$ denotes Lebesgue measure in $\R^d$.  
\index{$\Leb$, Lebesgue measure on $\R^d$}
Although this additivity is not immediately apparent from the definition (\ref{smashsumdef}), it holds for all bounded open $A,B \subset \R^d$ provided their boundaries have measure zero; see Corollary~\ref{volumesadd}.

We can derive a more general class of identities known as \emph{quadrature identities} involving integrals of harmonic functions over $A \oplus B$.  Let us first consider the discrete case.
If $h$ is a superharmonic function on $\Z^d$, and $\sigma$ is a mass configuration for the divisible sandpile (so each site $x \in \Z^d$ has mass $\sigma(x)$), the sum $\sum_{x \in \Z^d} h(x) \sigma(x)$ can only decrease when we perform a toppling.  Thus
	\begin{equation} \label{discretequadrature} \sum_{x \in \Z^d} h(x) \nu(x) \leq \sum_{x \in \Z^d} h(x) \sigma(x), \end{equation}
where $\nu$ is the final mass configuration.
We therefore expect the domain $D$ given by (\ref{thenoncoincidencesetintro}) to satisfy the \emph{quadrature inequality}
	\begin{equation} \label{quadratureineq} \int_D h(x) dx \leq \int_D h(x) \sigma(x) dx \end{equation}
for all integrable superharmonic functions $h$ on $D$.  For a proof under suitable smoothness assumptions on $\sigma$ and $h$, see Proposition~\ref{boundaryregularitysmooth}; see also \cite{Sakai}. 

A domain $D \subset \R^d$ satisfying an inequality of the form (\ref{quadratureineq}) is called a \emph{quadrature domain} for $\sigma$.  Such domains are widely studied in potential theory and have a variety of applications in fluid dynamics \cite{Crowdy,Richardson}.  For more on quadrature domains and their connection with the obstacle problem, see \cite{AS76,CKS,GS,KS,Sakai,Shahgholian}.  Equation (\ref{discretequadrature}) can be regarded as a discrete analogue of a quadrature inequality; in this sense our aggregation models, in particular the divisible sandpile, produce discrete analogues of quadrature domains.

In Proposition~\ref{ballquadrature}, we show that the smash sum of balls $B_1 \oplus \ldots \oplus B_k$ arising in Theorem~\ref{multiplepointsources} obeys the classical quadrature identity
	\begin{equation} \label{classicalquadrature} \int_{B_1 \oplus \ldots \oplus B_k} h(x) dx = \sum_{i=1}^k \lambda_i h(x_i) \end{equation}
for all harmonic functions $h$ on $B_1 \oplus \ldots \oplus B_k$.  This can be regarded as a generalization of the classical mean value property of harmonic functions, which corresponds to the case $k=1$.  Using results of Gustafsson \cite{Gustafsson83} and Sakai \cite{Sakai} on quadrature domains in the plane, we can deduce the following theorem, which is proved in section~\ref{multiplesources}.

\begin{theorem}
\label{algebraicboundary}
Let $B_1, \ldots, B_k$ be disks in $\R^2$ with distinct centers.  The boundary of the smash sum $B_1 \oplus \ldots \oplus B_k$ lies on an algebraic curve of degree $2k$.  More precisely, there is a polynomial $P \in \R[x_1,x_2]$ of the form
	\[ P(x_1, x_2) = \left(x_1^2 + x_2^2\right)^k + \mbox{\em lower order terms} \]
and there is a finite set of points $E \subset \R^2$, possibly empty, such that
	\[ \partial (B_1 \oplus \ldots \oplus B_k) = \{(x_1,x_2) \in \R^2 | P(x_1,x_2)=0\} - E. \]
\end{theorem}

For example, if $B_1$ and $B_2$ are disks of equal radius $r>1$ centered at $(1,0)$ and $(-1,0)$, then $\partial (B_1 \oplus B_2)$ is given by the quartic curve \cite{Shapiro}
	\begin{equation} \label{twosourcequartic} \left(x_1^2+x_2^2\right)^2 - 2r^2 \left(x_1^2 + x_2^2\right) - 2(x_1^2 - x_2^2) = 0. \end{equation}
This curve describes the shape of the rotor-router model with two point sources pictured in Figure~\ref{rotortwosource}.

\section{Aggregation on Trees}

Let $T$ be the infinite $d$-regular tree.  To define rotor-router walk on a tree, for each vertex of $T$ choose a cyclic ordering of its $d$ neighbors.  Each vertex is assigned a rotor which points to one of the neighboring vertices, and a particle walks by first rotating the rotor at each site it comes to, then stepping in the direction of the newly rotated rotor.  Fix a vertex $o \in T$ called the origin.  Beginning with $A_1 = \{o\}$, define the rotor-router aggregation cluster $A_n$ inductively by
	\[ A_n = A_{n-1} \cup \{x_n\}, \qquad n >1 \]
where $x_n \in T$ is the endpoint of a rotor-router walk started at $o$ and stopped on first exiting $A_{n-1}$.  We do not change the positions of the rotors when adding a new chip.  Thus the sequence $(A_n)_{n \geq 1}$ depends only on the choice of the initial rotor configuration.  

Our next result is the analogue of Theorem~\ref{rotorcircintro} on regular trees.
Call a configuration of rotors \emph{acyclic} if there are no directed cycles of rotors.  On a tree, this is equivalent to forbidding directed cycles of length $2$: for any pair of neighboring vertices $v\sim w$, if the rotor at $v$ points to $w$, then the rotor at $w$ does not point to $v$.  As the following result shows, provided we start with an acyclic configuration of rotors, the occupied cluster $A_n$ is a perfect ball for suitable values of $n$.

\begin{theorem}
\label{aggregintro}
Let $T$ be the infinite $d$-regular tree, and let
	\[ B_r = \{x \in T \,:\, |x| \leq r\} \]
be the ball of radius $r$ centered at the origin $o \in T$, where $|x|$ is the length of the shortest path from $o$ to $x$.  Write
	\[ b_r = \# B_r = 1 + d \frac{(d-1)^r-1}{d-2}. \]
Let $A_n$ be the region formed by rotor-router aggregation on the infinite $d$-regular tree, starting from $n$ chips at $o$.  If the initial rotor configuration is acyclic, then 
	 \[ A_{b_r} = B_r. \]
\end{theorem}

The proof of Theorem~\ref{aggregintro} uses the {\it sandpile group} of a finite regular tree with the leaves collapsed to a single vertex.  This is an abelian group defined for any graph $G$ whose order is the number of spanning trees of $G$.  In section~\ref{sandpiletree} we recall the definition of the sandpile group and prove the following decomposition theorem expressing the sandpile group of a finite regular tree as a product of cyclic groups.

\begin{theorem}
\label{sandpilegroupdecomp}
Let $T_n$ be the regular tree of degree $d=a+1$ and height $n$, with leaves
\index{$T_n$, regular tree with wired boundary}
collapsed to a single sink vertex and an edge joining the root to the sink.  Then writing $\Z_p^q$ for the group $(\Z/p\Z) \oplus \ldots \oplus (\Z/p\Z)$ with $q$ summands, the sandpile group of $T_n$ is given by
	\[ SP(T_n) \simeq \Z_{1+a}^{a^{n-3}(a-1)} \oplus \Z_{1+a+a^2}^{a^{n-4}(a-1)} \oplus \ldots \oplus \Z_{1+a+\ldots+a^{n-2}}^{a-1} \oplus \Z_{1+a+\ldots+a^{n-1}}. \]
\end{theorem}

For example, taking $d=3$ we obtain that the sandpile group of the regular ternary tree of height $n$ has the decomposition
	\[ SP(T_n) \simeq (\Z_3)^{2^{n-3}} \oplus (\Z_7)^{2^{n-4}} \oplus \ldots \oplus \Z_{2^{n-1}-1} \oplus \Z_{2^n-1}. \]

Toumpakari \cite{Toumpakari} studied the sandpile group of the ball $B_n$ inside the infinite $d$-regular tree.  Her setup differs slightly from ours in that there is no edge connecting the root to the sink.  She found the rank, exponent, and order of $SP(B_n)$ and conjectured a formula for the ranks of its Sylow $p$-subgroups.  We use Theorem~\ref{sandpilegroupdecomp} to give a proof of her conjecture in section~\ref{toumpakari}.

Chen and Schedler \cite{ChSch} study the sandpile group of thick trees (i.e.\ trees with multiple edges) without collapsing the leaves to the sink.  They obtain quite a different product formula in this setting.

In section~\ref{rotortree} we define the {\it rotor-router group} of a graph and show that it is isomorphic to the sandpile group.  We then use this isomorphism to prove Theorem~\ref{aggregintro}.

Much previous work on the rotor-router model has taken the form of comparing the behavior of rotor-router walk with the expected behavior of random walk.  For example, Cooper and Spencer \cite{CS} show that for any configuration of chips on even lattice sites in $\Z^d$, letting each chip perform rotor-router walk for $n$ steps results in a configuration that differs by only constant error from the expected configuration had the chips performed independent random walks.  We continue in this vein by investigating the recurrence and transience of rotor-router walk on trees.  A walk which never returns to the origin visits each vertex only finitely many times, so the positions of the rotors after a walk has escaped to infinity are well-defined.  We construct two ``extremal'' rotor configurations on the infinite ternary tree, one for which walks exactly alternate returning to the origin with escaping to infinity, and one for which every walk returns to the origin.  The latter behavior is something of a surprise: to our knowledge it represents the first example of rotor-router walk behaving fundamentally differently from the expected behavior of random walk.

In between these two extreme cases, a variety of intermediate behaviors are possible.  We say that a binary word $a_1 \ldots a_n$ is an {\it escape sequence} for the infinite ternary tree if there exists an initial rotor configuration on the tree so that the $k$-th chip escapes to infinity if and only if $a_k=1$.  The following result characterizes all possible escape sequences on the ternary tree.

\begin{theorem}
\label{escapeseqs}
Let $a = a_1 \ldots a_n$ be a binary word. For $j \in \{1,2,3\}$ write $a^{(j)} = a_{j} a_{j+3} a_{j+6} \ldots$. Then $a$ is an escape sequence for some rotor configuration on the infinite ternary tree if and only if for each $j$ and $k \geq 2$, every subword of $a^{(j)}$ of length $2^k-1$ contains at most $2^{k-1}$ ones.
\end{theorem}

Theorem~\ref{escapeseqs} is proved in section~\ref{recurrenceandtransience} by expressing the escape sequence corresponding to a rotor configuration on the full tree in terms of the escape sequences of the configurations on each of the principal subtrees. 

\chapter{Spherical Asymptotics for Point Sources}
\label{pointsource}

This chapter is devoted to the proofs of Theorems~\ref{rotorcircintro}, \ref{divsandcircintro} and~\ref{sandpilecircintro}.  In section~\ref{basicestimate}, we derive the basic Green's function estimates that are used in the proofs.  In section~\ref{divsandpointsource} we prove the abelian property and Theorem~\ref{divsandcircintro} for the divisible sandpile.  In section~\ref{sandpilepointsource} we adapt the methods used for the divisible sandpile to prove Theorem~\ref{sandpilecircintro} for the classical abelian sandpile model.  Section~\ref{rotorpointsource} is devoted to the proof of Theorem~\ref{rotorcircintro} for the rotor-router model.

\section{Basic Estimate}
\label{basicestimate}

Write $(X_k)_{k\geq 0}$ for simple random walk in $\Z^d$, 
\index{$X_k$, simple random walk}
and for $d\geq 3$ denote by 
	\begin{equation} \label{greensfunctiondef} g_1(x,y) = \EE_x \# \{ k|X_k=y \} 	\index{$g_1(x,y)$, Green's function on $\Z^d$} \end{equation}
the expected number of visits to~$y$ by simple random walk started at~$x$.  This is the \emph{discrete harmonic Green's function} in~$\Z^d$.  For fixed~$x$, the function~$g_1(x,\cdot)$ is harmonic except at $x$, where its discrete Laplacian is~$-1$.  Our notation~$g_1$ is chosen to distinguish between the discrete Green's function in~$\Z^d$ and its continuous counterpart~$g$ in~$\R^d$.  For the definition of $g$,  
see section~\ref{superharmonicpotentials}.  Estimates relating the discrete and continuous Green's functions are discussed in section~\ref{discretepotentialtheory}.  

In dimension $d=2$, simple random walk is recurrent, so the expectation on the right side of (\ref{greensfunctiondef}) infinite.  Here we define the {\it potential kernel}
	\begin{equation} \label{potentialkerneldef} g_1(x,y) = \lim_{n \to \infty} \left( g_1^n(x,y)-g_1^n(x,x) \right) \end{equation}
where
	\[ g_1^n(x,y) = \EE_x \# \{k \leq n|X_k=y\}. \]
The limit defining $g_1$ in (\ref{potentialkerneldef}) is finite \cite{Lawler,Spitzer}, and $g_1(x,\cdot)$ is harmonic except at $x$, where its discrete Laplacian is $-1$.  Note that (\ref{potentialkerneldef}) is the negative of the usual definition of the potential kernel; we have chosen this sign convention so that $g_1$ has the same Laplacian in dimension two as in higher dimensions.

Fix a real number $m>0$ and consider the function on $\Z^d$
	\begin{equation} \label{gammadef} \widetilde{\gamma}_d(x) = |x|^2 + mg_1(o,x). \end{equation}
Let $r$ be such that $m = \omega_d r^d$, and let 
	\begin{equation} \gamma_d(x) = \widetilde{\gamma}_d(x)-\widetilde{\gamma}_d(\floor{r}e_1) \label{subtractoffconstant} \index{$\gamma_d$, obstacle for a single point source} \end{equation}
where $e_1$ is the first standard basis vector in $\Z^d$.  The function $\gamma_d$ plays a central role in our analysis.  To see where it comes from, recall the divisible sandpile \emph{odometer function}
	\[ u(x) = \text{total mass emitted from $x$}. \]
Let $D_m \subset \Z^d$ be the domain of fully occupied sites for the divisible sandpile formed from a pile of mass $m$ at the origin.  For $x \in D_m$, since each neighbor $y$ of $x$ emits an equal amount of mass to each of its $2d$ neighbors, we have
	\begin{align*} \Delta u(x) &= \frac{1}{2d} \sum_{y \sim x} u(y) - u(x) \\
					        &= \text{mass received by $x$} - \text{mass emitted by $x$} \\
					        &= 1 - m\delta_{ox}. \end{align*}
Moreover, $u=0$ on $\partial D_m$, where for $A \subset \Z^d$ we write
	\[ \partial A = \{x \in A^c \,|\, x \sim y \mbox{ for some } y \in A\} \]
for the boundary of $A$.  
\index{$\partial A$, boundary of $A \subset \Z^d$}
By construction, the function $\gamma_d$ obeys the same Laplacian condition: $\Delta \gamma_d = 1 - m\delta_o$; and as we will see shortly, $\gamma_d \approx 0$ on $\partial B_r$.  Since we expect the domain $D_m$ to be close to the ball $B_r$, we should expect that $u \approx \gamma_d$.  In fact, we will first show that $u$ is close to $\gamma_d$, and then use this to conclude that $D_m$ is close to $B_r$.

We will use the following estimates for the Green's function \cite{FU,Uchiyama}; see also \cite[Theorems 1.5.4 and 1.6.2]{Lawler}.

 \begin{equation}  \label{standardgreensestimate}
  g_1(x,y) = \begin{cases}  -\frac{2}{\pi} \log |x-y| + \kappa + O(|x-y|^{-2}), & d=2 \\
			 a_d |x-y|^{2-d} + O(|x-y|^{-d}), & d\geq 3. \end{cases} \end{equation}
Here $a_d = \frac{2}{(d-2)\omega_d}$, where $\omega_d$ is the volume of the unit ball in $\R^d$, and $\kappa$ is a constant whose value we will not need to know.  
Here and throughout this thesis, constants in error terms denoted $O(\cdot)$ depend only on $d$.

We will need an estimate for the function $\gamma_d$ near the boundary of the ball $B_r$.  We first consider dimension $d=2$.  From (\ref{standardgreensestimate}) we have
	\begin{equation} \label{gammadefn} \widetilde{\gamma}_2(x) = \phi(x) - \kappa m + O(m|x|^{-2}), \end{equation}
where
	\[ \phi(x) = |x|^2 - \frac{2m}{\pi} \log |x|. \]
In the Taylor expansion of $\phi$ about $|x|=r$ 
	\begin{equation} \label{secondordertaylor} \phi(x) = \phi(r) - \phi'(r)(r-|x|) + \frac12 \phi''(t)(r-|x|)^2 \end{equation}
the linear term vanishes, leaving
	\begin{equation} \label{taylordim2} 
	\gamma_2(x) =  \left(1+\frac{m}{\pi t^2} \right)(r-|x|)^2 + O(m|x|^{-2})
	\end{equation}
for some $t$ between $|x|$ and $r$.

In dimensions $d \geq 3$, from (\ref{standardgreensestimate}) we have
	\[ \widetilde{\gamma}_d(x) = |x|^2 + a_d m |x|^{2-d} + O(m|x|^{-d}). \]
Setting $\phi(x) = |x|^2 + a_d m |x|^{2-d}$, the linear term in the Taylor expansion (\ref{secondordertaylor}) of $\phi$ about $|x|=r$ again vanishes, yielding
	\[ \gamma_d(x) = \left( 1 + (d-1)(r/t)^d \right) (r-|x|)^2 + O(m|x|^{-d}) \]
for $t$ between $|x|$ and $r$.  Together with (\ref{taylordim2}), this yields the following estimates in all dimensions $d\geq 2$.

\begin{lemma} \label{gammalowerbound}
Let $\gamma_d$ be given by (\ref{subtractoffconstant}).
For all $x \in \Z^d$ we have
	\begin{equation} \label{quadraticgrowtheq} \gamma_d(x) \geq (r-|x|)^2 + O\left(\frac{r^d}{|x|^d}\right). \end{equation}
\end{lemma}

\begin{lemma} \label{gammaupperbound}
Let $\gamma_d$ be given by (\ref{subtractoffconstant}).  Then uniformly in $r$,
	\[ \gamma_d(x) = O(1), \qquad x \in B_{r+1}-B_{r-1}. \]	
\end{lemma}

The following lemma is useful for $x$ near the origin, where the error term in (\ref{quadraticgrowtheq}) blows up.

\begin{lemma} \label{gammanearorigin}
Let $\gamma_d$ be given by (\ref{subtractoffconstant}).  Then for sufficiently large $r$, we have
	\[ \gamma_d(x) \geq \frac{r^2}{4}, \qquad \forall x \in B_{r/3}. \]
\end{lemma}

\begin{proof}
Since $\gamma_d(x)-|x|^2$ is superharmonic, it attains its minimum in $B_{r/3}$ at a point $z$ on the boundary.  Thus for any $x \in B_{r/3}$
	\[ \gamma_d(x)-|x|^2 \geq \gamma_d(z) - |z|^2, \]
hence by Lemma~\ref{gammalowerbound}
	\[ \gamma_d(x) \geq (2r/3)^2 - (r/3)^2 + O(1) > \frac{r^2}{4}. \qed \]
\renewcommand{\qedsymbol}{}
\end{proof}

Lemmas~\ref{gammalowerbound} and~\ref{gammanearorigin} together imply the following.

\begin{lemma} \label{gammanonnegativeeverywhere}
Let $\gamma_d$ be given by (\ref{subtractoffconstant}).  There is a constant $a$ depending only on $d$, such that $\gamma_d \geq -a$ everywhere.
\end{lemma}

\section{Divisible Sandpile}
\label{divsandpointsource}

\subsection{Abelian Property}

In this section we prove the abelian property of the divisible sandpile mentioned in the introduction.  We work in continuous time.  Fix $\tau>0$, and let $T : [0,\tau] \rightarrow \Z^d$ be a function having only finitely many discontinuities in the interval $[0,t]$ for every $t<\tau$.  The value $T(t)$ represents the site being toppled at time $t$.  The odometer function at time $t$ is given by
	\[ u_t(x) = \Leb \left(T^{-1}(x) \cap [0,t]\right), \]
where $\Leb$ denotes Lebesgue measure.  We will say that $T$ is a {\it legal toppling function} for an initial configuration $\nu_0$ if for every $0 \leq t \leq \tau$
	\[ \nu_t(T(t)) \geq 1 \]
where
	\begin{equation} \label{laplacianofu_t} \nu_t(x) = \nu_0(x) + \Delta u_t(x) \end{equation}
is the amount of mass present at $x$ at time $t$.  If in addition $\nu_\tau \leq 1$, we say that $T$ is {\it complete}.  The abelian property can now be stated as follows.

\begin{lemma}
\label{abelianproperty}
If $T_1 : [0,\tau_1] \rightarrow \Z^d$ and  $T_2 : [0,\tau_2] \rightarrow \Z^d$ are complete legal toppling functions for an initial configuration $\nu_0$, then for any $x \in \Z^d$
	\[ \Leb \left(T_1^{-1}(x)\right) = \Leb\left(T_2^{-1}(x)\right). \]
In particular, $\tau_1 = \tau_2$ and the final configurations $\nu_{\tau_1}$, $\nu_{\tau_2}$ are identical.
\end{lemma}

\begin{proof}
For $i=1,2$ write
	\begin{align*} 
	&u^i_t(x) = \Leb \left( T_i^{-1}(x) \cap [0,t] \right); \\
	&\nu^i_t(x) = \nu_0(x) + \Delta u^i_t(x).
	\end{align*}
Write $u^i = u^i_{\tau_i}$ and $\nu^i = \nu^i_{\tau_i}$.
Let $t(0)=0$ and let $t(1) < t(2) < \ldots$ be the points of discontinuity for $T_1$.  Let $x_k$ be the value of $T_1$ on the interval $(t(k-1),t(k))$.  We will show by induction on $k$ that 
	\begin{equation} \label{inductivehypothesis} u^2(x_k) \geq u^1_{t(k)}(x_k). \end{equation}
Note that for any $x \neq x_k$, if $u^1_{t(k)}(x)>0$, then letting $j<k$ be maximal such that $x_j=x$, since $T_1 \neq x$ on the interval $(t(j),t(k))$, it follows from the inductive hypothesis (\ref{inductivehypothesis}) that
	\begin{equation} \label{strongerformofindhyp} u^2(x) = u^2(x_j) \geq u^1_{t(j)}(x_j) = u^1_{t(k)}(x). \end{equation}
Since $T_1$ is legal and $T_2$ is complete, we have
	\[ \nu^2(x_k) \leq 1 \leq \nu^1_{t(k)}(x_k) \]
hence
	\[ \Delta u^2(x_k) \leq \Delta u^1_{t(k)}(x_k). \]
Since $T_1$ is constant on the interval $(t(k-1),t(k))$ we obtain
	\[ u^2(x_k) \geq u^1_{t(k)}(x_k) + \frac{1}{2d} \sum_{x \sim x_k} (u^2(x)-u^1_{t(k-1)}(x)). \]
By (\ref{strongerformofindhyp}), each term in the sum on the right side is nonnegative, completing the inductive step.

Since $t(k) \uparrow \tau_1$ as $k \uparrow \infty$, the right side of (\ref{strongerformofindhyp}) converges to $u^1(x)$ as $k \uparrow \infty$, hence $u^2 \geq u^1$.  After interchanging the roles of $T_1$ and $T_2$, the result follows.  
\end{proof}

\subsection{Proof of Theorem~\ref{divsandcircintro}}

Recall that a function $s$ on $\Z^d$ is {\it superharmonic} if $\Delta s \leq 0$.  Given a function $\gamma$ on $\Z^d$ the {\it least superharmonic majorant} of $\gamma$ is the function
	\[ s(x) = \text{inf} \{ f(x)~|~ f \text{ is superharmonic and } f\geq \gamma \}. \]
Note that if $f$ is superharmonic and $f \geq \gamma$ then
	\[ f(x) \geq \frac{1}{2d} \sum_{y \sim x} f(y) \geq \frac{1}{2d} \sum_{y \sim x} s(y). \]
Taking the infimum on the left side we obtain that $s$ is superharmonic.

\begin{lemma}
\label{discretemajorant}
Let $T : [0,\tau] \to \Z^d$ be a complete legal toppling function for the initial configuration $\nu_0$, and let
	\[ u(x) = \Leb(T^{-1}(x)) \]
be the corresponding odometer function for the divisible sandpile.  Then $u = s + \gamma$, where
	\[ \gamma(x) = |x|^2 + \sum_{y \in \Z^d} g_1(x,y) \nu_0(y) \]
and $s$ is the least superharmonic majorant of $-\gamma$.  
\end{lemma}

\begin{proof}
From (\ref{laplacianofu_t}) we have
	\[ \Delta u = \nu_\tau - \nu_0 \leq 1-\nu_0. \]
Since $\Delta \gamma = 1-\nu_0$, the difference $u-\gamma$ is superharmonic.  As $u$ is nonnegative, it follows that $u-\gamma \geq s$.  For the reverse inequality, note that $s+\gamma-u$ is superharmonic on the domain $D = \{x~|~\nu_\tau(x)=1\}$ of fully occupied sites and is nonnegative outside $D$, hence nonnegative inside $D$ as well. 
\end{proof}

We now turn to the case of a point source mass $m$ started at the origin: $\nu_0=m\delta_{o}$.  More general starting densities are treated in Chapter~\ref{scalinglimit}.  In the case of a point source of mass $m$, the natural question is to identify the shape of the resulting domain $D_m$ of fully occupied sites, i.e.\ sites $x$ for which $\mu(x)=1$.  According to Theorem~\ref{divsandcirc}, $D_m$ is extremely close to a ball of volume $m$; in fact, the error in the radius is a constant independent of $m$.
As before, for $r\geq 0$ we write
	\[ B_r = \{ x \in \Z^d ~:~ |x| < r \} \]
for the lattice ball of radius $r$ centered at the origin.
\begin{theorem}
\label{divsandcirc}
For $m \geq 0$ let $D_m \subset \Z^d$ be the domain of fully occupied sites for the divisible sandpile formed from a pile of size $m$ at the origin.  There exist constants $c,c'$ depending only on $d$, such that
	\[ B_{r-c} \subset D_m \subset B_{r+c'}, \]
where $r = (m/\omega_d)^{1/d}$ and $\omega_d$ is the volume of the unit ball in $\R^d$.
\end{theorem} 


The idea of the proof is to use Lemma~\ref{discretemajorant} along with the basic estimates on $\gamma$, Lemmas~\ref{gammalowerbound} and~\ref{gammaupperbound}, to obtain estimates on the odometer function
	\[ u(x) = \text{total mass emitted from $x$}. \]
We will need the following simple observation.
	
\begin{lemma}
\label{boundarypath}
For every point $x\in D_m - \{o\}$ there is a path $x=x_0 \sim x_1 \sim \ldots \sim x_k=o$ in $D_m$ with $u(x_{i+1}) \geq u(x_i)+1$.
\end{lemma}
\begin{proof}
If $x_i \in D_m - \{o\}$, let $x_{i+1}$ be a neighbor of $x_i$ maximizing $u(x_{i+1})$.  Then $x_{i+1} \in D_m$ and
	\begin{align*} u(x_{i+1}) &\geq \frac{1}{2d} \sum_{y \sim x_i} u(y) \\
		&= u(x_i) + \Delta u(x_i) \\
		&= u(x_i)+1,  \end{align*}
where in the last step we have used the fact that $x_i \in D_m$.
\end{proof}
   
\begin{proof}[Proof of Theorem~\ref{divsandcirc}]
We first treat the inner estimate.  Let $\gamma_d$ be given by (\ref{subtractoffconstant}).
By Lemma~\ref{discretemajorant} the function $u-\gamma_d$ is superharmonic,  so its minimum in the ball $B_r$ is attained on the boundary.  Since $u\geq 0$, we have by  Lemma~\ref{gammaupperbound}
	\[ u(x)-\gamma_d(x) \geq O(1), \qquad x \in \partial B_r. \]
Hence by Lemma~\ref{gammalowerbound},
	\begin{equation} \label{odomlowerbound} u(x) \geq  (r-|x|)^2 + O(r^d/|x|^d), \qquad x \in B_r. \end{equation}
It follows that there is a constant $c$, depending only on $d$, such that $u(x)>0$ whenever $r/3\leq |x|<r-c$.  Thus $B_{r-c}-B_{r/3} \subset D_m$.  For $x \in B_{r/3}$, by Lemma~\ref{gammanearorigin} we have $u(x)>r^2/4+O(1)>0$, hence $B_{r/3} \subset D_m$.

For the outer estimate, note that $u-\gamma_d$ is harmonic on $D_m$.  By Lemma~\ref{gammanonnegativeeverywhere} we have $\gamma_d \geq -a$ everywhere, where $a$ depends only on $d$.  Since $u$ vanishes on $\partial D_m$ it follows that $u-\gamma_d \leq a$ on $D_m$.  
Now for any $x \in D_m$ with $r-1<|x|\leq r$, we have by Lemma~\ref{gammaupperbound}
	\[ u(x) \leq \gamma_d(x) + a \leq c' \]
for a constant $c'$ depending only on $d$.  Lemma~\ref{boundarypath} now implies that $D_m \subset B_{r+c'+1}$.
\end{proof}

\section{Classical Sandpile}
\label{sandpilepointsource}

We consider a generalization of the classical abelian sandpile, proposed by Fey and Redig \cite{FR}.  Each site in $\Z^d$ begins with a ``hole'' of depth $H$.  Thus, each site absorbs the first $H$ grains it receives, and thereafter functions normally, toppling once for each additional $2d$ grains it receives.
\index{$H$, hole depth}
If $H$ is negative, we can interpret this as saying that every site starts with $h=-H$ grains of sand already present.  Aggregation is only well-defined in the regime $h\leq 2d-2$, since for $h=2d-1$ the addition of a single grain already causes every site in $\Z^d$ to topple infinitely often.

Let $S_{n,H}$ be the set of sites that are visited if $n$ particles start at the origin in $\Z^d$.
\index{$S_{n,H}$, sandpile with hole depth $H$} 
Fey and Redig \cite[Theorem~4.7]{FR} prove that
	\[ \lim_{H \to \infty} \limsup_{n \to \infty}  \frac{H}{n} \# \left( S_{n,H} \bigtriangleup B_{H^{-1/d}r} \right) = 0, \]
where $n = \omega_d r^d$, and $\bigtriangleup$ denotes symmetric difference.  The following theorem strengthens this result.

\begin{theorem}
\label{sandpilecirc}
Fix an integer $H \geq 2-2d$.  Let $S_n = S_{n,H}$ be the set of sites that are visited by the classical abelian sandpile model in $\Z^d$, starting from $n$ particles at the origin, if every lattice site begins with a hole of depth $H$.  Write $n = \omega_d r^d$.  Then
	\[ B_{c_1r - c_2} \subset S_{n,H} \]
where
	\[ c_1 = (2d-1+H)^{-1/d} \]
and $c_2$ is a constant depending only on $d$.  Moreover if $H \geq 1-d$, then for any $\epsilon>0$ we have
	 \[ S_{n,H} \subset B_{c'_1r + c'_2} \]
where
 	\[ c'_1 = (d-\epsilon+H)^{-1/d} \]
and $c'_2$ is independent of $n$ but may depend on $d$, $H$ and $\epsilon$.
\end{theorem}

Note that the ratio $c_1/c'_1 \uparrow 1$ as $H \uparrow \infty$.  Thus, the classical abelian sandpile run from an initial state in which each lattice site starts with a deep hole yields a shape very close to a ball.  Intuitively, one can think of the classical sandpile with deep holes as approximating the divisible sandpile, whose limiting shape is a ball by Theorem~\ref{divsandcirc}.  Following this intuition, we can adapt the proof of Theorem~\ref{divsandcirc} to prove Theorem~\ref{sandpilecirc}; just one additional averaging trick is needed, which we explain below.

Consider the odometer function for the abelian sandpile
	\[  u(x) = \text{total number of grains emitted from }x. \]
Let $T_n = \{x|u(x)>0\}$ be the set of sites which topple at least once.  Then 
	\[ T_n \subset S_n \subset  T_n \cup \partial T_n. \]
In the final state, each site which has toppled retains between $0$ and $2d-1$ grains, in addition to the $H$ that it absorbed.  Hence
	\begin{equation} \label{dumblaplbounds} H \leq \Delta u(x) + n\delta_{ox} \leq 2d-1+H, \qquad x \in T_n. \ 
	\end{equation}
We can improve the lower bound by averaging over a small box.  For $x \in \Z^d$ let
	\[ Q(x,k) = \{ y\in \Z^d \,:\, ||x-y||_{\infty} \leq k \} \index{$Q(x,k)$, discrete cube} \]
be the box of side length $2k+1$ centered at $x$, and let
	\[ u^{(k)}(x) = (2k+1)^{-d} \sum_{y \in Q(x,k)} u(y). \]
Write
	\[ T_n^{(k)} = \{ x \,|\, Q(x,k) \subset T_n \}. \]
Le Borgne and Rossin \cite{LBR} observe that if $T$ is a set of sites all of which topple, the number of grains remaining in $T$ is at least the number of edges internal to $T$: indeed, for each internal edge, the endpoint that topples last sends the other a grain which never moves again.  Since the box $Q(x,k)$ has $2dk(2k+1)^{d-1}$ internal edges, we have
	\begin{equation} \label{averagedinabox} \Delta u^{(k)}(x) \geq (d+H) \frac{2k}{2k+1}, \qquad x \in T_n^{(k)}. \end{equation}
The following lemma is analogous to Lemma~\ref{boundarypath}.
 
\begin{lemma}
\label{boundarypathology}
For every point $x\in T_n$ adjacent to $\partial T_n$ there is a path $x=x_0 \sim x_1 \sim \ldots \sim x_m = o$ in $T_n$ with $u(x_{i+1}) \geq u(x_i)+1$.
\end{lemma}
 
\begin{proof}
By (\ref{dumblaplbounds}) we have
	\[ \frac{1}{2d} \sum_{y \sim x_i} u(y) \geq u(x_i). \]
Since $u(x_{i-1})<u(x_i)$, some term $u(y)$ in the sum above must exceed $u(x_i)$.  Let $x_{i+1}=y$.
\end{proof} 
 
\begin{proof}[Proof of Theorem~\ref{sandpilecirc}]
Let
	\[ \widetilde{\xi}_d(x) = (2d-1+H)|x|^2 + ng_1(o,x), \]
and let 
	\[ \xi_d(x) = \widetilde{\xi}_d(x) - \widetilde{\xi}_d(\floor{c_1r}e_1). \]
Taking $m = n/(2d-1+H)$ in Lemma~\ref{gammaupperbound}, we have
	\begin{equation} \label{minattainedonboundary} u(x)-\xi_d(x) \geq -\xi_d(x) = O(1), \qquad x\in\partial B_{c_1r}. \end{equation}
By (\ref{dumblaplbounds}), $u-\xi_d$ is superharmonic, so (\ref{minattainedonboundary}) holds in all of $B_{c_1r}$.  Hence by Lemma~\ref{gammalowerbound}
	\begin{equation} \label{odomlowerboundagain} u(x) \geq \frac{(r-|x|)^2}{2d-1+H} + O(r^d/|x|^d), \qquad x \in B_{c_1 r}. \end{equation}
It follows that $u$ is positive on $B_{c_1r - c_2}-B_{c_1 r/3}$ for a suitable constant $c_2$.  For $x \in B_{c_1 r/3}$, by Lemma~\ref{gammanearorigin} we have $u(x) > c_1^2 r^2 / 4 + O(1)>0$.  Thus $B_{c_1 r - c_2} \subset T_n \subset S_n$.  

For the outer estimate, let
	\[ \hat{\psi}_d(x) = (d-\epsilon+H)|x|^2 + ng_1(o,x). \]
Choose $k$ large enough so that $(d+H) \frac{2k}{2k+1} \geq d-\epsilon+H$ and define
	\[ \widetilde{\psi}_d(x) = (2k+1)^{-d} \sum_{y \in Q(x,k)} \hat{\psi}_d(y). \]
Finally, let
	\[ \psi_d(x) = \widetilde{\psi}_d(x) - \widetilde{\psi}_d(\floor{c'_1r}e_1). \]
By (\ref{averagedinabox}), $u^{(k)}-\psi_d$ is subharmonic on $T_n^{(k)}$.  Taking $m=n/(d-\epsilon+H)$ in Lemma~\ref{gammanonnegativeeverywhere}, there is a constant $a$ such that $\psi_d \geq -a$ everywhere.  Since $u^{(k)} \leq (2d+H)^{(d+1)k}$ on $\partial T_n^{(k)}$ it follows that $u^{(k)}-\psi_d \leq a+(2d+H)^{(d+1)k}$ on $T_n^{(k)}$.  Now for any $x \in S_n$ with $c'_1r-1 < |x| \leq c'_1 r$ we have by Lemma~\ref{gammaupperbound}
	\[ u^{(k)}(x) \leq \psi_d(x) + a + (2d+H)^{(d+1)k} \leq \tilde{c}_2 \]
for a constant $\tilde{c}_2$ depending only on $d$, $H$ and $\epsilon$.  Then $u(x) \leq c'_2 := (2k+1)^d \tilde{c}_2$.  Lemma~\ref{boundarypathology} now implies that $T_n \subset B_{c'_1 r + c'_2}$, and hence 
	\[ S_n \subset T_n \cup \partial T_n \subset B_{c'_1 r + c'_2+1}. \qed \]
\renewcommand{\qedsymbol}{}
\end{proof}

We remark that the crude bound of $(2d+H)^{(d+1)k}$ used in the proof of the outer estimate can be improved to a bound of order $k^2H$, and the final factor of $(2k+1)^d$ can be replaced by a constant factor independent of $k$ and $H$, using the fact that a nonnegative function on $\Z^d$ with bounded Laplacian cannot grow faster than quadratically; see Lemma~\ref{atmostquadratic}.

\section{Rotor-Router Model}
\label{rotorpointsource}

Given a function $f$ on $\Z^d$, for a directed edge $(x,y)$ write
	\[ \nabla f (x,y) = f(y)-f(x). \index{$\nabla$, discrete gradient} \]
Given a function $s$ on directed edges in $\Z^d$, write
	\[ \div s (x) = \frac{1}{2d} \sum_{y \sim x} s(x,y). \index{$\div$, discrete divergence} \]
The discrete Laplacian of $f$ is then given by
	\[ \Delta f(x) = \div \nabla f = \frac{1}{2d} \sum_{y \sim x} f(y)-f(x). \]

\subsection{Inner Estimate}

Fixing $n \geq 1$, consider the odometer function for rotor-router aggregation
	\[ u(x) = \text{total number exits from $x$ by the first $n$ particles}. \]
We learned the idea of using the odometer function to study the rotor-router shape from Matt Cook \cite{Cook}.

\begin{lemma}
\label{odomflow}
For a directed edge $(x,y)$ in $\Z^d$, denote by $\kappa(x,y)$ the net number of crossings 
from $x$ to $y$ performed by the first $n$ particles in rotor-router aggregation.  
Then
	\begin{equation} \label{gradodom} \nabla u(x,y) = -2d\kappa(x,y) + R(x,y) 
\end{equation}
for some edge function $R$ which satisfies
	\[ |R(x,y)| \leq 4d-2 \]
for all edges $(x,y)$.
\end{lemma}

\begin{remark}
In the more general setting of rotor stacks of bounded discrepancy, the $4d-2$ will be replaced by a different constant here.
\end{remark}

\begin{proof}
Writing $N(x,y)$ for the number of particles routed from $x$ to $y$, we have
	\[ \frac{u(x)-2d+1}{2d} \leq N(x,y) \leq \frac{u(x)+2d-1}{2d} \]
hence
	\begin{align*}  | \nabla u(x,y) + 2d \kappa(x,y) | &= | u(y) - u(x) + 2dN(x,y) - 2dN(y,x) | \\
		&\leq 4d - 2. \qed \end{align*} 
\renewcommand{\qedsymbol}{} 
\end{proof}

For the remainder of this section $C_0, C_1, \ldots$ denote constants depending only on~$d$.

\begin{lemma}
\label{diameterbound}
Let $\Omega \subset \Z^d-\{o\}$ be a finite set.  Then
	\[ \sum_{y \in \Omega} |y|^{1-d} \leq C_0 \,\text{\em Diam}(\Omega). \]
\end{lemma}

\begin{proof}
Let
	\[ \mathcal{S}_k = \{y \in \Z^d \,:\, k\leq |y| < k+1 \}. \]
Then 
	\[ \sum_{y \in S_k} |y|^{1-d} \leq k^{1-d} \# \mathcal{S}_k \leq C_0. \]
Since $\Omega$ can intersect at most Diam$(\Omega)+1$ distinct sets $\mathcal{S}_k$, the proof is complete.
\end{proof}

\begin{lemma}
\label{letsgetthispartystarted}
Let $G = G_{B_r}$ be the Green's function for simple random walk in $\Z^d$ stopped on exiting $B_r$. 
Let $x \in B_r$ with $|x| > r - \rho$.  Then
	\begin{equation} \label{theparty}
	\sum_{\substack{y \in B_r \\ |x-y| < 3\rho}} \sum_{z \sim  y} |G(x,y)-G(x,z)| \leq C_1\rho.
	\end{equation}
\end{lemma}

\begin{proof}
Let $(X_t)_{t\geq 0}$ denote simple random walk in $\Z^d$,  and let $T$ be the first exit time from $B_r$.  
For fixed $y$, the function 
	\begin{equation} \label{stoppedgreen} A(x) = g_1(x,y) - \EE_x g_1(X_T,y) \end{equation}
has Laplacian $\Delta A(x) = - \delta_{xy}$ in $B_r$ and vanishes on $\partial B_r$, hence $A(x)=G(x,y)$.  
Let $x,y \in B_r$ and $z\sim y$.  From (\ref{standardgreensestimate}) we have
	\[ |g_1(x,y)-g_1(x,z)| \leq\frac{C_2}{|x-y|^{d-1}},  \qquad  y,z \neq x. \]  
Using the triangle inequality together with (\ref{stoppedgreen}), we obtain
	\begin{align*} |G(x,y)-G(x,z)| &\leq |g_1(x,y)-g_1(x,z)| + \EE_x | g_1(X_T,y)-g_1(X_T,z)| \\	
		&\leq  \frac{C_2}{|x-y|^{d-1}} + \sum_{w\in \partial B_r} H_x(w) \frac{C_2}{|w-y|^{d-1}}, \end{align*}
where $H_x(w) = \PP_x(X_T = w)$.

Write $D = \{y \in B_r \,:\, |x-y|<3\rho\}$.  Then
	\begin{equation} \label{twointegrals}  \sum_{\substack{y \in D \\ y \neq x}} \sum_{\substack{z \sim  y \\ z\neq x}} |G(x,y)-G(x,z)| \leq C_3 \rho + C_2 \sum_{w \in \partial B_r} H_x(w) \sum_{y \in D} |w-y|^{1-d}.
	\end{equation}
By Lemma~\ref{diameterbound}, the inner sum on the right is at most $C_0\,$Diam$(D) = 6C_0\rho$, so the right side of (\ref{twointegrals}) is bounded above by $C_1 \rho$ for a suitable $C_1$.

Finally, the terms in which $y$ or $z$ coincides with $x$ make a negligible contribution to the sum in (\ref{theparty}), since for $y\sim x \in \Z^d$
	\[ |G(x,x)-G(x,y)| \leq |g_1(x,x)-g_1(x,y)| + \EE_x |g_1(X_T,x)-g_1(X_T,y)| \leq C_4. \qed \]
\renewcommand{\qedsymbol}{}
\end{proof}

\begin{figure}
\centering
\resizebox{4in}{!}{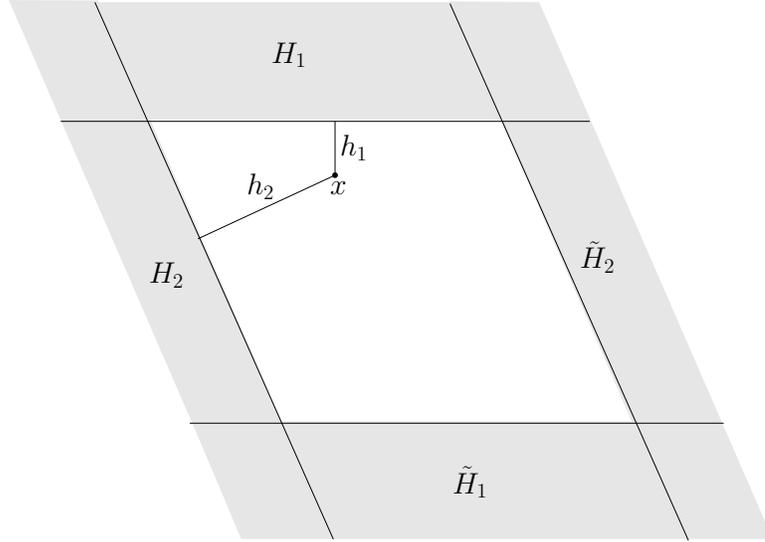}
\caption{Diagram for the Proof of Lemma~\ref{hitthehyperplane}.}
\label{hitthehyperplanefigure}
\end{figure}


\begin{lemma}
\label{hitthehyperplane}
Let $H_1, H_2$ be linear half-spaces in $\Z^d$, not necessarily parallel to the coordinate axes.  Let $T_i$ be the first hitting time of $H_i$.  If $x \notin H_1 \cup H_2$, then
	\[ \PP_x(T_1>T_2) \leq \frac52 \frac{h_1+1}{h_2} \left(1+\frac{1}{2h_2}\right)^2 \]
where $h_i$ is the distance from $x$ to $H_i$.
\end{lemma}
 
\begin{proof}
If one of $H_1, H_2$ contains the other, the result is vacuous.  Otherwise, let $\widetilde{H}_i$ be the half-space shifted parallel to $H_i^c$ by distance $2h_2$ in the direction of $x$, and let $\widetilde{T}_i$ be the first hitting time of $H_i \cup \widetilde{H}_i$.  Let $(X_t)_{t\geq 0}$ denote simple random walk in $\Z^d$, and write $M_t$ for the (signed) distance from $X_t$ to the hyperplane defining the boundary of $H_1$, with $M_0=h_1$.   Then $M_t$ is a martingale with bounded increments.
Since $\EE_x \widetilde{T}_1 < \infty$, we obtain from optional stopping
	\[ h_1 = \EE_x M_{\widetilde{T}_1} \geq 2h_2 \PP_x(X_{\widetilde{T}_1} \in \widetilde{H}_1) - \PP_x(X_{\widetilde{T}_1} \in H_1), \]
hence
	\begin{equation} \label{justgamblersruin} \PP_x(X_{\widetilde{T}_1} \in \widetilde{H}_1) \leq \frac{h_1+1}{2h_2}. \end{equation}
Likewise, $dM_t^2 - t$ is a martingale with bounded increments, giving
	\begin{align} \label{timeupperbound} \EE_x \widetilde{T}_1 &\leq d\EE_x M_{\widetilde{T}_1}^2 \nonumber \\
		&\leq d (2h_2+1)^2 \PP_x(X_{\widetilde{T}_1} \in \widetilde{H}_1) \nonumber \\
		&\leq d (h_1+1) (2h_2+1)\left(1+\frac{1}{2h_2}\right). \end{align}
	
Let $T = \text{min}(\widetilde{T}_1, \widetilde{T}_2)$.  Denoting by $D_t$ the distance from $X_t$ to the hyperplane defining the boundary of $H_2$, the quantity 
	\[ N_t = \frac{d}{2} \big( D_t^2 + (2h_2 - D_t)^2 \big) - t \]
is a martingale.  Writing $p = \PP_x(T=\widetilde{T}_2)$ we have
	\begin{align*} dh_2^2 = \EE N_0 = \EE N_T &\geq p \frac{d}{2}(2h_2)^2  + (1-p)dh_2^2 - \EE_x T \\
			&\geq (1+p)dh_2^2 - \EE_x T
	\end{align*}  
hence by (\ref{timeupperbound})
	\[ p \leq \frac{\EE_x T}{dh_2^2} \leq 2 \frac{h_1+1}{h_2} \left(1+\frac{1}{2h_2}\right)^2. \]
Finally by (\ref{justgamblersruin})
	\[ \PP(T_1>T_2) \leq p + \PP(X_{\widetilde{T}_1} \in \widetilde{H}_1) \leq \frac52 \frac{h_1+1}{h_2} \left(1+\frac{1}{2h_2}\right)^2. \qed \]
\renewcommand{\qedsymbol}{}
\end{proof} 


\begin{figure}
\centering
\resizebox{4in}{!}{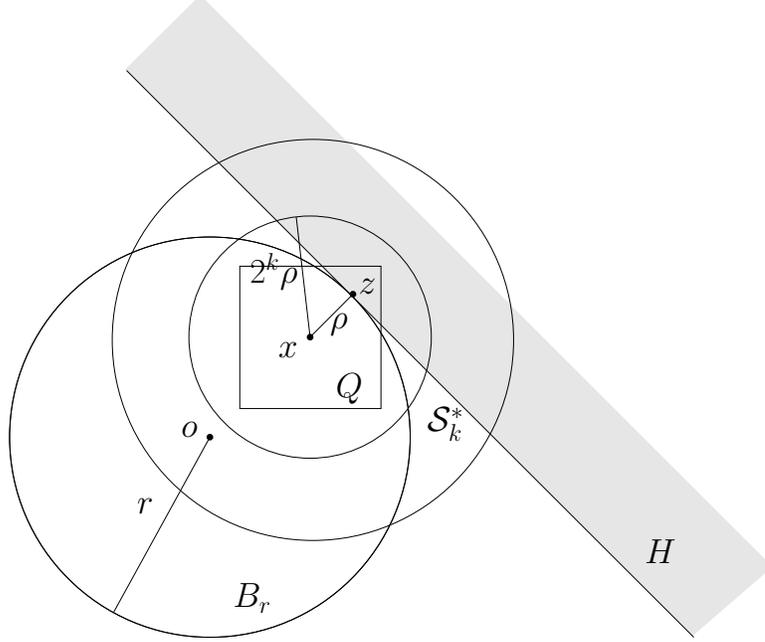}
\caption{Diagram for the proof of Lemma~\ref{geoshells}.}
\label{geoshellsfig}
\end{figure}

\begin{lemma}
\label{geoshells}
Let $x \in B_r$ and let $\rho = r+1-|x|$.  Let
	\begin{equation} \label{shells} \mathcal{S}^*_k = \{ y \in B_r ~:~ 2^k \rho < |x-y| \leq 2^{k+1} \rho \}. \end{equation}
Let $\tau_k$ be the first hitting time of $\mathcal{S}^*_k$, and $T$ the first exit time from $B_r$.  Then 
	\[ \PP_x(\tau_k < T) \leq C_2 2^{-k}. \]
\end{lemma}

\begin{proof}
Let $H$ be the outer half-space tangent to $B_r$ at the point $z\in \partial B_r$ closest to $x$.  Let $Q$ be the cube of side length $2^k \rho/\sqrt{d}$ centered at $x$.  Then $Q$ is disjoint from $\mathcal{S}^*_k$, hence
	\[ \PP_x(\tau_k < T) \leq \PP_x(T_{\partial Q}<T) \leq \PP_x(T_{\partial Q} < T_H) \]
where $T_{\partial Q}$ and $T_H$ are the first hitting times of $\partial Q$ and $H$.  Let $H_1, \ldots, H_{2d}$ be the outer half-spaces defining the faces of $Q$, so that $Q = H_1^c \cap \ldots \cap H_{2d}^c$.  By Lemma~\ref{hitthehyperplane} we have
	\begin{align*}  \PP_x(T_{\partial Q} < T_H) &\leq \sum_{i=1}^{2d} \PP_x(T_{H_i} < T_H) \\
			&\leq \frac52 \sum_{i=1}^{2d} \frac{\text{dist}(x,H)+1}{\text{dist}(x,H_i)} \left(1+\frac{1}{2\,\text{dist}(x,H_i)}\right)^2. \end{align*}
Since dist$(x,H) = |x-z| \leq \rho$ and dist$(x,H_i) = 2^{k-1} \rho/\sqrt{d}$, and $\rho \geq 1$, taking $C_2 = 20\,d^{3/2}(1+\sqrt{d})^2$ completes the proof.
\end{proof}


\begin{lemma}
\label{greengradbound}
Let $G=G_{B_r}$ be the Green's function for random walk stopped on exiting $B_r$.  Let $x \in B_r$ and let $\rho = r-|x|$.   Then
 	\[ \sum_{y\in B_r} \sum_{z\sim y} |G(x,y)-G(x,z)| \leq C_3 \rho \log \frac{r}{\rho}. \]
\end{lemma}

\begin{proof}
Let $\mathcal{S}^*_k$ be given by (\ref{shells}), and let 
	\[ W = \{ w \in \partial( \mathcal{S}^*_k \cup \partial \mathcal{S}^*_k) ~:~ |w-x|<2^k \rho \}. \] 
Let $\tau_W$ be the first hitting time of $W$ and $T$ the first exit time from $B_r$.  For $w\in W$ let 
	\[ H_x(w) = \PP_x(X_{\tau_W \wedge T} = w). \]
Fixing $y \in \mathcal{S}^*_k$ and $z \sim y$, simple random walk started at $x$ must hit $W$ before hitting either $y$ or $z$, hence
	\[ |G(x,y) - G(x,z)| \leq \sum_{w \in W}  H_x(w)|G(w,y)-G(w,z)|. \]
For any $w \in W$ we have $|w|> r - (2^k+1) \rho$ and $|y-w|<3\cdot 2^k \rho$.  Lemma~\ref{letsgetthispartystarted} yields
	\[ \sum_{y \in \mathcal{S}^*_k} \sum_{z \sim y} |G(x,y) - G(x,z)| \leq C_1 (2^k+1) \rho \sum_{w \in W} H_x(w). \]
By Lemma~\ref{geoshells} we have $\sum_{w \in W} H_x(w) \leq C_2 2^{-k}$, so the above sum is at most  $2 C_1 C_2 \rho$.  Since the union of shells $\mathcal{S}^*_0, \mathcal{S}^*_1, \ldots, \mathcal{S}^*_{\ceil{\log_2(r/\rho)}}$ covers all of $B_r$ except for those points within distance $\rho$ of $x$, and $\sum_{|x-y|\leq\rho} \sum_{z \sim y} |G(x,y)-G(x,z)| \leq C_1 \rho$  by Lemma~\ref{letsgetthispartystarted}, the result follows.
\end{proof}

\begin{proof}[Proof of Theorem~\ref{rotorcircintro}, Inner Estimate]
Let $\kappa$ and $R$ be defined as in Lemma \ref{odomflow}.  Since the net number of particles to enter a site $x\neq o$ is at most one, we have $2d~\div \kappa (x) \geq -1$.  Likewise $2d~\div \kappa(o) = n-1$.  Taking the divergence in (\ref{gradodom}), we obtain
	\begin{align} \Delta u(x) &\leq 1 + \div R(x), \qquad x\neq o;  \label{laplodom}  \\ 
	\Delta u(o) &= 1-n + \div R(o). \label{laplodomat0} 
\end{align}
Let $T$ be the first exit time from $B_r$, and define
	\[ f(x) = \EE_x u(X_T) - \EE_x T + n \EE_x \# \{j<T | X_j=0\}. \]
Then $\Delta f(x) = 1$ for $x \in B_r-\{o\}$ and $\Delta f(o)=1-n$.  Moreover $f \geq 0$ on $\partial B_r$.  It follows from Lemma~\ref{gammaupperbound} with $m=n$ that $f \geq \gamma - C_4$ on $B_r$ for a suitable constant $C_4$.
	
We have
	\[ u(x)-\EE_x u(X_T) = \sum_{k \geq 0} \EE_x \Big( u(X_{k\wedge T}) - u(X_{(k+1)\wedge 
T}) \Big). \] 
Each summand on the right side is zero on the event $\{T \leq k\}$, hence
	\[ \EE_x \Big( u(X_{k\wedge T}) - u(X_{(k+1)\wedge T}) ~|~ \mathcal{F}_{k\wedge 
T} \Big) = -\Delta u (X_k) 1_{\{T>k\}}. \]
Taking expectations and using (\ref{laplodom}) and (\ref{laplodomat0}), we obtain
	\begin{align*}
	u(x)-\EE_x u(X_T) 
	&\geq \sum_{k\geq 0} \EE_x \Big[ 1_{\{T>k\}} (n 1_{\{X_k=o\}} -1 - \div R(X_k) ) \Big] \\
	& \hspace{-0.25in} = n\EE_x \# \{k < T | X_k=o \} - \EE_x T -  \sum_{k\geq 0} \EE_x \Big[ 1_{\{T>k\}}  \div R(X_k) \Big],
	\end{align*}
hence
	\begin{equation} \label{sumofRs}
	u(x) - f(x) \geq - \frac{1}{2d} \sum_{k \geq 0} \EE_x \left[ 1_{\{T>k\}} \sum_{z \sim X_k} R(X_k,z) \right]. \end{equation}
Since random walk exits $B_r$ with probability at least $\frac{1}{2d}$ every time it reaches a site adjacent to the boundary $\partial B_r$, the expected time spent adjacent to the boundary before time $T$ is at most $2d$.  Since $|R|\leq 4d$, the terms in (\ref{sumofRs}) with $z \in \partial B_r$ contribute at most $16d^3$ to the sum.  Thus
	\[ u(x) - f(x) \geq - \frac{1}{2d}  \sum_{k\geq 0} \EE_x  \left[ \sum_{\substack{y,z\in B_r \\ y\sim z}} 1_{\{T>k\} \cap \{X_k=y\}}R(y,z) \right] - 8d^2. \]
For $y \in B_r$ we have $\{X_k = y\} \cap \{T>k\} = \{X_{k\wedge T} = y\}$, hence
	\begin{equation} \label{triplesum} u(x) - f(x) \geq - \frac{1}{2d} \sum_{k\geq 0} \sum_{\substack{y,z \in B_r \\ y\sim z}} \PP_x(X_{k\wedge T}=y)R(y,z) - 8d^2. \end{equation}
Write $p_k(y) = \PP_x(X_{k\wedge T} = y)$.  Note that since $\nabla f$ and $\kappa$ are antisymmetric, $R$ is antisymmetric.  Thus
	\begin{align*} \sum_{\substack{y,z\in B_r \\ y\sim z}} p_k(y) R(y,z)
	&=  - \sum_{\substack{y,z\in B_r \\ y\sim z}} p_k(z) R(y,z) \\
	&= \sum_{\substack{y,z\in B_r \\ y\sim z}} \frac{p_k(y)-p_k(z)}{2} R(y,z). \end{align*}
Summing over $k$ and using the fact that $|R|\leq 4d$, we conclude from (\ref{triplesum}) that 
	\[ u(x) \geq f(x) - \sum_{\substack{y,z\in B_r \\ y\sim z}} |G(x,y) - G(x,z)| - 8d^2, \]
where $G=G_{B_r}$ is the Green's function for simple random walk stopped on exiting $B_r$.     By Lemma~\ref{greengradbound} we obtain
	\[ u(x) \geq f(x) - C_3 (r-|x|) \log \frac{r}{r-|x|} - 8d^2. \]
Using the fact that $f \geq \gamma - C_4$, we obtain from Lemma~\ref{gammalowerbound}
	\[ u(x) \geq (r-|x|)^2 - C_3 (r-|x|) \log \frac{r}{r-|x|} + O\left(\frac{r^d}{|x|^d}\right). \]
The right side is positive provided $r/3 \leq |x| < r- C_5 \log r$.  For $x\in B_{r/3}$, by Lemma~\ref{gammanearorigin} we have $u(x)>r^2/4-C_3 r \log \frac32 >0$, hence $B_{r-C_5\log r} \subset A_n$.
\end{proof}

\subsection{Outer Estimate}

The following result is due to Holroyd and Propp (unpublished); we include a proof for the sake of completeness.  Notice that the bound in (\ref{HPbound}) does not depend on the number of particles.

\begin{prop} \label{HP}
Let $\Gamma$ be a finite connected graph, and let $Y\subset Z$ be subsets of the vertex set of $\Gamma$.  Let $s$ be a nonnegative integer-valued function on the vertices of $\Gamma$.  Let $H_w(s,Y)$ be the expected number of particles stopping in $Y$ if $s(x)$ particles start at each vertex $x$ and perform independent simple random walks stopped on first hitting $Z$.  Let $H_r(s,Y)$ be the number of particles stopping in $Y$ if $s(x)$ particles start at each vertex $x$ and perform rotor-router walks stopped on first hitting $Z$.  Let $H(x) = H_w(1_{x},Y)$.  Then
	\begin{equation} \label{HPbound} 
	| H_r(s,Y) - H_w(s,Y) | \leq  \sum_{u \in G} \sum_{v \sim u} | H(u)-H(v) |
	\end{equation}
independent of $s$ and the initial positions of the rotors.
\end{prop} 

\begin{proof}
For each vertex $u$, arbitrarily choose a neighbor $\eta(u)$.  Order the neighbors $\eta(u)=v_1, v_2, \ldots, v_d$ of $u$ so that the rotor at $u$ points to $v_{i+1}$ immediately after pointing to $v_i$ (indices mod $d$).  We assign {\it weight} $w(u,\eta(u))=0$ to a rotor pointing from $u$ to $\eta(u)$, and weight $w(u,v_i) = H(u) - H(v_i) + w(u,v_{i-1})$ to a rotor pointing from $u$ to $v_i$.  These assignments are consistent since $H$ is a harmonic function: $\sum_i (H(u) - H(v_i)) = 0$.  We also assign weight $H(u)$ to a particle located at $u$.  The sum of rotor and particle weights in any configuration is invariant under the operation of routing a particle and rotating the corresponding rotor.  Initially, the sum of all particle weights is $H_w(s,Y)$.  After all particles have stopped, the sum of the particle weights is $H_r(s,Y)$.  Their difference is thus at most the change in rotor weights, which is bounded above by the sum in (\ref{HPbound}).
\end{proof}


%
  
For $\rho \in \Z$ let
	\begin{equation} \label{shelldefn} \mathcal{S}_\rho = \{x\in \Z^d ~:~ \rho \leq |x| < \rho+1 \}. \end{equation}
Then
	\[ B_\rho = \{x\in \Z^d ~:~ |x| <\rho \} = \mathcal{S}_0 \cup \ldots \cup \mathcal{S}_{\rho-1}. \]
Note that for simple random walk started in $B_\rho$, the first exit time of $B_\rho$ and first hitting time of $\mathcal{S}_\rho$ coincide.  Our next result is a modification of Lemma~5(b) of \cite{LBG}.

\begin{figure}
\centering
\resizebox{4in}{!}{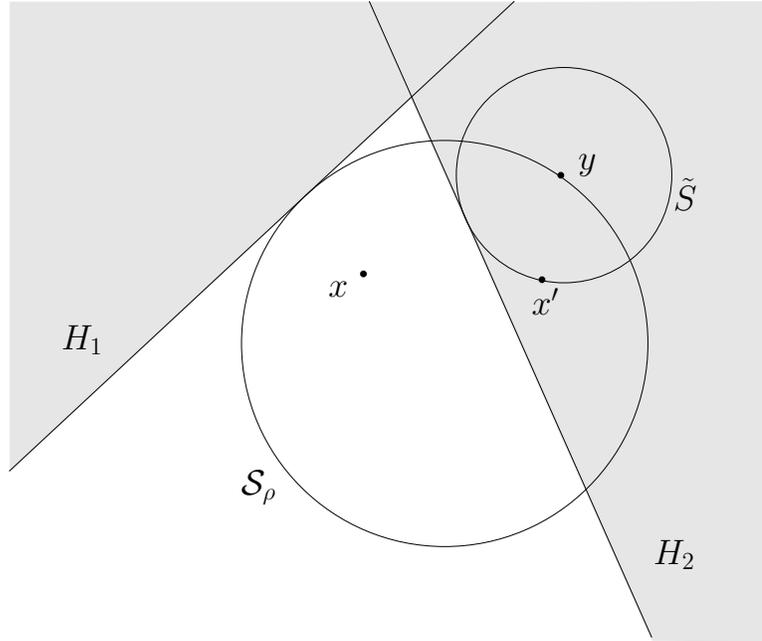}
\caption{Diagram for the proof of Lemma~\ref{inthezone}.}
\label{inthezonefigure}
\end{figure}

\begin{lemma} \label{inthezone}
Fix $\rho \geq 1$ and $y\in \mathcal{S}_{\rho}$.  For $x \in B_\rho$ let $H(x) = \PP_x(X_T=y)$, where $T$ is the first hitting time of $\mathcal{S}_\rho$.  Then
	\begin{equation} \label{spreadingout} H(x) \leq \frac{J}{|x-y|^{d-1}} \end{equation}
for a constant J depending only on $d$.
\end{lemma}

\begin{proof}
We induct on the distance $|x-y|$, assuming the result holds for all $x'$ with $|x'-y| \leq \frac12 |x-y|$; the base case can be made trivial by choosing $J$ sufficiently large.  
By Lemma~5(b) of \cite{LBG}, we can choose $J$ large enough so that the result holds provided $|y|-|x| \geq 2^{-d-3} |x-y|$.  Otherwise, let $H_1$ be the outer half-space tangent to $\mathcal{S}_{\rho}$ at the point of $\mathcal{S}_{\rho}$ closest to $x$, and let $H_2$ be the inner half-space tangent to the ball $\widetilde{S}$ of radius $\frac12 |x-y|$ about $y$, at the point of $\widetilde{S}$ closest to $x$.  By Lemma~\ref{hitthehyperplane} applied to these half-spaces, the probability that random walk started at $x$ reaches $\widetilde{S}$ before hitting $\mathcal{S}_\rho$ is at most $2^{1-d}$.   Writing $\widetilde{T}$ for the first hitting time of $\widetilde{S} \cup \mathcal{S}_{\rho}$, we have
	\[ H(x) \leq \sum_{x' \in \widetilde{S}} \PP_x(X_{\widetilde{T}}=x') H(x') \leq 2^{1-d} J \cdot \left(\frac{|x-y|}{2}\right)^{1-d} \]
where we have used the inductive hypothesis to bound $H(x')$.
\end{proof}
 
The \emph{lazy} random walk in $\Z^d$ stays in place with probability $\frac12$, and moves to each of the $2d$ neighbors with probability $\frac{1}{4d}$.
We will need the following standard result, which can be derived e.g.\ from the estimates in \cite{Lindvall}, section II.12; we include a proof for the sake of completeness. 
 
\begin{lemma}
\label{lazycoupling}
Given $u \sim v \in \Z^d$, lazy random walks started at $u$ and $v$ can be coupled with probability $1-C/R$ before either reaches distance $R$ from $u$, where $C$ depends only on $d$.
\end{lemma}

\begin{proof}
Let $i$ be the coordinate such that $u_i \neq v_i$.  To define a step of the coupling, choose one of the $d$ coordinates uniformly at random.  If the chosen coordinate is different from $i$, let the two walks take the same lazy step so that they still agree in this coordinate.  If the chosen coordinate is $i$, let one walk take a step while the other stays in place.  With probability $\frac12$ the walks will then be coupled.  Otherwise, they are located at points $u',v'$ with $|u'-v'|=2$.  Moreover, $\PP\big(|u-u'|\geq \frac{R}{2\sqrt{d}}\big)<\frac{C'}{R}$ for a constant $C'$ depending only on $d$.  From now on, whenever coordinate $i$ is chosen, let the two walks take lazy steps in opposite directions. 

Let
	\[ H_1 = \left\{x\,\Big|\,x_i = \frac{u'_i+v'_i}{2}\right\} \]
be the hyperplane bisecting the segment $[u',v']$.  Since the steps of one walk are reflections in $H_1$ of the steps of the other, the walks couple when they hit $H_1$.  Let $Q$ be the cube of side length $R/\sqrt{d}+2$ centered at $u$, and let $H_2$ be a hyperplane defining one of the faces of $Q$.  By Lemma~\ref{hitthehyperplane} with $h_1=1$ and $h_2=R/4\sqrt{d}$, the probability that one of the walks exits $Q$ before the walks couple is at most $2d \cdot \frac52 \frac{h_1+1}{h_2} \left(1+\frac{1}{2h_2}\right)^2 \leq 40\, d^{3/2} \big(1+2\sqrt{d}\big)^2 /R$.
\end{proof}

\begin{lemma}
\label{harmonicmeasuregrad}
With $H$ defined as in Lemma~{\em \ref{inthezone}}, we have
	\[ \sum_{u \in B_\rho} \sum_{v \sim u} |H(u)-H(v)| \leq J' \log \rho \]
for a constant $J'$ depending only on $d$.
\end{lemma}

\begin{proof}
Given $u \in B_\rho$ and $v \sim u$, by Lemma~\ref{lazycoupling}, lazy random walks started at $u$ and $v$ can be coupled with probability $1-2C/|u-y|$ before either reaches distance $|u-y|/2$ from $u$.  If the walks reach this distance without coupling, by Lemma~\ref{inthezone} each has still has probability at most $J/|u-y|^{d-1}$ of exiting $B_\rho$ at $y$.  By the strong Markov property it follows that
	\[ |H(u) - H(v)| \leq \frac{2CJ}{|u-y|^d}. \]
Summing in spherical shells about $y$, we obtain
	\[ \sum_{u \in B_\rho} \sum_{v \sim u} |H(u)-H(v)| \leq \sum_{t=1}^{\rho} d\omega_d t^{d-1} \frac{2CJ}{t^d} \leq J' \log \rho. \qed \]
\renewcommand{\qedsymbol}{}
\end{proof}

We remark that Lemma~\ref{harmonicmeasuregrad} could also be inferred from Lemma~\ref{inthezone} using \cite[Thm.\ 1.7.1]{Lawler} in a ball of radius $|u-y|/2$ about $u$.

\begin{proof}[Proof of Theorem~\ref{rotorcircintro}, Outer Estimate]
Fix integers $\rho \geq r$, $h\geq 1$.
In the setting of Proposition~\ref{HP}, let $G$ be the lattice ball $B_{\rho+h+1}$, and let $Z = \mathcal{S}_{\rho+h}$.  Fix $y \in \mathcal{S}_{\rho+h}$ and let $Y=\{y\}$.  For $x \in \mathcal{S}_\rho$ let $s(x)$ be the number of particles stopped at $x$ if all particles in rotor-router aggregation are stopped upon reaching $\mathcal{S}_\rho$.  Write 
	 \[ H(x) = \PP_x(X_T=y) \]
where $T$ is the first hitting time of $\mathcal{S}_{\rho+h}$.   By Lemma~\ref{inthezone} we have
	\begin{equation} \label{H_wbound} H_w(s,y) = \sum_{x \in \mathcal{S}_\rho} s(x)H(x) \leq \frac{J N_\rho}{h^{d-1}} \end{equation}
where
	 \[ N_\rho = \sum_{x \in \mathcal{S}_\rho} s(x) \]
is the number of particles that ever visit the shell $\mathcal{S}_\rho$ in the course of rotor-router aggregation.

By Lemma~\ref{harmonicmeasuregrad} the sum in (\ref{HPbound}) is at most $J' \log h$, hence from Propositon~\ref{HP} and (\ref{H_wbound}) we have
	\begin{equation} \label{tinyRRs} H_r(s,y) \leq \frac{J N_\rho}{h^{d-1}} + J' \log h. \end{equation}
Let $\rho(0) = r$, and define $\rho(i)$ inductively by
	\begin{equation} \label{doublemin} \rho(i+1) = \text{min}\left\{\rho(i)+N_{\rho(i)}^{2/(2d-1)},~\text{min} \{\rho> \rho(i) | N_\rho \leq N_{\rho(i)}/2 \}\right\}. \end{equation}
Fixing $h < \rho(i+1)-\rho(i)$, we have
	\[ h^{d-1} \log h \leq N_{\rho(i)}^{\frac{2d-2}{2d-1}} \log N_{\rho(i)} \leq N_{\rho(i)}; \]
so (\ref{tinyRRs}) with $\rho=\rho(i)$ simplifies to
	\begin{equation} \label{eventinierRRs} H_r(s,y) \leq \frac{CN_{\rho(i)}}{h^{d-1}} \end{equation}
where $C = J + J'$.

Since all particles that visit $\mathcal{S}_{\rho(i)+h}$ during rotor-router aggregation must pass through $\mathcal{S}_{\rho(i)}$, we have
	\begin{equation} \label{alltermsaresmall} N_{\rho(i)+h} \leq \sum_{y\in \mathcal{S}_{\rho(i)+h}} H_r(s,y). \end{equation}
Let $M_k = \# (A_n \cap \mathcal{S}_k)$.  There are at most $M_{\rho(i)+h}$ nonzero terms in the sum on the right side of (\ref{alltermsaresmall}), and each term is bounded above by (\ref{eventinierRRs}), hence
	\[ M_{\rho(i)+h} \geq N_{\rho(i)+h} \frac{h^{d-1}}{CN_{\rho(i)}} \geq  \frac{h^{d-1}}{2C} \]
where the second inequality follows from $N_{\rho(i)+h} \geq N_{\rho(i)}/2$.  Summing over $h$, we obtain
	\begin{equation} \label{lotsavolume}  \sum_{\rho=\rho(i)+1}^{\rho(i+1)-1} M_\rho \geq \frac{1}{2dC} (\rho(i+1)-\rho(i)-1)^d. \end{equation}
The left side is at most $N_{\rho(i)}$, hence
	\[ \rho(i+1)-\rho(i) \leq (2dC N_{\rho(i)})^{1/d} \leq N_{\rho(i)}^{2/(2d-1)} \]
provided $N_{\rho(i)} \geq C' := (2dC)^{2d-1}$.
Thus the minimum in (\ref{doublemin}) is not attained by its first argument.  It follows that $N_{\rho(i+1)}\leq N_{\rho(i)}/2$, hence $N_{\rho(a \log r)} < C' $ for a sufficiently large constant $a$.

By the inner estimate, since the ball $B_{r-c\log r}$ is entirely occupied, we have
	\begin{align*} \sum_{\rho \geq r} M_\rho &\leq \omega_d r^d - \omega_d (r-c\log r)^d \\
	&\leq cd\omega_d r^{d-1} \log r. \end{align*}
Write $x_i = \rho(i+1)-\rho(i)-1$; by (\ref{lotsavolume}) we have
	\[ \sum_{i=0}^{a \log r} x_i^d \leq cd\omega_d r^{d-1} \log r, \]
By Jensen's inequality, subject to  this constraint, $\sum x_i$ is maximized when all $x_i$ are equal, in which case $x_i \leq C'' r^{1-1/d}$ and 
	\begin{equation} \rho(a \log r) = r+\sum x_i \leq r+C''r^{1-1/d} \log r. \label{outershellbound} \end{equation}
Since $N_{\rho(a \log r)} < C'$ we have $N_{\rho(a \log r)+C'} = 0$; that is, no particles reach the shell $\mathcal{S}_{\rho(a \log r)+C'}$.  Taking $c' = C' + C''$, we obtain from (\ref{outershellbound})
	\[ A_n \subset B_{r (1 + c'r^{-1/d}\log r)}. \qed \]
\renewcommand{\qedsymbol}{}
\end{proof}

\chapter{Scaling Limits for General Sources}
\label{scalinglimit}

This chapter is devoted to proving Theorems~\ref{intromain}, \ref{DFsum} and~\ref{multiplepointsources}.  The proofs use many ideas from potential theory, and the relevant background is developed in section~\ref{potentialtheorybackground}.  The proof of Theorem~\ref{intromain} is broken into three sections, one for each of the three aggregation models.  Of the three models the divisible sandpile is the most straightforward and is treated in section~\ref{divsandscalinglimit}.  The rotor-router model and internal DLA are treated in sections~\ref{rotorscalinglimit} and~\ref{IDLAscalinglimit}, respectively.  Finally, in section~\ref{multiplesources} we deduce Theorem~\ref{multiplepointsources} for multiple point sources from Theorem~\ref{DFsum} along with our results for single point sources.

\section{Potential Theory Background}
\label{potentialtheorybackground}

In this section we review the basic properties of superharmonic potentials and of the least superharmonic majorant.  
For more background on potential theory in general, we refer the reader to \cite{HFT,Doob}; for the obstacle problem in particular, see \cite{Caffarelli,Friedman}.

\subsection{Least Superharmonic Majorant}

Since we will often be working with functions on $\R^d$ which may not be twice differentiable, it is desirable to define superharmonicity without making reference to the Laplacian.  Instead we use the mean value property.  A function $u$ on an open set $\Omega \subset \R^d$ is {\it superharmonic} if it is lower-semicontinuous and for any ball $B(x,r) \subset \Omega$
	\begin{equation} \label{meanvalueproperty} u(x) \geq A_r u (x) := \frac{1}{\omega_d r^d} \int_{B(x,r)} u(y) dy. \index{$A_r u$, average in a ball} \end{equation}
Here $\omega_d$ is the volume of the unit ball in $\R^d$.  We say that $u$ is subharmonic if $-u$ is superharmonic, and harmonic if it is both super- and subharmonic.

The following properties of superharmonic functions are well known; for proofs, see e.g.\ \cite{HFT}, \cite{Doob} or \cite{LL}.

\begin{lemma}
\label{basicproperties}
Let $u$ be a superharmonic function on an open set $\Omega \subset \R^d$.  Then
\begin{enumerate}
\item[(i)] $u$ attains its minimum in $\bar{\Omega}$ on the boundary.
\item[(ii)] If $h$ is harmonic on $\Omega$ and $h=u$ on $\partial \Omega$, then $u \geq h$.
\item[(iii)] If $B(x,r_0) \subset B(x,r_1) \subset \Omega$, then
	\[ A_{r_0} u (x) \geq A_{r_1} u (x). \]
\item[(iv)] If $u$ is twice differentiable on $\Omega$, then $\Delta u \leq 0$ on $\Omega$.
\item[(v)] If $B\subset \Omega$ is an open ball, and $v$ is a function on $\Omega$ which is harmonic on $B$, continuous on $\bar{B}$, and agrees with $u$ on $B^c$, then $v$ is superharmonic.
\end{enumerate}
\end{lemma}

Given a function $\gamma$ on $\R^d$ which is bounded above, the {\it least superharmonic majorant} of $\gamma$ (also called the solution to the obstacle problem with obstacle $\gamma$) is the function
	\begin{equation} \label{majorantdef} s(x) = \inf \{f(x) | f \text{ is continuous, superharmonic and }  f \geq \gamma \}. \end{equation}
Note that since $\gamma$ is bounded above, the infimum is taken over a nonempty set.

\begin{lemma}
\label{majorantbasicprops}
Let $\gamma$ be a uniformly continuous function which is bounded above, and let $s$ be given by (\ref{majorantdef}).  Then
\begin{enumerate}
\item[(i)] $s$ is superharmonic.
\item[(ii)] $s$ is continuous.
\item[(iii)] $s$ is harmonic on the domain
	\[ D = \{x \in \R^d | s(x)>\gamma(x) \}. \]
\end{enumerate}
\end{lemma}

\begin{proof}
(i) Let $f \geq \gamma$ be continuous and superharmonic.  Then $f \geq s$.  By the mean value property (\ref{meanvalueproperty}), we have
	\[ f \geq A_r f \geq A_r s. \]
Taking the infimum over $f$ on the left side, we conclude that $s \geq A_r s$.  

It remains to show that $s$ is lower-semicontinuous.  Let
	\[ \omega(\gamma,r) = \sup_{x,y \in \R^d, |x-y|\leq r} |\gamma(x)-\gamma(y)|. \]
Since
	\[ A_r s \geq A_r \gamma \geq \gamma - \omega(\gamma,r) \]
the function $A_r s + \omega(\gamma,r)$ is continuous, superharmonic, and lies above $\gamma$, so
	\[ A_r s \leq s \leq A_r s + \omega(\gamma,r). \]
Since $\gamma$ is uniformly continuous, we have $\omega(\gamma,r) \downarrow 0$ as $r \downarrow 0$, hence $A_r s \rightarrow s$ as $r \downarrow 0$.  Moreover if $r_0<r_1$, then by Lemma~\ref{basicproperties}(iii)
	\[ A_{r_0} s = \lim_{r \rightarrow 0} A_{r_0} A_r s \geq \lim_{r \rightarrow 0} A_{r_1} A_r s = A_{r_1} s. \]
Thus $s$ is an increasing limit of continuous functions and hence lower-semicontinuous.

(ii) Since $s$ is defined as an infimum of continuous functions, it is also upper-semicontinuous.

(iii) Given $x \in D$, write $\epsilon = s(x)-\gamma(x)$.  Choose $\delta$ small enough so that for all $y \in B = B(x,\delta)$
	\[ |\gamma(x)-\gamma(y)| < \frac{\epsilon}{2} \quad\text{and}\quad |s(x)-s(y)| < \frac{\epsilon}{2}. \]
Let $f$ be the continuous function which is harmonic in $B$ and agrees with $s$ outside $B$.  By Lemma~\ref{basicproperties}(v), $f$ is superharmonic.  By Lemma~\ref{basicproperties}(i), $f$ attains its minimum in $\overline{B}$ at a point $z \in \partial B$, hence for $y\in B$
	\[ f(y) \geq f(z) = s(z) \geq s(x) - \frac{\epsilon}{2} = \gamma(x) + \frac{\epsilon}{2} > \gamma(y). \]
It follows that $f \geq\gamma$ everywhere, hence $f \geq s$.  From Lemma~\ref{basicproperties}(ii) we conclude that $f=s$, and hence $s$ is harmonic at $x$.
\end{proof}

\subsection{Superharmonic Potentials}
\label{superharmonicpotentials}
Next we describe the particular class of obstacles which relate to the aggregation models we are studying.  For a bounded measurable function $\sigma$ on $\R^d$ with compact support, write
	\begin{equation} \label{thepotential} G\sigma(x) = \int_{\R^d} g(x,y) \sigma(y) dy, \index{$G\sigma$, superharmonic potential} \end{equation}
where
\begin{equation} \label{greenskernel} g(x,y) = \begin{cases} -\frac{2}{\pi} \log |x-y|, & d=2; \\ a_d |x-y|^{2-d}, & d\geq 3. \end{cases} \index{$g(x,y)$, Green's function on $\R^d$} \end{equation}
Here $a_d = \frac{2}{(d-2)\omega_d}$, where $\omega_d$ is the volume of the unit ball in $\R^d$.  Note that (\ref{greenskernel}) differs by a factor of $2d$ from the standard harmonic potential in $\R^d$; however, the normalization we have chosen is most convenient when working with the discrete Laplacian and random walk. 

The following result is standard; see \cite[Theorem 1.I.7.2]{Doob}.

\begin{lemma} Let $\sigma$ be a measurable function on $\R^d$ with compact support.
\label{smoothnessofpotential}
\begin{enumerate}
\item[(i)] If $\sigma$ is bounded, then $G\sigma$ is continuous.
\item[(ii)] If $\sigma$ is $C^1$, then $G\sigma$ is $C^2$ and
	\begin{equation} \label{laplacianofpotential} \Delta G \sigma = -2d\sigma. \end{equation}
\end{enumerate}
\end{lemma}

 
Regarding (ii), if we remove the smoothness assumption on $\sigma$, equation (\ref{laplacianofpotential}) remains true in the sense of distributions.
For our applications, however, we will not need this fact, and the following lemma will suffice.

\begin{lemma}
\label{superharmonicpotential}
Let $\sigma$ be a bounded measurable function on $\R^d$ with compact support.
If $\sigma \geq 0$ on an open set $\Omega \subset \R^d$, then $G \sigma$ is superharmonic on $\Omega$.  
\end{lemma}

\begin{proof}
Suppose $B(x,r) \subset \Omega$.  Since for any fixed $y$ the function $f(x) = g(x,y)$ is superharmonic in $x$, we have
	\begin{align} G\sigma(x) &= \int_{\R^d} \sigma(y) g(x,y) dy \nonumber \\
			&\geq \int_{\R^d} \sigma(y) A_r f(y) dy \nonumber \\
			&= A_r G\sigma(x). \qed \nonumber \end{align}
\renewcommand{\qedsymbol}{}
\end{proof}

\noindent By applying Lemma~\ref{superharmonicpotential} both to $\sigma$ and to $-\sigma$, we obtain that $G\sigma$ is harmonic off the support of $\sigma$.

Let $B=B(o,r)$ be the ball of radius $r$ centered at the origin in $\R^d$.  We compute in dimensions $d\geq 3$
	\begin{equation} \label{ballpotential} G1_B(x) = \begin{cases} \frac{dr^2}{d-2} - |x|^2, & |x|<r \\
						  \frac{2 r^2}{d-2} \Big( \frac{r}{|x|} \Big)^{d-2}, & |x|\geq r. \end{cases} 		\end{equation}
Likewise in two dimensions
	\begin{equation} \label{ballpotentialdim2} G1_B(x) = \begin{cases} r^2(1-2\log r)-|x|^2, & |x|<r \\
						 -2r^2 \log |x|, & |x| \geq r. \end{cases}
						 \end{equation}
						 
Fix a bounded nonnegative function $\sigma$ on $\R^d$ with compact support\index{$\sigma$, mass density on $\R^d$}, and let
	\begin{equation} \label{theobstacle} \gamma(x) = -|x|^2 - G\sigma(x). \index{$\gamma$, obstacle on $\R^d$} \end{equation}
Let
	\begin{equation} \label{themajorant} s(x) = \inf \{f(x) | f \text{ is continuous, superharmonic and }  f \geq \gamma \} \index{$s$, superharmonic majorant in~$\R^d$} \end{equation}
be the least superharmonic majorant of $\gamma$, and let
	\begin{equation} \label{thenoncoincidenceset} D = \{x\in \R^d| s(x)>\gamma(x)\} \index{$D$, noncoincidence set} \end{equation}
be the noncoincidence set.

\begin{lemma}
\label{laplacianofobstacle}
\begin{itemize}
\item[(i)] $\gamma(x) + |x|^2$ is subharmonic on $\R^d$. 
\item[(ii)] If $\sigma \leq M$ on an open set $\Omega \subset \R^d$, then $\gamma(x) - (M-1)|x|^2$ is superharmonic on $\Omega$.
\end{itemize}
\end{lemma}

\begin{proof}
(i) By Lemma~\ref{superharmonicpotential}, since $\sigma$ is nonnegative, the function $\gamma(x) + |x|^2 = -G\sigma(x)$ is subharmonic on $\R^d$.  

(ii) Let $B=B(o,R)$ be a ball containing the support of $\sigma$.  By (\ref{ballpotential}) and (\ref{ballpotentialdim2}), for $x \in B$ we have
	\[ |x|^2 = c_d R^2 - G1_B(x) \]
where $c_2 = 1-2\log R$ and $c_d = \frac{d}{d-2}$ for $d \geq 3$.  Hence for $x \in B$ we have
	\begin{align*} \gamma(x) - (M-1)|x|^2 &= -G\sigma(x) -M|x|^2 \\
								  &= G(M1_B - \sigma)(x) - c_d M R^2. \end{align*}
Since $\sigma \leq M1_B$ on $\Omega$, by Lemma~\ref{superharmonicpotential} the function $\gamma - (M-1)|x|^2$ is superharmonic in $B \cap \Omega$.  Since this holds for all sufficiently large $R$, it follows that $\gamma - (M-1)|x|^2$ is superharmonic on all of $\Omega$.
\end{proof}

\begin{lemma}
\label{laplacianofodometer}
Let $u=s-\gamma$, where $\gamma$ and $s$ are given by (\ref{theobstacle}) and (\ref{themajorant}).  Then
\begin{itemize}
\item[(i)] $u(x) - |x|^2$ is superharmonic on $\R^d$.
\item[(ii)] If $\sigma \leq M$ on an open set $\Omega \subset \R^d$, then $u(x) + M|x|^2$ is subharmonic on $\Omega$.
\end{itemize}
\end{lemma}

\begin{proof}
(i) By Lemmas~\ref{majorantbasicprops}(i) and~\ref{laplacianofobstacle}(i), the function
	\[ u - |x|^2 = s - (\gamma + |x|^2) \]
is the difference of a superharmonic and a subharmonic function, hence superharmonic on $\R^d$.

(ii) With $A_r$ defined by (\ref{meanvalueproperty}), we have
	\begin{align*} A_r |x|^2 &= \frac{1}{\omega_d r^d} \int_{B(o,r)} (|x|^2 + 2x\cdot y + |y|^2) \,dy  \\ 
	&= |x|^2 + \frac{1}{\omega_d r^d} \int_0^r (d\omega_d t^{d-1})t^2 \,dt \\
	&= |x|^2 + \frac{dr^2}{d+2}. \end{align*}
By Lemma~\ref{laplacianofobstacle} we have
	\[ A_r s + \frac{dr^2}{d+2} \geq A_r \gamma + A_r |x|^2 - |x|^2 \geq \gamma. \]
Since $A_r s$ is continuous and superharmonic, it follows that $A_r s + \frac{dr^2}{d+2} \geq s$, hence
	\[ A_r s + A_r |x|^2 = A_r s + |x|^2 + \frac{dr^2}{d+2} \geq s + |x|^2. \]
Thus $s + |x|^2$ is subharmonic on $\R^d$, and hence by Lemma~\ref{laplacianofobstacle}(ii) the function
	\[ u + M|x|^2 = (s+|x|^2) - (\gamma - (M-1)|x|^2) \]
is subharmonic on $\Omega$.
\end{proof}

For $A \subset \R^d$, write $A^o$ for the interior of $A$ and $\bar{A}$ for the closure of $A$.  The boundary of $A$ is $\partial A = \bar{A}-A^o$. \index{$A^o$, interior} \index{$\bar{A}$, closure}

\begin{lemma}
\label{startingdensitygreaterthan1}
Let $D$ be given by (\ref{thenoncoincidenceset}).  Then $\{\sigma>1\}^o \subset D$.
\end{lemma}

\begin{proof}
If $\sigma>1$ in a ball $B=B(x,r)$, by (\ref{ballpotential}) and (\ref{ballpotentialdim2}), for $y\in B$ we have
	\[ \gamma(y) = -|y|^2 - G\sigma(y) = - c_d r^2 - G(\sigma-1_B)(y). \]
By Lemma~\ref{superharmonicpotential} it follows that $\gamma$ is subharmonic in $B$.  In particular, $s>\gamma$ in $B$, so $x \in D$.
\end{proof}

The next lemma concerns the monotonicity of our model: by starting with more mass, we obtain a larger odometer and a larger noncoincidence set; see also \cite{Sakai82}.

\begin{lemma}
\label{monotonicity}
Let $\sigma_1 \leq \sigma_2$ be functions on $\R^d$ with compact support, and suppose that $\int_{\R^d} \sigma_2(x)dx < \infty$.  Let
	\[ \gamma_i(x) = -|x|^2 - G\sigma_i(x), \qquad i=1,2. \]
Let $s_i$ be the least superharmonic majorant of $\gamma_i$, let $u_i = s_i-\gamma_i$, and let
	\[ D_i = \{ x | s_i(x)>\gamma_i(x) \}. \]
Then $u_1 \leq u_2$ and $D_1 \subset D_2$.
\end{lemma}

\begin{proof}
Let
	\[ \tilde{s} = s_2 + G(\sigma_2-\sigma_1). \]
Then $\tilde{s}$ is continuous and superharmonic, and since $s_2(x) \geq -|x|^2-G\sigma_2(x)$ we have
	\[ \tilde{s}(x) \geq - |x|^2 - G\sigma_1(x) \]
hence $\tilde{s} \geq s_1$.  Now 
	\[ u_2-u_1 = s_2 - s_1 + G(\sigma_2-\sigma_1)	 = \tilde{s}-s_1 \geq 0. \]
Since $D_i$ is the support of $u_i$, the result follows.
\end{proof}

Our next lemma shows that we can restrict to a domain $\Omega \subset \R^d$ when taking the least superharmonic majorant, provided that $\Omega$ contains the noncoincidence set.

\begin{lemma}
\label{majorantonacompactset}
Let $\gamma,s,D$ be given by (\ref{theobstacle})-(\ref{thenoncoincidenceset}).  Let $\Omega \subset \R^d$ be an open set with $D \subset \Omega$.  Then 
	\begin{equation} \label{majorantinadomain} s(x) = \inf \{f(x)|\text{$f$ is superharmonic on $\Omega$, continuous, and $f \geq \gamma$}\}. \end{equation}
\end{lemma}

\begin{proof}
Let $f$ be any continuous function which is superharmonic on $\Omega$ and $\geq \gamma$.  By Lemma~\ref{majorantbasicprops}(iii), $s$ is harmonic on $D$, so $f-s$ is superharmonic on $D$ and attains its minimum in $\overline{D}$ on the boundary.  Hence $f-s$ attains its minimum in $\overline{\Omega}$ at a point $x$ where $s(x)=\gamma(x)$.  Since $f \geq \gamma$ we conclude that $f \geq s$ on $\Omega$ and hence everywhere.  Thus $s$ is at most the infimum in (\ref{majorantinadomain}).
Since the infimum in (\ref{majorantinadomain}) is taken over a strictly larger set than that in (\ref{themajorant}), the reverse inequality is trivial.
\end{proof}

\subsection{Boundary Regularity for the Obstacle Problem}

Next we turn to the regularity of the solution to the obstacle problem (\ref{themajorant}) and of the free boundary $\partial D$.  There is a substantial literature on boundary regularity for the obstacle problem. In our setup, however, extra care must be taken near points where $\sigma(x)=1$: at these points the obstacle (\ref{theobstacle}) is harmonic, and the free boundary can be badly behaved.
We show only the minimal amount of regularity required for the proofs of our main theorems.  Much stronger regularity results are known in related settings; see, for example, \cite{Caffarelli, CKS}.

The following lemma shows that if the obstacle is sufficiently smooth, then the superharmonic majorant cannot separate too quickly from the obstacle near the boundary of the noncoincidence set.   For the proof, we follow the sketch in Caffarelli \cite[Theorem 2]{Caffarelli}.
As usual, we write $D_\epsilon$ for the inner $\epsilon$-neighborhood of $D$, given by (\ref{innerepsilonneighborhood})

\begin{lemma}
\label{smoothdensity}
Let $\sigma$ be a $C^1$ function on $\R^d$ with compact support.  Let $\gamma,s,D$ be given by (\ref{theobstacle})-(\ref{thenoncoincidenceset}), and write $u = s-\gamma$.  Then $u$ is $C^1$, and for $y \in \partial D_\epsilon$ we have $|\nabla u (y)| \leq C_0 \epsilon$, for a constant $C_0$ depending on $\sigma$.
\end{lemma}

\begin{proof}
Fix $x_0 \in \partial D$, and define
 	\[ L(x) = \gamma(x_0) + \langle \nabla \gamma(x_0), x-x_0 \rangle. \]
Since $\sigma$ is $C^1$, we have that $\gamma$ is $C^2$ by Lemma~\ref{smoothnessofpotential}(ii).
Let $A$ be the maximum second partial of $\gamma$ in the ball $B  = B(x_0, 4\epsilon)$.  By the mean value theorem and Cauchy-Schwarz, for $x \in B$ we have
	\begin{align} |L(x) - \gamma(x)| &= |\langle \nabla \gamma(x_0) - \nabla \gamma(x_*), x-x_0 \rangle| \nonumber \\
	&\leq A \sqrt{d} |x_0 - x_*| |x-x_0| \leq c \epsilon^2, \label{secondordererror} \end{align}
where $c = 16 A \sqrt{d}$.  Hence
	\[ s(x)  \geq \gamma(x) \geq L(x) - c\epsilon^2, \qquad x \in B. \]
Thus the function
	\[ w = s - L + c\epsilon^2 \]
is nonnegative and superharmonic in $B$.  Write $w = w_0+w_1$, where $w_0$ is harmonic and equal to $w$ on $\partial B$.  Then since $s(x_0) = \gamma(x_0)$, we have
	\[ w_0(x_0) \leq w(x_0) = s(x_0) - L(x_0) + c\epsilon^2 = c \epsilon^2. \]
By the Harnack inequality, it follows that $w_0 \leq c' \epsilon^2$ on the ball $B' = B(x_0, 2\epsilon)$, for a suitable constant $c'$. 

Since $w_1$ is nonnegative and vanishes on $\partial B$, it attains its maximum in $\bar{B}$ at a point $x_1$ in the support of its Laplacian.  Since $\Delta w_1 = \Delta s$, by Lemma~\ref{majorantbasicprops}(iii) we have $s(x_1) = \gamma(x_1)$, hence
	\[ w_1(x_1) \leq w(x_1) = s(x_1) - L(x_1) + c\epsilon^2 \leq 2 c \epsilon^2, \]
where in the last step we have used (\ref{secondordererror}).  We conclude that $0 \leq w \leq (2c+c')\epsilon^2$ on $B'$ and hence $|s-L| \leq (c+c') \epsilon^2$ on $B'$.  Thus on $B'$ we have
	\[ |u| = |s-\gamma| \leq |s - L| + |\gamma-L| \leq (2c+c')\epsilon^2. \]
In particular, $u$ is differentiable at $x_0$, and $\nabla u(x_0)=0$.  Since $s$ is harmonic in $D$ and equal to $\gamma$ outside $D$, it follows that $u$ is differentiable everywhere, and $C^1$ off $\partial D$.

Given $y \in \partial D_\epsilon$, let $x_0$ be the closest point in $\partial D$ to $y$.  Since $B(y,\epsilon) \subset D$, by Lemma~\ref{majorantbasicprops}(iii) the function $w$ is harmonic on $B(y,\epsilon)$, so by the Cauchy estimate \cite[Theorem 2.4]{HFT} we have
	\begin{equation} \label{fromcauchyest} |\nabla s(y) - \nabla \gamma(x_0)| = |\nabla w(y) | \leq \frac{C}{\epsilon} \sup_{z\in B(y,\epsilon)} w(z). \end{equation}
Since $B(y,\epsilon) \subset B'$, the right side of (\ref{fromcauchyest}) is at most $C(2c+c')\epsilon$, hence
	\[ |\nabla u(y)| \leq |\nabla s(y) - \nabla \gamma(x_0)| +  |\nabla \gamma(y) - \nabla \gamma(x_0)| \leq C_0 \epsilon \]
where $C_0 = C(c+c') + A \sqrt{d}$.  Thus $u$ is $C^1$ on $\partial D$ as well.
\end{proof}

Next we show that mass is conserved in our model: the amount of mass starting in $D$ is $\int_D \sigma(x) dx$, while the amount of mass ending in $D$ is $\Leb(D)$, the Lebesgue measure of $D$.
Since no mass moves outside of $D$, we expect these to be equal.  Although this seems intuitively obvious, the proof takes some work because we have no {\it a priori} control of the domain $D$.  In particular, we first need to show that the boundary of $D$ cannot have positive $d$-dimensional Lebesgue measure.

\begin{prop}
\label{boundaryregularitysmooth}
Let $\sigma$ be a $C^1$ function on $\R^d$ with compact support, such that $\Leb(\sigma^{-1}(1))=0$. Let $D$ be given by (\ref{thenoncoincidenceset}).  Then 
\begin{itemize}
\item[(i)] $\Leb (\partial D) = 0$.
\item[(ii)] For any function $h \in C^1(\bar{D})$ which is superharmonic on $D$,
	\[  \int_D h(x) dx \leq \int_D h(x) \sigma(x) dx. \]
\end{itemize}
\end{prop}

Note that by applying (ii) both to $h$ and to $-h$, equality holds whenever $h$ is harmonic on $D$.  In particular, taking $h=1$ yields the conservation of mass: $\int_D \sigma(x)dx = \Leb(D)$.

The proof of part (i) follows Friedman \cite[Ch.\ 2, Theorem~3.5]{Friedman}.  It uses the Lebesgue density theorem, as stated in the next lemma.

\begin{lemma}
\label{density1}
Let $A \subset \R^d$ be a Lebesgue measurable set, and let
	\[ A' = \big\{x \in A \,|\, \liminf_{\epsilon \rightarrow 0} \epsilon^{-d} \Leb\big(B(x,\epsilon) \cap A^c\big) > 0 \big\}. \]
Then $\Leb(A')=0$.
\end{lemma}

To prove the first part of Proposition~\ref{boundaryregularitysmooth}, given a boundary point $x \in \partial D$, the idea is first to find a point $y$ in the ball $B(x,\epsilon)$ where $u=s-\gamma$ is relatively large, and then to argue using Lemma~\ref{smoothdensity} that a ball $B(y,c\epsilon)$ must be entirely contained in $D$.  Taking $A=\partial D$ in the Lebesgue density theorem, we obtain that $x \in A'$.

The proof of the second part of Proposition~\ref{boundaryregularitysmooth} uses Green's theorem in the form
	\begin{equation} \label{greensidentity} \int_{D'} (u\Delta h -  \Delta u h) dx = \int_{\partial D'} \left(u \frac{\partial h}{\partial\normal} - \frac{\partial u}{\partial \normal} h \right) dr. \end{equation}
Here $D'$ is the union of boxes $x^\Box$ that are contained in $D$, and $dx$ is the volume measure in $D'$, while $dr$ is the $(d-1)$-dimensional surface measure on $\partial D'$.  The partial derivatives on the right side of (\ref{greensidentity}) are in the outward normal direction from $\partial D'$.  

\begin{proof}[Proof of Propostion~\ref{boundaryregularitysmooth}]
(i) Fix $0<\lambda<1$.  For $x \in \partial D$ with $\sigma(x) \leq \lambda$, for small enough $\epsilon$ we have $\sigma \leq \frac{1+\lambda}{2}$ on $B(x,\epsilon)$.  By Lemma~\ref{laplacianofobstacle}(ii) the function
	\[ f(y) = \gamma(y) + \frac{1-\lambda}{2} |y|^2 \]
is superharmonic on $B(x,\epsilon) \cap D$, so by Lemma~\ref{majorantbasicprops}(iii) the function
	\begin{align*} w(y) &= u(y) - \frac{1-\lambda}{2} |x-y|^2 \\
					&=  s(y) - f(y) + (1-\lambda) \langle x,y \rangle - \frac{1-\lambda}{2} |x|^2 \end{align*}
is subharmonic on $B(x,\epsilon) \cap D$.  Since $w(x)=0$, its maximum is not attained on $\partial D$, so it must be attained on $\partial B(x,\epsilon)$, so there is a point $y$ with $|x-y|=\epsilon$ and
	\[ u(y) \geq \frac{1-\lambda}{2} \epsilon^2. \]

By Lemma~\ref{smoothdensity} we have $|\nabla u| \leq (1+c) C_0 \epsilon$ in the ball $B(y,c\epsilon)$.  Taking $c = \frac{1-\lambda}{4C_0}$,
we obtain for $z \in B(y,c\epsilon)$
	\begin{align*} u(z) &= u(y) + \langle \nabla u (y_*), z-y \rangle \\
				&\geq \frac{1-\lambda}{2} \epsilon^2 - c (1+c) C_0 \epsilon^2 > 0. \end{align*}
Thus for any $x \in \partial D \cap \{\sigma \leq \lambda\}$
	\[ \Leb(B(x,(1+c)\epsilon) \cap (\partial D)^c) \geq \omega_d (c \epsilon)^d. \]
By the Lebesgue density theorem it follows that 
	\begin{equation} \label{fromthedensitytheorem} \Leb(\partial D \cap \{\sigma \leq \lambda\}) = 0. \end{equation} 
By Lemma~\ref{startingdensitygreaterthan1} we have $\sigma \leq 1$ on $\partial D$.  Taking $\lambda \uparrow 1$ in (\ref{fromthedensitytheorem}), we obtain 
	\[ \Leb(\partial D) \leq \Leb(\partial D \cap \{\sigma < 1\}) + \Leb(\sigma^{-1}(1)) = 0. \]

(ii) Fix $\delta>0$ and let
	\[ D' = \bigcup_{x\in \delta \Z^d \,:\, x^\Box \subset D} x^\Box, \]
where $x^\Box = x + [-\frac{\delta}{2}, \frac{\delta}{2}]^d$.  By Lemmas~\ref{majorantbasicprops}(iii) and~\ref{smoothnessofpotential}(ii), in $D'$ we have
	\[ \Delta u = \Delta s - \Delta \gamma = 2d(1-\sigma). \]
Now by Green's theorem (\ref{greensidentity}), since $u$ is nonnegative and $h$ is superharmonic,
	\begin{equation} \label{greensthm} \int_{D'} (1-\sigma) h \,dx = \frac{1}{2d} \int_{D'} \Delta u h \,dx \leq \frac{1}{2d} \int_{\partial D'} \left(u \frac{\partial h}{\partial \normal} - \frac{\partial u}{\partial\normal}h \right) dr \end{equation}
where $\normal$ denotes the unit outward normal vector to $\partial D'$.
By Lemma~\ref{smoothdensity}, the integral on the right side is bounded by
	\begin{equation} \label{surfaceintegral} \int_{\partial D'} \left| u \frac{\partial h}{\partial \normal} \right| + \left| \frac{\partial u}{\partial\normal}h \right| dr \leq C_0 \sqrt{d} (\delta^2 ||\nabla h||_{\infty} + \delta ||h||_{\infty}) \Leb_{d-1}(\partial D'), \end{equation}
where $\Leb_{d-1}$ denotes $(d-1)$-dimensional Lebesgue measure.  Let
	\[ S = \bigcup_{x \in \delta \Z^d \,:\, x^\Box \cap \partial D \neq \emptyset} x^\Box. \]
Since $\Leb(\partial D)=0$, given $\epsilon>0$ we can choose $\delta$ small enough so that $\Leb(S)<\epsilon$.  Since $\partial D' \subset \partial S$, we have
	\[ \Leb_{d-1}(\partial D') \leq \Leb_{d-1}(\partial S) \leq \frac{2d\Leb(S)}{\delta}. \]
Since $D$ is open, $\Leb(D') \uparrow \Leb(D)$ as $\delta \downarrow 0$. Taking $\delta<\epsilon$ smaller if necessary so that $\Leb(D') \geq \Leb(D) - \epsilon$, we obtain from (\ref{greensthm}) and (\ref{surfaceintegral})
	\[ \int_{D} (1-\sigma)h\,dx \leq (M+1)\epsilon + C_0 \sqrt{d} (||h||_{\infty} + ||\nabla h||_{\infty}) \epsilon \]
where $M$ is the maximum of $|\sigma|$.  Since this holds for any $\epsilon>0$, we conclude that $\int_D h(x) dx \leq \int_D h(x) \sigma(x) dx$.
\end{proof}

We will need a version of Proposition~\ref{boundaryregularitysmooth} which does not assume that $\sigma$ is $C^1$ or even continuous.  We can replace the $C^1$ assumption and the condition that $\Leb(\sigma^{-1}(1))=0$ by the following condition.
\begin{equation} \label{boundedawayfrom1again} 
	\text{For all }x\in \R^d \text{ either }\sigma(x) \leq \lambda \text{ or } \sigma(x)\geq 1
	\end{equation}
for a constant $\lambda<1$.  Then we have the following result.

\begin{prop}
\label{boundaryregularity}
Let $\sigma$ be a bounded function on $\R^d$ with compact support.  Let $D$ be given by (\ref{thenoncoincidenceset}), and let $\widetilde{D} = D \cup \{\sigma \geq 1\}^o$.  If $\sigma$ is continuous almost everywhere and satisfies (\ref{boundedawayfrom1again}), then
\begin{itemize}
\item[(i)] $\Leb \big(\partial \widetilde{D}\big) = 0$.
\item[(ii)] $\Leb \big(D\big) = \int_D \sigma(x) dx$.
\item[(iii)] $\Leb \big(\widetilde{D}\big) = \int_{\widetilde{D}} \sigma(x) dx$.
\end{itemize}
\end{prop}

In particular, taking $\sigma = 1_A + 1_B$, we have 
	\[ A \oplus B = A \cup B \cup D = \widetilde{D} \]
where $\oplus$ denotes the smash sum (\ref{smashsumdef}).
From (iii) we obtain the volume additivity of the smash sum.

\begin{corollary}
\label{volumesadd}
Let $A,B \subset \R^d$ be bounded open sets whose boundaries have measure zero.  Then
	\[ \Leb(A \oplus B) = \Leb(A) + \Leb(B). \]
\end{corollary}

\begin{proof}[Proof of Proposition~\ref{boundaryregularity}]
(i) Fix $\epsilon>0$.  Since $\sigma$ is continuous almost everywhere, there exist $C^1$ functions $\sigma_0 \leq \sigma \leq \sigma_1$ with $\int_{\R^d} (\sigma_1 - \sigma_0) dx < \epsilon$.  Scaling by a factor of $1+\delta$ for sufficiently small $\delta$, we can ensure that $\Leb(\sigma_i^{-1}(1))=0$, so that $\sigma_0$ and $\sigma_1$ satisfy the hypotheses of Proposition~\ref{boundaryregularitysmooth}.  For $i=0,1$ let 
	\[ \gamma_i(x) = -|x|^2 - G\sigma_i(x), \]
and let $s_i$ be the least superharmonic majorant of $\gamma_i$.  Choose $\alpha$ with $\lambda<\alpha<1$ such that $\Leb(\sigma_i^{-1}(\alpha))=0$ for $i=0,1$, and write 
	\begin{align*} &D_i = \{s_i>\gamma_i\}; \\ 
		&S_i = D_i \cup \{\sigma_i \geq \alpha\}. \end{align*}	
By (\ref{boundedawayfrom1again}) we have $\{\sigma_0 \geq \alpha \} \subset \{\sigma \geq 1\}^o$, hence by Lemma~\ref{monotonicity}
	\begin{align} \label{leftsqueeze} S_0 = D_0 \cup \{\sigma_0 \geq \alpha\} \subset D \cup \{\sigma \geq 1\}^o &= \widetilde{D} \\  
	\label{rightsqueeze} &\subset \overline{\widetilde{D}} \subset \overline{D_1 \cup \{\sigma \geq 1\}} = \overline{S_1}. \end{align}
For $i=0,1$ write
	\[ \sigma_i^\circ (x)  =  \begin{cases} 1, & x \in D_i \\ \sigma_i(x), & x \not\in D_i. \end{cases} \]
By Proposition~\ref{boundaryregularitysmooth}(ii) with $h=1$, we have
	\begin{equation} \label{conservation} \int_{\R^d} \sigma_i^\circ = \Leb(D_i) + \int_{D_i^c} \sigma_i = \int_{\R^d} \sigma_i. \end{equation}
For $0< \alpha_0 < \alpha$, we have
	\begin{equation} \label{thisissmall} \Leb(S_1 - S_0) \leq \Leb \big( S_1 \cap \{ \sigma_0^\circ \leq \alpha_0\} \big) + \Leb \big( \{ \alpha_0 < \sigma_0^\circ < \alpha \} \big). \end{equation}
Since $\sigma_1^\circ \geq \alpha$ on $S_1$, the first term is bounded by
	\[ \Leb \big( S_1 \cap \{ \sigma_0^\circ \leq \alpha_0\} \big) \leq \frac{|| \sigma_1^\circ - \sigma_0^\circ ||_1}{\alpha-\alpha_0} = \frac{||\sigma_1 - \sigma_0||_1}{\alpha-\alpha_0} < \frac{\epsilon}{\alpha-\alpha_0}, \]
where the equality in the middle step follows from (\ref{conservation}).

By Proposition~\ref{boundaryregularitysmooth}(i) we have $\Leb(\partial D_1) = 0$, hence $\Leb(\partial S_1) = 0$.  Taking $\alpha_0 = \alpha - \sqrt{\epsilon}$ we obtain from (\ref{leftsqueeze}), (\ref{rightsqueeze}), and (\ref{thisissmall})
	 \[ \Leb \big(\partial \widetilde{D} \big) \leq \Leb \big(\overline{S_1} - S_0\big) = \Leb(S_1 - S_0) < \sqrt{\epsilon} + \Leb \big( \{ \alpha -\sqrt{\epsilon} < \sigma_0 < \alpha \} \big). \]
Since this holds for any $\epsilon>0$, the result follows.

(ii) Write
	\[ \sigma^\circ (x)  =  \begin{cases} 1, & x \in D \\ \sigma(x), & x \not\in D. \end{cases} \]
From (\ref{conservation}) we have
	\[ \int_{\R^d} \sigma_0 = \int_{\R^d} \sigma_0^\circ \leq \int_{\R^d} \sigma^\circ \leq \int_{\R^d} \sigma_1^\circ = \int_{\R^d} \sigma_1. \]
The left and right side differ by at most $\epsilon$.  Since $\int_{\R^d} \sigma_0 \leq \int_{\R^d} \sigma \leq \int_{\R^d} \sigma_1$, and $\epsilon>0$ is arbitrary, it follows that $\int_{\R^d} \sigma = \int_{\R^d}\sigma^\circ$, hence
	\begin{equation} \label{conservationinD}  \int_D \sigma = \int_{\R^d} \sigma^\circ - \int_{D^c} \sigma = \Leb(D). \end{equation}
	
(iii) By Lemma~\ref{startingdensitygreaterthan1}, if $\sigma(x)>1$ and $x \notin D$, then $\sigma$ is discontinuous at~$x$.  Thus $\sigma=1$ almost everywhere on $\widetilde{D}-D$, and we obtain from (\ref{conservationinD})
	\[ \int_{\widetilde{D}} \sigma = \Leb(D) + \int_{\widetilde{D}-D} \sigma = \Leb(\widetilde{D}). \qed \]
\renewcommand{\qedsymbol}{}
\end{proof}

Our next lemma describes the domain resulting from starting mass $m>1$ on a ball in $\R^d$.  Not surprisingly, the result is another ball, concentric with the original, of $m$ times the volume.  In particular, if $m$ is an integer, the $m$-fold smash sum of a ball with itself is again a ball.
						 
\begin{lemma}
\label{relaxingaball}
Fix $m>1$, and let $D$ be given by (\ref{thenoncoincidenceset}) with $\sigma = m1_{B(o,r)}$.  Then $D = B(o,m^{1/d}r)$.
\end{lemma}

\begin{proof}
Since $\gamma(x)=-|x|^2-G\sigma(x)$ is spherically symmetric and the least superharmonic majorant commutes with rotations, $D$ is a ball centered at the origin.  By Proposition~\ref{boundaryregularity}(ii) we have $\Leb(D) = m \Leb(B(o,r))$.
\end{proof}

Next we show that the noncoincidence set is bounded; see also \cite[Cor.~7.2]{Sakai}.

\begin{lemma}
\label{occupieddomainisbounded}
Let $\sigma$ be a function on $\R^d$ with compact support, satisfying $0 \leq \sigma \leq M$.  Let $D$ be given by (\ref{thenoncoincidenceset}).  Then $D$ is bounded.
\end{lemma}

\begin{proof}
Let $B=B(o,r)$ be a ball containing the support of $\sigma$.  By Lemmas~\ref{monotonicity} and~\ref{relaxingaball}, the difference $s-\gamma$ is supported in $B(o,M^{1/d}r)$.
\end{proof}

\subsection{Convergence of Obstacles, Majorants and Domains}

In broad outline, many of our arguments have the following basic structure:
	\begin{align*} \text{convergence of densities} &\implies \text{convergence of obstacles} \\
			&\implies \text{convergence of majorants} \\
			&\implies \text{convergence of domains.} \end{align*}
An appropriate convergence of starting densities is built into the hypotheses of the theorems.  From these hypotheses we use Green's function estimates to deduce the relevant convergence of obstacles.  Next, as we have already seen in Lemmas~\ref{laplacianofobstacle} and~\ref{laplacianofodometer}, properties of the obstacle can often be parlayed into corresponding properties of the least superharmonic majorant.  Finally, deducing convergence of domains (i.e., noncoincidence sets) from the convergence of the majorants often requires additional regularity assumptions.  The following lemma illustrates this basic three-step approach.  

\begin{lemma}
\label{threesteps}
Let $\sigma$ and $\sigma_n$, $n=1,2,\dots$ be densities on $\R^d$ satisfying
	\[ 0 \leq \sigma, \sigma_n \leq M 1_B \]
for a constant $M$ and ball $B=B(o,R)$.  Suppose that $\sigma$ is continuous except on a set of Lebesgue measure zero, and that	
	\begin{equation} \label{convergingdensitiescontinuum} \sigma_n(x) \rightarrow \sigma(x) \end{equation}
as $n \rightarrow \infty$, for all continuity points $x$ of $\sigma$.  Then
\begin{enumerate}
\item[(i)] $G \sigma_n \rightarrow G \sigma$ uniformly on compact subsets of $\R^d$.
\item[(ii)] $s_n \rightarrow s$ uniformly on compact subsets of $\R^d$, where $s,s_n$ are the least superharmonic majorants of the functions $\gamma = -|x|^2 - G\sigma$, and $\gamma_n = -|x|^2 - G\sigma_n$, respectively.
\item[(iii)] For any $\epsilon>0$ we have $D_\epsilon \subset D_{(n)}$ for all sufficiently large $n$, where $D,D_{(n)}$ are the noncoincidence sets $\{s>\gamma\}$ and $\{s_n>\gamma_n\}$, respectively.
\end{enumerate}
\end{lemma}

\begin{proof}
(i) Since $\sigma,\sigma_n$ are supported on $B$, we have
	\begin{equation} \label{triangleineq} | G\sigma_n(x) - G\sigma(x) | \leq \int_B |g(x,y)| |\sigma_n(y)-\sigma(y)| dy. \end{equation}
Fix $0<\epsilon<1$.  Since $0 \leq \sigma,\sigma_n \leq M$, we have by (\ref{ballpotential}) and (\ref{ballpotentialdim2})
	\begin{equation} \label{tinyball} \int_{B(x,\epsilon)} |g(x,y)| |\sigma_n(y) - \sigma(y)| dy \leq 3M \epsilon^2.	
\end{equation}
Now let $B_0 = B-B(x,\epsilon)$ and
	\[ A_N = \{ y\in B_0 ~:~ |\sigma_n(y)-\sigma(y)|<\epsilon \text{ for all } n\geq N\}. \]
Then by (\ref{convergingdensitiescontinuum})
	\[ \bigcup_{N\geq 1} A_N \supseteq B_0-DC(\sigma) \]
where $DC(\sigma)$ is the set of discontinuities of $\sigma$.  As $\Leb(DC(\sigma))=0$ and the sets $A_N$ are monotone increasing in $N$, we can choose $N$ large enough so that
	\begin{equation} \label{almostallcontinuitypoints} \Leb(A_N) > \Leb(B_0) - \frac{\epsilon^{d-1}}{\log(R/\epsilon)}. \end{equation}
Given $n\geq N$, for $y \in B_0$ we have
	\[ |g(x,y)| \leq a_d \epsilon^{2-d} \log \frac{R}{\epsilon} \leq \frac{a_d \epsilon}{\Leb(B_0-A_n)}. \] 
Splitting the right side of (\ref{triangleineq}) into separate integrals over $B(x,\epsilon)$, $A_n$, and $B_0-A_n$, we obtain
	\begin{align*}	| G\sigma_n(x) - G\sigma(x) | 
		&< 6M\epsilon^2 + \epsilon \int_{A_n} |g(x,y)| dy + 2Ma_d \epsilon \\
		&\leq \left(3R^2 \log \text{max}(R,|x|) + (2a_d+3)M \right) \epsilon. \end{align*}

(ii) By Lemmas~\ref{monotonicity} and~\ref{relaxingaball}, the noncoincidence sets $D,D_{(n)}$ are contained in the ball $B_1 = M^{1/d} B$.  Given $\epsilon>0$ and a compact set $K$ containing $B_1$, choose $N$ large enough so that $|\gamma_n-\gamma|<\epsilon$ on $K$ for all $n\geq N$.  Then
	\[ s+\epsilon \geq \gamma + \epsilon > \gamma_n \]
on $K$, so the function $f_n= $ max$(s+\epsilon,\gamma_n)$ is superharmonic on $K$.  By Lemma~\ref{majorantonacompactset} we have $f_n \geq s_n$, and hence $s + \epsilon \geq s_n$ on $K$.  Likewise
	\[ s_n + \epsilon \geq \gamma_n + \epsilon > \gamma \]
on $K$, so the function $\tilde{f}_n =$ max$(s_n+\epsilon,\gamma)$ is superharmonic on $K$.  By Lemma~\ref{majorantonacompactset} we have $\tilde{f}_n \geq s$ and hence $s_n + \epsilon \geq s$ on $K$.  Thus $s_n \rightarrow s$ uniformly on $K$.

(iii) Let $\beta>0$ be the minimum value of $s-\gamma$ on $\overline{D_\epsilon}$, and choose $N$ large enough so that $|\gamma-\gamma_n|, |s-s_n| < \beta/2$ on $\overline{D}$ for all $n \geq N$.  Then
	\[ s_n - \gamma_n > s- \gamma - \beta > 0 \]
on $D_\epsilon$, hence $D_\epsilon \subset D_{(n)}$ for all $n \geq N$.
\end{proof}

According to the following lemma, a nonnegative function on $\Z^d$ with bounded Laplacian can grow at most quadratically.  We will use this fact repeatedly.

\begin{lemma}
\label{atmostquadratic}
Fix $0<\beta<1$.  There is a constant $c_\beta$ such that any nonnegative function $f$ on $\Z^d$ with $f(o)=0$ and $|\Delta f| \leq \lambda$ in $B(o,R)$ satisfies
	\[ f(x) \leq c_\beta \lambda |x|^2, \qquad  x \in B(o,\beta R). \]
\end{lemma}

\begin{proof}
Consider the functions
	\begin{equation} \label{quadraticcorrection} f_\pm(y) = f(y) \pm \lambda |y|^2. \end{equation}
By hypothesis, $f_+$ is subharmonic and $f_-$ is superharmonic on $B(o,R)$.  Given $x \in B(o,\beta R)$, let $r =$ min$(R,2|x|)$, and let $h_\pm$ be the harmonic function on the ball $B=B(o,r-1)$ which agrees with $f_\pm$ on $\partial B$.  Then 
	 \begin{equation} \label{f_+upperbound} f_+ \leq h_+ \leq h_- + 2 \lambda r^2 \leq f_- + 2\lambda r^2. \end{equation}
By the Harnack inequality \cite[Theorem 1.7.2]{Lawler} there is a constant $\tilde{c}_\beta$ such that
	\[ h_+(x) \leq \tilde{c}_\beta h_+(o). \]
Since $f_-(o)=f(o)=0$ we have from (\ref{f_+upperbound})
	\[ h_+(o) \leq f_-(o) + 2 \lambda r^2 \leq 8 \lambda |x|^2. \]
hence 
	\[ f(x) \leq h_+(x) - \lambda |x|^2 \leq c_\beta \lambda |x|^2 \] 
with $c_\beta = 8 \tilde{c}_\beta -1$.
\end{proof}

The following continuous version of Lemma~\ref{atmostquadratic} is proved in the same way, using the continuous Harnack inequality in place of the discrete one, and replacing $|y|^2$ by $|y|^2/2d$ in (\ref{quadraticcorrection}).

\begin{lemma}
\label{continuumatmostquadratic}
Fix $0<\beta<1$.  There is a constant $c'_\beta$ such that the following holds.  Let $f$ be a nonnegative function on $\R^d$ with $f(o)=0$, and let
	\[ f_\pm(x) = f(x) \pm \lambda \frac{|x|^2}{2d}. \]
If $f_+$ is subharmonic and $f_-$ is superharmonic in $B(o,R)$, then
	\[ f(x) \leq c'_\beta \lambda |x|^2, \qquad x \in B(o,\beta R). \]
\end{lemma}

In the following lemma, $A_\epsilon$ denotes the inner $\epsilon$-neighborhood of $A$, as defined by (\ref{innerepsilonneighborhood}).

\begin{lemma}
\label{pointset}
Let $A,B \subset \R^d$ be bounded open sets.  For any $\epsilon>0$ there exists $\eta>0$ with
	\[ (A \cup B)_\epsilon \subset A_\eta \cup B_\eta. \]
\end{lemma}

\begin{proof}
Let $K$ be the closure of $(A \cup B)_\epsilon$, and let $Y_n = A_{1/n} \cup B_{1/n}$.  Since $K$ is contained in $\bigcup Y_n = A\cup B$, the sets $K\cap Y_n$ form an open cover of $K$, which has a finite subcover, i.e.\ $K\subset Y_n$ for some $n$.
\end{proof}

\subsection{Discrete Potential Theory}
\label{discretepotentialtheory}

\begin{table}
\label{boxandpoints}
\centering
\begin{tabular}{r | ll}
~ & $\delta_n\Z^d$ & $\R^d$ \\
\hline 
points & $x^\Points = \Big(x+\left( -\frac{\delta_n}{2}, \frac{\delta_n}{2} \right]^d \Big) \cap \delta_n \Z^d$ & $x^\Box = x + \left[ -\frac{\delta_n}{2}, \frac{\delta_n}{2} \right]^d$ \\ [0.5ex]
sets & $A^\Points = A \cap \delta_n\Z^d$ & $A^\Box = A + \left[ \frac{-\delta_n}{2}, \frac{\delta_n}{2} \right]^d$ \\ [0.5ex]
functions & $f^\Points = f|_{\delta_n \Z^d}$ & $f^\Box(x) = f(x^\Points)$ \\ [0.5ex]
\end{tabular}
\caption{Notation for transitioning between Euclidean space and the lattice.}
\index{$x^\Points$, closest lattice point}
\index{$A^\Points = A \cap \delta_n \Z^d$}
\index{$f^\Points$, restriction to $\delta_n \Z^d$}
\index{$x^\Box$, box of side $\delta_n$ centered at $x$}
\index{$A^\Box$, union of boxes}
\index{$f^\Box$, step function}
\end{table}

Fix a sequence $\delta_n \downarrow 0$ with $\delta_1 =1$.  In this section we relate discrete superharmonic potentials in the lattice $\delta_n \Z^d$ to their continuous couterparts in $\R^d$.
If $A$ is a domain in $\delta_n \Z^d$, write	
	\[ A^\Box = A + \left[ -\frac{\delta_n}{2}, \frac{\delta_n}{2} \right]^d \]
for the corresponding union of cubes in $\R^d$.    If $A$ is a domain in $\R^d$, write $A^\Points = A \cap \delta_n \Z^d$.  Given $x \in \R^d$, write
	\[ x^\Points = \left(x - \frac{\delta_n}{2}, x+ \frac{\delta_n}{2} \right]^d \cap \delta_n \Z^d \]
for the closest lattice point to $x$, breaking ties to the right.  For a function $f$ on $\delta_n \Z^d$, write
	\[ f^\Box (x) = f(x^\Points) \]
for the corresponding step function on $\R^d$.  Likewise, for a function $f$ on $\R^d$, write $f^\Points = f|_{\delta_n \Z^d}$.  These notations are summarized in Table~1.

We define the discrete Laplacian of a function $f$ on $\delta_n \Z^d$ to be
	\[ \Delta f(x) = \delta_n^{-2} \left( \frac{1}{2d} \sum_{y\sim x} f(y)-f(x) \right). \]
According to the following lemma, if $f$ is sufficiently smooth, then its discrete Laplacian on $\delta_n \Z^d$ approximates its Laplacian on $\R^d$.

\begin{lemma} \label{thirdderiv}
If $f$ has continuous third derivative in a $\delta_n$-neighborhood of $x \in \delta_n \Z^d$, and $A$ is the maximum pure third partial of $f$ in this neighborhood, then
	\[ |\Delta f (x) - 2d \Delta f^\Points (x)| \leq \frac{d}{3} A \delta_n. \]
\end{lemma}

\begin{proof}
By Taylor's theorem with remainder
	\[ f(x+\delta_n e_i) - 2f(x) + f(x-\delta_n e_i) = \frac{\partial^2f}{\partial x_i^2}\delta_n^2 + \frac16 \frac{\partial^3 f}{\partial x_i^3}(x+te_i) \delta_n^3 - \frac16 \frac{\partial^3 f}{\partial x_i^3}(x-ue_i) \delta_n^3 \]
for some $0\leq t,u\leq 1$.  Summing over $i=1,\ldots,d$ and dividing by $\delta_n^2$ gives the result.
\end{proof}

In three and higher dimensions, for $x,y \in \delta_n\Z^d$ write 
	\begin{equation} \label{definitionofg_n} g_n(x,y) = \delta_n^{2-d} g_1(\delta_n^{-1}x,\delta_n^{-1}y), \qquad d\geq 3, \index{$g_n(x,y)$, Green's function on $\delta_n \Z^d$} \end{equation}
where 
	\[ g_1(x,y) = \EE_x \# \{k|X_k=y\} \] 
is the Green's function for simple random walk on $\Z^d$.  The scaling in (\ref{definitionofg_n}) is chosen so that $\Delta_x g_n(x,y) = -\delta_n^{-d} 1_{\{x=y\}}$.  In two dimensions, write
	\begin{equation} \label{definitionofg_ndimension2} g_n(x,y) = -a(\delta_n^{-1}x,\delta_n^{-1}y) + \frac{2}{\pi} \log \delta_n, \qquad d=2, \end{equation}
where
	\[ a(x,y) = \lim_{m \to \infty} \big( \EE_x \# \{k\leq m | X_k=x\} - \EE_x \# \{k\leq m | X_k=y \} \big) \] 
is the recurrent potential kernel on $\Z^2$.

\begin{lemma}
\label{greensfunctionconvergence}
In all dimensions $d\geq 2$,
 	 \[ g_n(x,y) = g(x,y) + O(\delta_n^2 |x-y|^{-d}) \]
where $g$ is given by (\ref{greenskernel}). 
\end{lemma}

\begin{proof}
In dimensions $d\geq 3$ we have from (\ref{definitionofg_n}) and the standard estimate for the discrete Green's function (\ref{standardgreensestimate})
	\begin{align*} g_n(x,y) &= \delta_n^{2-d} \left( a_d \left(\frac{|x-y|}{\delta_n}\right)^{2-d} + O\left(\frac{|x-y|}{\delta_n}\right)^{-d}\right) \\
		&= a_d |x-y|^{2-d} + O(\delta_n^2 |x-y|^{-d}). \end{align*}
Likewise, in dimension two, using the standard estimate for the potential kernel (\ref{standardgreensestimate}) in (\ref{definitionofg_ndimension2}) gives
	\[ g_n(x,y) =  -\frac{2}{\pi} \log \frac{|x-y|}{\delta_n} + \frac{2}{\pi} \log \delta_n + O\left(\frac{|x-y|}{\delta_n}\right)^{-2}. \qed \]
\renewcommand{\qedsymbol}{}
\end{proof}

Our next result adapts Lemma~\ref{threesteps}(i) to the discrete setting.
We list here our standing assumptions on the starting densities.  Let $\sigma$ be a function on $\R^d$ with compact support, such that
 	\begin{equation} \label{absolutebound} 0 \leq \sigma \leq M \index{$M$, uniform bound on $\sigma$} \end{equation}
for some absolute constant $M$.  Suppose that $\sigma$ satisfies
	\begin{equation} \label{discontmeasurezero}  \Leb \left(DC(\sigma)\right)=0 \end{equation}
where $DC(\sigma)$ denotes the set of points in $\R^d$ where $\sigma$ is discontinuous.
\index{$DC(\sigma)$, discontinuities of $\sigma$}

For $n=1,2,\ldots$ let $\sigma_n$ be a function on $\delta_n \Z^d$ satisfying
	\begin{equation} \label{discreteabsolutebound} 0 \leq \sigma_n \leq M \index{$\sigma_n$, mass density on $\delta_n \Z^d$} \end{equation}
and
\begin{equation} \label{convergingdensities} \sigma_n^\Box(x)\rightarrow \sigma(x), 
	\qquad x\notin DC(\sigma). \end{equation}
Finally, suppose that
	\begin{equation} \label{uniformcompactsupport} \text{There is a ball $B\subset \R^d$ containing the supports of $\sigma$ and $\sigma_n^\Box$ for all $n$.} \end{equation}
Although for maximum generality we have chosen to state our hypotheses on $\sigma$ and $\sigma_n$ separately, we remark that the above hypotheses on $\sigma_n$ are satisfied in the particular case when $\sigma_n$ is given by averaging $\sigma$ over a small box:
	\begin{equation} \label{averageinabox} \sigma_n(x) = \delta_n^{-d} \int_{x^\Box} \sigma(y)dy. \end{equation}

In parallel to (\ref{thepotential}), for $x \in \delta_n \Z^d$ write
	\begin{equation} \label{thediscretepotential} G_n \sigma_n (x) = \delta_n^{d} \sum_{y \in \delta_n \Z^d} g_n(x,y) \sigma_n(y), \index{$G_n \sigma_n$, discrete potential} \end{equation}
where $g_n$ is given by (\ref{definitionofg_n}) and (\ref{definitionofg_ndimension2}).

\begin{lemma}
\label{greensintegralconvergencecont}
If $\sigma$, $\sigma_n$ satisfy (\ref{absolutebound})-(\ref{uniformcompactsupport}), then
	\[ (G_n \sigma_n)^\Box \rightarrow G \sigma \]
uniformly on compact subsets of $\R^d$. 
\end{lemma}


\begin{proof}
Let $K \subset \R^d$ be compact.  By the triangle inequality
	\[ |(G_n\sigma_n)^\Box - G\sigma| \leq |(G_n \sigma_n)^\Box - G \sigma_n^\Box| + |G\sigma_n^\Box - G\sigma|. \]
By Lemma~\ref{threesteps}(i) the second term on the right side is $<\epsilon/2$ on $K$ for sufficiently large $n$.  The first term is at most
	\begin{align} |(G_n \sigma_n)^\Box(x) - G \sigma_n^\Box(x)| 
	&\leq \sum_{y \in \delta_n\Z^d} \sigma_n(y) \left| \delta_n^d g_n(x^\Points,y) - \int_{y^\Box} g(x,z) dz \right| \nonumber \\
	&\leq M \sum_{y \in B^\Points} \int_{y^\Box} |g_n(x^\Points,y) - g(x,z)| dz \nonumber \\
	&= M \int_{B^{\Points\Box}} |g_n(x^\Points,z^\Points) - g(x,z)| dz. \label{gotitwherewewantit} \end{align}
By Lemma~\ref{greensfunctionconvergence}, we have
	\begin{align*} |g_n(x^\Points,z^\Points) - g(x,z)| &\leq 
		|g_n(x^\Points,z^\Points)-g(x^\Points,z^\Points)| + |g(x^\Points,z^\Points)-g(x,z)| \\
			&\leq C\delta_n|x-z|^{1-d} \end{align*}
for a constant $C$ depending only on $d$.  Integrating (\ref{gotitwherewewantit}) in spherical shells about $x$, we obtain
		\[ |(G_n \sigma_n)^\Box - G \sigma_n^\Box| \leq Cd\omega_dMR\delta_n, \]
where $R$ is the radius of $B$.  Taking $n$ large enough so that the right side is $<\epsilon/2$, the proof is complete.
\end{proof}

Our next result adapts Lemma~\ref{majorantonacompactset} to the discrete setting.  Let $\sigma_n$ be a function on $\delta_n \Z^d$ with finite support, and let $\gamma_n$ be the function on $\delta_n \Z^d$ defined by
	\begin{equation} \label{thediscreteobstacle} \gamma_n(x) = -|x|^2 - G_n \sigma_n(x). \index{$\gamma_n$, obstacle on $\delta_n \Z^d$} \end{equation}
Let 
	\begin{equation} \label{thediscretemajorant} s_n = \inf \{f(x)|\text{$f$ is superharmonic on $\delta_n \Z^d$ and $f \geq \gamma_n$}\} \index{$s_n$, superharmonic majorant in~$\delta_n \Z^d$} \end{equation}
be the discrete least superharmonic majorant of $\gamma_n$, and let
	\begin{equation} \label{thediscretenoncoincidenceset} D_n = \{x \in \delta_n \Z^d| s_n(x)>\gamma_n(x)\}. \end{equation}

\begin{lemma}
\label{discretemajorantonacompactset}
Let $\gamma_n,s_n,D_n$ be given by (\ref{thediscreteobstacle})-(\ref{thediscretenoncoincidenceset}).  If $\Omega \subset \delta_n \Z^d$ satisfies $D_n \subset \Omega$, then 
	\begin{equation} \label{discretemajorantinadomain} s_n(x) = \inf \{f(x)|\text{\em $f$ is superharmonic on $\Omega$ and $f \geq \gamma_n$}\}. \end{equation}
\end{lemma}

\begin{proof}
Let $f$ be any function which is superharmonic on $\Omega$ and $\geq \gamma_n$.  Since $s_n$ is harmonic on $D_n$, so $f-s_n$ is superharmonic on $D_n$ and attains its minimum in $D_n \cup \partial D_n$ on the boundary.  Hence $f-s_n$ attains its minimum in $\Omega \cup \partial \Omega$ at a point $x$ where $s_n(x)=\gamma_n(x)$.  Since $f \geq \gamma_n$ we conclude that $f \geq s_n$ on $\Omega$ and hence everywhere.  Thus $s_n$ is at most the infimum in (\ref{discretemajorantinadomain}).
Since the infimum in (\ref{discretemajorantinadomain}) is taken over a strictly larger set than that in (\ref{thediscretemajorant}), the reverse inequality is trivial.
\end{proof}

\section{Divisible Sandpile}
\label{divsandscalinglimit}

\subsection{Convergence of Odometers}
	
By the odometer function for the divisible sandpile on $\delta_n \Z^d$ with source density $\sigma_n$, we will mean the function
	\[ u_n(x) = \delta_n^2 \cdot \text{total mass emitted from $x$ if each site $y$ starts with mass } \sigma_n(y). \index{$u_n$, divisible sandpile odometer} \]

\begin{theorem}
\label{odomconvergence}
Let $u_n$ be the odometer function for the divisible sandpile on $\delta_n \Z^d$ with source density $\sigma_n$.  If $\sigma$, $\sigma_n$ satisfy (\ref{absolutebound})-(\ref{uniformcompactsupport}), then 		\[ u_n^\Box \to s-\gamma \qquad \text{\em uniformly,} \]
where
	\[ \gamma(x) = -|x|^2 - G \sigma(x), \]
$G\sigma$ is given by (\ref{thepotential}), and $s$ is the least superharmonic majorant of $\gamma$.
\end{theorem}

\begin{lemma}
\label{bigballs}
Let $D_n$ be the set of fully occupied sites for the divisible sandpile in $\delta_n \Z^d$ started from source density $\sigma_n$.  There is a ball $\Omega \subset \R^d$ with 
	\[ \bigcup_{n\geq 1} D_n^\Box \cup D \subset \Omega. \]
\end{lemma}

\begin{proof}
Let $A_n$ be the set of fully occupied sites for the divisible sandpile in $\delta_n \Z^d$ started from source density $\tau(x) = M 1_{x \in B}$, where $B$ is given by (\ref{uniformcompactsupport}).  From the abelian property, Lemma~\ref{abelianproperty}, we have $D_n \subset A_n$.  By the inner bound of Theorem~\ref{divsandcircintro} if we start with mass $m=2\delta_n^{-d} \Leb(B)$ at the origin in $\delta_n \Z^d$, the resulting set of fully occupied sites contains $B^\Points$; by the abelian property it follows that if we start with mass $Mm$ at the origin in $\delta_n \Z^d$, the resulting set $\Omega_n$ of fully occupied sites contains $A_n$.  By the outer bound of Theorem~\ref{divsandcircintro}, $\Omega_n$ is contained in a ball $\Omega$ of volume $3M \Leb(B)$.  By Lemma~\ref{occupieddomainisbounded} we can enlarge $\Omega$ if necessary to contain $D$.
\end{proof}

For $x \in \delta_n \Z^d$ write
	\[ \gamma_n(x) = -|x|^2-G_n\sigma_n(x), \]
where $G_n$ is defined by (\ref{thediscretepotential}).  Denote by $s_n$ the least superharmonic majorant of $\gamma_n$ in the lattice $\delta_n \Z^d$.

\begin{lemma}
\label{gamma_nupperbound}
Let $\Omega$ be as in Lemma~\ref{bigballs}.  There is a constant $M'$ independent of $n$, such that $|\gamma_n| \leq M'$ in $\Omega^\Points$.
\end{lemma}

\begin{proof}
By Lemma~\ref{greensfunctionconvergence} we have for $x \in \Omega^\Points$
	\begin{align*} |G_n\sigma_n(x)| &\leq \delta_n^d \sum_{y \in B^\Points} |g_n(x,y)| \sigma_n(y) \\
			&\leq 2M \delta_n^d \sum_{y \in B^\Points} |g(x,y)| \\
			&\leq CM R^2 \log R \end{align*}
for a constant $C$ depending only on $d$, where $R$ is the radius of $\Omega$.  It follows that $|\gamma_n| \leq (CM+1)R^2 \log R$ in $\Omega^\Points$.  
\end{proof}

\begin{lemma}
\label{alphabounds}
Fix $x\in \R^d$, and for $y \in \delta_n\Z^d$ let 
	\begin{equation} \label{alphadef} \alpha(y) = \delta_n^d g_n(x^\Points,y) - \int_{y^\Box} g(x,z) dz. \end{equation}
There are constants $C_1, C_2$ depending only on $d$, such that
	\begin{enumerate}
	\item[{\em (i)}] $|\alpha(y)| \leq C_1 \delta_n^{1+d} |x-y|^{1-d}$.
	\item[{\em (ii)}] If $y_1 \sim y_2$, then $|\alpha(y_1)-\alpha(y_2)| \leq C_2 \delta_n^{2+d} |x-y_1|^{-d}$.
	\end{enumerate}
\end{lemma}

\begin{proof}
(i) By Lemma~\ref{greensfunctionconvergence}, if $z\in y^\Box$ then
	\begin{align*} | g_n(x^\Points,y) - g(x,z) | &\leq 
|g_n(x^\Points,y)-g(x^\Points,y)| + |g(x^\Points,y)-g(x,z)| \\
	&\leq C \delta_n |x-y|^{1-d} \end{align*}
for a constant $C$ depending only on $d$.  Integrating over $y^\Box$ gives the result.
	
(ii) Let $p=y_2-y_1$.  By Lemma~\ref{greensfunctionconvergence}, we have
	\begin{align*} |\alpha(y_1) - \alpha(y_2)| &\leq \int_{y_1^\Box} |g_n(x^\Points,y_1)-g_n(x^\Points,y_2)-g(x,z)+g(x,z+p)| \,dz \\
	&\leq \int_{y_1^\Box} |g(x^\Points,y_1)-g(x^\Points,y_2)-g(x,z)+g(x,z+p)| \,dz \;+ \\ 
		&\qquad\qquad\qquad\qquad\qquad + O(\delta_n^{d+2} |x-y_1|^{-d}).  \end{align*}
Writing $q=x^\Points - x - y_1+z$, the quantity inside the integral can be expressed as
	\[ |f(o)-f(p)-f(q)+f(p+q)|, \]
where
	\[ f(w) = g(w,x^\Points-y_1). \]
Since $|p|,|q|,|p+q| \leq \delta_n (\sqrt{d}+1)$, we have by Taylor's theorem with remainder
	\begin{align*} |f(o)-f(p)-f(q)+f(p+q)| &\leq 3d\delta_n^2 \sum_{i,j=1}^{d} \left| \frac{\partial^2 f}{\partial x_i \partial x_j} (o) \right| \\
		& \leq \begin{cases} 6d^2 (d-1)(d-2) a_d \delta_n^2 |x-y_1|^{-d}, &d\geq 3 \\
						(24/\pi) \delta_n^2 |x-y_1|^{-2}, &d=2. \qed \end{cases}
		\end{align*}
\renewcommand{\qedsymbol}{}
\end{proof}

\begin{lemma}
\label{truncatedlaplacian}
Let $\Omega$ be an open ball as in Lemma~\ref{bigballs}, and let $\Omega_1$ be a ball with $\bar{\Omega} \subset \Omega_1$.  Let
	\[ \phi_n = -\Delta(s_n1_{\Omega_1^\Points}). \]
Then
	\[ \left| s_n^\Box - G\phi_n^\Box \right| \rightarrow 0 \]
uniformly on $\Omega$.
\end{lemma}

\begin{proof}
Let $\nu_n(x)$ be the amount of mass present at $x$ in the final state of the divisible sandpile on $\delta_n \Z^d$.  By Lemma~\ref{discretemajorant} we have $s_n = u_n + \gamma_n$, hence
	\[ \Delta s_n = (\nu_n - \sigma_n) + (\sigma_n - 1) = \nu_n -1. \]
In particular, $|\Delta s_n| \leq 1$.

Since
	\[ G\phi_n^\Box(x) = \sum_{y \in \Omega_1^\Points \cup \partial \Omega_1^\Points} \phi_n(y) \int_{y^\Box} g(x,z) dz \]
and $s_n = G_n \phi_n$ in $\Omega^\Points$, we have for $x\in \Omega$
	\begin{align*} G\phi_n^\Box(x) - s_n^\Box(x) &= G\phi_n^\Box(x) -  G_n\phi_n(x^\Points) \\
			&= - \sum_{y \in \Omega_1^\Points \cup \partial \Omega_1^\Points} \phi_n(y) \alpha(y) \end{align*}
where $\alpha(y)$ is given by (\ref{alphadef}).  Hence
	\begin{equation} \label{interiorandboundaryterms} G\phi_n^\Box(x) - s_n^\Box(x) = \sum_{y \in \Omega_1^\Points} \Delta s_n(y) \alpha(y) + \delta_n^{-2} \sum_{\substack{y \in \Omega_1^\Points, z \notin \Omega_1^\Points \\ y\sim z}} (s_n(y)\alpha(z) - s_n(z)\alpha(y)). \end{equation}

   Let $R,R_1$ be the radii of $\Omega,\Omega_1$.  By Lemma~\ref{alphabounds}(i), summing in spherical shells about $x$, the first sum in (\ref{interiorandboundaryterms}) is bounded in absolute value by 
	\[ \sum_{y \in \Omega_1^\Points} |\alpha(y)| \leq \int_0^{2R_1} \frac{d \omega_d r^{d-1}}{\delta_n^d} (C_1 \delta_n^{1+d} r^{1-d}) dr = C_1 d \omega_d R_1\delta_n. \]
To bound the second sum in (\ref{interiorandboundaryterms}), note that $s_n=\gamma_n$ outside $\Omega$, so
	\begin{equation} \label{twogradients} |s_n(y)\alpha(z) - s_n(z)\alpha(y)| \leq |\gamma_n(y)| |\alpha(y)-\alpha(z)| + |\alpha(y)| |\gamma_n(y)-\gamma_n(z)|. \end{equation}
By Lemmas~\ref{gamma_nupperbound} and~\ref{alphabounds}(ii), the first term is bounded by
	\[ |\gamma_n(y)| |\alpha(y)-\alpha(z)| \leq C_2 M'\delta_n^{2+d} |x-y|^{-d}. \]
Fix $\epsilon>0$, and let $\Omega_2$ be a ball with $\bar{\Omega}_1 \subset \Omega_2$.  Since $\gamma$ is uniformly continuous on $\Omega_2$, and $\gamma_n \rightarrow \gamma$ uniformly on $\Omega_2$ by Lemma~\ref{greensintegralconvergencecont}, for sufficiently large $n$ we have
	\[ |\gamma_n(y)-\gamma_n(z)| \leq |\gamma_n(y)-\gamma(y)| + |\gamma(y)-\gamma(z)| + |\gamma(z)-\gamma_n(z)| \leq \epsilon. \]
Thus by Lemma~\ref{alphabounds}(i) the second term in (\ref{twogradients}) is bounded by
	\[ |\alpha(y)| |\gamma_n(y)-\gamma_n(z)| \leq C_1 \delta_n^{1+d} |x-y|^{1-d} \epsilon. \]
Since $x\in \Omega$ and $y$ is adjacent to $\partial \Omega_1$, we have $|x-y| \geq R_1-R-\delta_n$, so the second term in (\ref{interiorandboundaryterms}) is bounded in absolute value by 
	\[ 2\delta_n^{-2} \# \partial \Omega_1^\Points \left(C_2 M' \delta_n^{2+d} (R_1-R)^{-d} + C_1 \delta_n^{1+d} (R_1-R)^{1-d} \epsilon \right) \leq C_3 \epsilon \]
for sufficiently large $n$.
\end{proof}

\begin{lemma}
\label{majorantconvergence}
$s_n^\Box \rightarrow s$ uniformly on compact subsets of $\R^d$.
\end{lemma}

\begin{proof}
By Lemma~\ref{bigballs} there is a ball $\Omega$ containing $D$ and $D_n^\Box$ for all $n$.  Outside $\Omega$ we have 
 	\[ s_n^\Box = \gamma_n^\Box \rightarrow \gamma = s \]
uniformly on compact sets by Lemma~\ref{greensintegralconvergencecont}.  

To show convergence in $\Omega$, write
	\begin{equation} \label{mollified} \tilde{s}(x) = \int_{\R^d} s(y) \lambda^{-d} \eta \left(\frac{x-y}{\lambda}\right) dy, \end{equation}
where $\eta$ is the standard smooth mollifier
	\[ \eta(x) = \begin{cases} Ce^{1/(|x|^2-1)}, &|x|<1 \\
					0, &|x|\geq 1  \end{cases} \]
normalized so that $\int_{\R^d} \eta \, dx = 1$ (see \cite[Appendix C.4]{Evans}).  Then $\tilde{s}$ is smooth and superharmonic.  Fix $\epsilon>0$.  By Lemma~\ref{majorantbasicprops}(ii) and compactness, $s$ is uniformly continuous on $\bar{\Omega}$, so taking $\lambda$ sufficiently small in (\ref{mollified}) we have $|s-\tilde{s}|<\epsilon$ in $\Omega$.   Let $A_\epsilon$ be the maximum third partial of $\tilde{s}$ in $\Omega$.  By Lemma~\ref{thirdderiv} the function
	\[ q_n(x) = \tilde{s}^\Points(x) - \frac16 A_\epsilon \delta_n |x|^2 \]
is superharmonic in $\Omega^\Points$.  By Lemma~\ref{greensintegralconvergencecont} we have $\gamma_n^\Box \rightarrow \gamma$ uniformly in $\Omega$.  Taking $N$ large enough so that $\frac16 A_\epsilon \delta_n |x|^2 < \epsilon$ in $\Omega$ and $|\gamma_n-\gamma^\Points|<\epsilon$ in $\Omega^\Points$ for all $n>N$, we obtain
	\[ q_n > \tilde{s}^\Points - \epsilon > s^\Points - 2\epsilon \geq \gamma^\Points - 2\epsilon > \gamma_n - 3\epsilon \]
in $\Omega^\Points$.  In particular, the function $f_n = \text{max}(q_n+3\epsilon,\gamma_n)$ is superharmonic in $\Omega^\Points$.  By Lemma~\ref{discretemajorantonacompactset} it follows that $f_n \geq s_n$, hence
	\[ s_n \leq q_n + 3\epsilon < \tilde{s}^\Points + 3\epsilon < s^\Points + 4\epsilon \]
in $\Omega^\Points$.  By the uniform continuity of $s$ on $\bar{\Omega}$, taking $N$ larger if necessary we have $|s-s^{\Points\Box}|<\epsilon$ in $\Omega$, and hence $s_n^\Box < s + 5\epsilon$ in $\Omega$ for all $n>N$.

For the reverse inequality, let
	\[ \phi_n = - \Delta(s_n 1_{\Omega_1^\Points}) \] 
where $\Omega_1$ is an open ball containing $\bar{\Omega}$.  By Lemma~\ref{truncatedlaplacian} we have
	\[ \left| s_n^\Box - G\phi_n^\Box \right| < \epsilon \]
and hence
	\begin{equation} \label{abovegamma} G\phi_n^\Box > \gamma_n^\Box - \epsilon > \gamma - 2\epsilon. \end{equation}
on $\Omega$ for sufficiently large $n$.  Since $\phi_n^\Box$ is nonnegative on $\Omega$, by Lemma~\ref{superharmonicpotential} the function $G\phi_n^\Box$ is superharmonic on $\Omega$, so by (\ref{abovegamma}) the function 
	\[ \psi_n = \text{max}(G\phi_n^\Box+2\epsilon,\gamma) \]
is superharmonic on $\Omega$ for sufficiently large $n$.  By Lemma~\ref{majorantonacompactset} it follows that $\psi_n \geq s$, hence
	\[ s_n^\Box > G\phi_n^\Box - \epsilon \geq s-3\epsilon \]
on $\Omega$ for sufficiently large $n$.
\end{proof}

\begin{proof}[Proof of Theorem~\ref{odomconvergence}]
Let $\Omega$ be as in Lemma~\ref{bigballs}.  By Lemmas~\ref{greensintegralconvergencecont} and~\ref{majorantconvergence} we have $\gamma_n^\Box \to \gamma$ and $s_n^\Box \to s$ uniformly on $\Omega$.  By Lemma~\ref{discretemajorant},
we have $u_n = s_n-\gamma_n$.   Since $u_n^\Box = 0 = s-\gamma$ off $\Omega$, we conclude that $u_n^\Box \rightarrow s-\gamma$ uniformly.
\end{proof}

\subsection{Convergence of Domains}

In addition to the conditions (\ref{absolutebound})-(\ref{uniformcompactsupport}) assumed in the previous section, we assume in this section that the source density $\sigma$ satisfies 
	\begin{equation} \label{boundedawayfrom1} 
	\text{For all }x\in \R^d \text{ either }\sigma(x) \leq \lambda \text{ or } \sigma(x)\geq 1 \index{$\lambda$, bounds $\sigma$ away from $1$}
	\end{equation}
for a constant $\lambda<1$.  We also assume that
	\begin{equation} \label{closureofinterior}
	\{\sigma \geq 1\} = \overline{\{\sigma \geq 1\}^o}. 
	\end{equation}
Moreover, we assume that for any $\epsilon>0$ there exists $N(\epsilon)$ such that
	\begin{equation} \label{sillytechnicality}
	\text{If }x \in \{\sigma \geq 1\}_\epsilon, \text{ then } \sigma_n(x) \geq 1 \text{ for all }n \geq N(\epsilon);
	\end{equation}
and
	\begin{equation} \label{discretetechnicality}
	\text{If }x \notin \{\sigma \geq 1\}^\epsilon, \text{ then } \sigma_n(x) \leq \lambda \text{ for all }n \geq N(\epsilon).
	\end{equation}
As before, we have chosen to state the hypotheses on $\sigma$ and $\sigma_n$ separately for maximum generality, but all hypotheses on $\sigma_n$ are satisfied in the particular case when $\sigma_n$ is given by averaging $\sigma$ in a small box (\ref{averageinabox}).

We set
	\[ \gamma(x) = -|x|^2-G\sigma(x) \]
with $s$ the least superharmonic majorant of $\gamma$ and 
	\[ D = \{x \in\R^d | s(x)>\gamma(x)\}. \]
We also write
	\[ \widetilde{D} = D \cup \{x \in \R^d | \sigma(x)\geq 1 \}^o. \index{$\widetilde{D}$, enlarged noncoincidence set} \]
For a domain $A \subset \R^d$, denote by $A_\epsilon$ and $A^\epsilon$ its inner and outer open $\epsilon$-neighborhoods, respectively.  

\begin{theorem}
\label{domainconvergence}
Let $\sigma$ and $\sigma_n$ satisfy (\ref{absolutebound})-(\ref{uniformcompactsupport}) and (\ref{boundedawayfrom1})-(\ref{discretetechnicality}).  For $n\geq 1$ let $D_n$ be the domain of fully occupied sites for the divisible sandpile in $\delta_n \Z^d$ started from source density $\sigma_n$.  For any $\epsilon>0$ we have for large enough $n$
	\begin{equation} \label{domainbounds} \widetilde{D}_\epsilon^\Points \subset D_n \subset \widetilde{D}^{\epsilon\Points}. \end{equation}
\end{theorem}
							
According to the following lemma, near any occupied site $x \in D_n$ lying outside $\widetilde{D}^\epsilon$, we can find a site $y$ where the odometer $u_n$ is relatively large.  This is a discrete version of a standard argument for the obstacle problem; see for example Friedman \cite{Friedman}, Ch.\ 2 Lemma 3.1. 

\begin{lemma}
\label{quadraticgrowth}
Fix $\epsilon>0$ and $x \in D_n$ with $x\notin \widetilde{D}^\epsilon$.  If $n$ is sufficiently large, there is a point $y\in \delta_n\Z^d$ with $|x-y|\leq \frac{\epsilon}{2}+\delta_n$ and 
	\[ u_n(y) \geq u_n(x) + \frac{1-\lambda}{4}\epsilon^2. \]
\end{lemma}

\begin{proof}


By (\ref{closureofinterior}), we have $\{ \sigma \geq 1\} \subset \overline{\widetilde{D}}$.  Thus if $x \notin \widetilde{D}^\epsilon$, the ball $B=B(x,\frac{\epsilon}{2})^\Points$ is disjoint from $\{\sigma \geq 1\}^{\epsilon/2}$.  In particular, if $n \geq N(\frac{\epsilon}{2})$, then by (\ref{discretetechnicality}) we have $\sigma_n \leq \lambda$ on $B$.  Thus the function
	\[ w(y) = u_n(y) - (1-\lambda)|x-y|^2 \]
is subharmonic on $B \cap D_n$, so it attains its maximum on the boundary.  Since $w(x)\geq 0$, the maximum cannot be attained on $\partial D_n$, where $u_n$ vanishes; so it is attained at some point $y \in \partial B$, and
	\begin{align*} u_n(y) &\geq w(y) + (1-\lambda)\left(\frac{\epsilon}{2}\right)^2.
\end{align*}
Since $w(y)\geq w(x)=u_n(x)$, the proof is complete.
\end{proof}

\begin{proof}[Proof of Theorem~\ref{domainconvergence}]
Fix $\epsilon>0$.  By Lemma~\ref{pointset} we have 
	\[ \widetilde{D}_\epsilon \subset D_{\eta} \cup \{\sigma \geq 1\}_{\eta} \] 
for some $\eta>0$.  Let $u=s-\gamma$.  Since the closure of $D_{\eta}$ is compact and contained in $D$, we have $u \geq m_{\eta}$ on $D_{\eta}$ for some $m_{\eta}>0$.  By Theorem~\ref{odomconvergence}, for sufficiently large $n$ we have $u_n > u - \frac12 m_{\eta} > 0$ on $D_{\eta}^\Points$, hence $D_{\eta}^\Points \subset D_n$.  Likewise, by (\ref{sillytechnicality}) we have $\{\sigma \geq 1\}_{\eta}^\Points \subset D_n$ for large enough $n$.  Thus $\widetilde{D}_\epsilon \subset D_n$.

For the other inclusion, fix $x \in \delta_n\Z^d$ with $x \notin \widetilde{D}^\epsilon$.  Since $u$ vanishes on the ball $B=B(x,\frac{\epsilon}{2})$, by Theorem~\ref{odomconvergence} we have $u_n < \frac{1-\lambda}{4} \epsilon^2$ on $B^\Points \cup \partial B^\Points$ for all sufficiently large $n$.  By Lemma~\ref{quadraticgrowth} we conclude that $x \notin D_n$, and hence $D_n \subset \widetilde{D}^{\epsilon\Points}$.
\end{proof}

\section{Rotor-Router Model}
\label{rotorscalinglimit}

In trying to adapt the proofs of Theorems~\ref{odomconvergence} and~\ref{domainconvergence} to the rotor-router model, we are faced with two main problems.  The first is to define an appropriate notion of convergence of integer-valued densities $\sigma_n$ on $\delta_n \Z^d$ to a real-valued density $\sigma$ on $\R^d$.  The requirement that $\sigma_n$ take only integer values is of course imposed on us by the nature of the rotor-router model itself, since unlike the divisible sandpile, the rotor-router works with discrete, indivisible particles.  The second problem is to find an appropriate analogue of Lemma~\ref{discretemajorant} for the rotor-router model.

Although these two problems may seem unrelated, the single technique of \emph{smoothing} neatly takes care of them both.  To illustrate the basic idea, suppose we are given a domain $A \subset \R^d$, and let $\sigma_n$ be the function on $\delta_n \Z^d$ taking the value $1$ on odd lattice points in $A^\Points$, and the value $2$ on even lattice points in $A^\Points$, while vanishing outside $A^\Points$.  We would like to find a sense in which $\sigma_n$ converges to the real-valued density $\sigma = \frac32 1_A$.  One approach is to average $\sigma_n$ in a box whose side length $L_n$ goes to zero more slowly than the lattice spacing: $L_n \downarrow 0$ while $L_n / \delta_n \uparrow \infty$ as $n \uparrow \infty$.  The resulting ``smoothed'' version of $\sigma_n$ converges to $\sigma$ pointwise away from the boundary of $A$.  

By smoothing the odometer function in the same way, we can obtain an approximate analogue of Lemma~\ref{discretemajorant} for the rotor-router model.  Rather than average in a box as described above, however, it is technically more convenient to average according to the distribution of a lazy random walk run for a fixed number $\alpha(n)$ of steps.  Denote by $(X_k)_{k\geq 0}$ the lazy random walk in $\delta_n \Z^d$ which stays in place with 
probability $\frac12$ and moves to each of the $2d$ neighbors with probability 
$\frac{1}{4d}$.  Given a function $f$ on $\delta_n \Z^d$, define its $k$-{\it smoothing}
	\begin{equation} \label{smoothingdef} S_k f (x) = \EE (X_k | X_0=x). \index{$S_k f$, binomial smoothing} \end{equation}
From the Markov property we have $S_k S_\ell = S_{k+\ell}$.  Also, the discrete Laplacian can be written as
	\begin{equation} \label{deltaintermsofsmoothing} \Delta = 2\delta_n^{-2}(S_1 - S_0). \end{equation}
In particular, $\Delta S_k = S_k \Delta$.

\subsection{Convergence of Odometers}

For $n=1,2,\ldots$ let $\sigma_n$ be an integer-valued function on $\delta_n\Z^d$ satisfying $0\leq \sigma_n \leq M$.  We assume as usual that there is a ball $B \subset \R^d$ containing the support of $\sigma_n$ for all $n$.  Let $\sigma$ be a function on $\R^d$ supported in $B$ satisfying (\ref{absolutebound}) and (\ref{discontmeasurezero}).  In place of condition (\ref{convergingdensities}) we assume that there exist integers $\alpha(n) \uparrow \infty$ with $\delta_n \alpha(n) \downarrow 0$ such that \index{$\alpha(n)$, scale of smoothing}
	\begin{equation} \label{smootheddensityconvergence} (S_{\alpha(n)} \sigma_n)^\Box(x) \rightarrow \sigma(x), \qquad x \notin DC(\sigma). \end{equation}
	
By the odometer function for rotor-router aggregation starting from source density $\sigma_n$, we will mean the function
	\[ u_n(x) = \delta_n^2 \cdot \text{number of particles emitted from $x$} \]
if $\sigma_n(y)$ particles start at each site $y$.

\begin{theorem}
\label{rotorodomconvergence}
Let $u_n$ be the odometer function for rotor-router aggregation on $\delta_n\Z^d$ starting from source density $\sigma_n$.  If $\sigma, \sigma_n$ satisfy (\ref{absolutebound})-(\ref{discreteabsolutebound}), (\ref{uniformcompactsupport}) and (\ref{smootheddensityconvergence}),    then $u_n^\Box \rightarrow s-\gamma$ uniformly, where
	\[ \gamma(x) = -|x|^2 - \int_{\R^d} g(x,y) \sigma(y) dy \]
and $s$ is the least superharmonic majorant of $\gamma$.
\end{theorem}	

Given a function $f$ on $\delta_n \Z^d$, for an edge $(x,y)$ write
	\[ \nabla f (x,y) = \frac{f(y)-f(x)}{\delta_n}. \]
Given a function $\kappa$ on edges in $\delta_n \Z^d$, write
	\[ \div \kappa (x) = \frac{1}{2d\delta_n} \sum_{y \sim x} \kappa(x,y). \]
The discrete Laplacian on $\delta_n \Z^d$ is then given by
	\[ \Delta f(x) = \div \nabla f = \delta_n^{-2} \left( \frac{1}{2d} \sum_{y \sim x} f(y)-f(x) \right). \]
The following ``rescaled" version of Lemma~\ref{odomflow} is proved in the same way.

\begin{lemma}
\label{scaledodomflow}
For an edge $(x,y)$ in $\delta_n \Z^d$, denote by $\kappa(x,y)$ the net number of crossings 
from $x$ to $y$ performed by particles during a sequence of rotor-router moves.  Let 
	\[u(x) = \delta_n^2 \cdot \text{\em number of particles emitted from $x$ during this sequence}.\] 
Then
	\begin{equation} \label{scaledgradodom} \nabla u(x,y) = \delta_n (-2d\kappa(x,y) + \rho(x,y)). 
\end{equation}
for some edge function $\rho$ which satisfies
	\[ |\rho(x,y)| \leq 4d-2 \]
independent of $x$, $y$, $n$ and the chosen sequence of moves.
\end{lemma}

\begin{proof}
Writing $N(x,y)$ for the number of particles routed from $x$ to $y$, for any $y,z \sim x$ we have
	\[ |N(x,y)-N(x,z)| \leq 1. \]
Since $u(x) = \delta_n^2 \sum_{y \sim x} N(x,y)$, we obtain
	\[ \delta_n^{-2}u(x)-2d+1 \leq 2d N(x,y) \leq \delta_n^{-2}u(x)+2d-1 \]
hence
	\begin{align*}  | \nabla u(x,y) + 2d\delta_n \kappa(x,y) | &= \delta_n|\delta_n^{-2} u(y) - \delta_n^{-2}u(x) + 2d N(x,y) - 2d N(y,x) | \\
		&\leq (4d - 2)\delta_n. \qed \end{align*}
\renewcommand{\qedsymbol}{} 
\end{proof}

\vspace{-3ex}
Recall that the divisible sandpile odometer function has Laplacian $1-\sigma_n$ inside the set $D_n$ of fully occupied sites.  The next lemma shows that the same is approximately true of the smoothed rotor-router odometer function.  Denote by $R_n$ the occupied shape for rotor-router aggregation on $\delta_n \Z^d$ starting from source density $\sigma_n$.

\begin{lemma}
\label{laplacianofsmoothing}
$|\Delta S_k u_n(x) - \PP_x(X_k \in R_n) + S_k\sigma_n(x)| \leq C_0 / \sqrt{k+1}$.
\end{lemma}

\begin{proof}
Let $\kappa$ and $\rho$ be defined as in Lemma~\ref{scaledodomflow}.  Since each site $x$ starts with $\sigma_n(x)$ particles and ends with either one particle or none accordingly as $x\in R_n$ or $x \notin R_n$, we have 
	\[ 2d\delta_n\, \div \kappa = \sigma_n - 1_{R_n}. \]
Taking the divergence in (\ref{scaledgradodom}) we obtain
	\begin{equation} \Delta u_n = 1_{R_n} - \sigma_n + \delta_n\, \div \rho.   \end{equation}
Using the fact that $S_k$ and $\Delta$ commute, it follows that
	\begin{equation} \label{divergenceerrorterm} \Delta S_k u_n(x) = \PP_x(X_k \in R_n) - \EE_x \sigma_n(X_k) + \delta_n \EE_x \div \rho(X_k). \end{equation}
Since $\nabla u$ and $\kappa$ are antisymmetric, $\rho$ is antisymmetric by (\ref{scaledgradodom}).
Thus the last term in (\ref{divergenceerrorterm}) can be written
	\begin{align} \delta_n \EE_x \div \rho(X_k) &= \frac{1}{2d} \sum_{y \sim z} \PP_x(X_k=y)R(y,z) \nonumber \\
	&= - \frac{1}{2d} \sum_{y \sim z} \PP_x(X_k=z)\rho(y,z) \nonumber \\
	&= \frac{1}{4d} \sum_{y\sim z} (\PP_x(X_k= y) - \PP_x(X_k=z)) \rho(y,z) \label{readyfortirangleineq} \end{align}
where the sums are taken over all pairs of neighboring sites $y,z \in \delta_n\Z^d$.

We can couple lazy random walk $X_k$ started at $x$ with lazy random walk $X'_k$ started at a 
uniform neighbor of $x$ so that the probability of not coupling in $k$ steps is at most $C/\sqrt{k+1}$, where $C$ is a constant depending only on $d$ \cite{Lindvall}.  Since the total variation distance between $X_k$ and $X'_k$ is at most the probability of not coupling, we obtain
	\[ \sum_{y\sim z} |\PP_x(X_k= y) - \PP_x(X_k=z)| \leq \frac{C}{\sqrt{k+1}}. \]
Using the fact that $\rho(y,z)$ is uniformly bounded in (\ref{readyfortirangleineq}), taking $C_0 = 4Cd$ completes the proof.
\end{proof}



The next lemma shows that smoothing the odometer function does not introduce much extra error.

\begin{lemma}
\label{smoothingdist}
$|S_k u_n-u_n| \leq \delta_n^2(\frac12 Mk + C_0\sqrt{k+1})$.
\end{lemma}
 
\begin{proof}
From (\ref{deltaintermsofsmoothing}) we have
	\begin{align*} |S_k u_n-u_n| &\leq \sum_{j=0}^{k-1} |S_{j+1}u_n-S_j u_n| \\
			&= \frac{\delta_n^2}{2} \sum_{j=0}^{k-1} |\Delta S_j u_n|. \end{align*}
But by Lemma~\ref{laplacianofsmoothing}
	\[ |\Delta S_j u_n| \leq M + \frac{C_0}{\sqrt{j+1}}. \]
Summing over $j$ yields the result.
\end{proof}

Let
	\[ \gamma_n(x) = -|x|^2 - G_n S_{\alpha(n)} \sigma_n(x) \]
and let $s_n$ be the least superharmonic majorant of $\gamma_n$.  By Lemma~\ref{discretemajorant}, the difference $s_n-\gamma_n$ is the odometer function for the divisible sandpile  on $\delta_n \Z^d$ starting from the smoothed source density $\widetilde{\sigma}_n = S_{\alpha(n)} \sigma_n$.  Note that by Lemma~\ref{bigballs}, there is a ball $\Omega \subset \R^d$ containing the supports of $s-\gamma$ and of $s_n-\gamma_n$ for all $n$.  The next lemma compares the smoothed rotor-router odometer for the source density $\sigma_n$ with the divisible sandpile odometer for the smoothed density $\widetilde{\sigma}_n$.

\begin{lemma}
\label{smoothedodomlowerbound}
Let $\Omega \subset \R^d$ be a ball centered at the origin containing the supports of $s-\gamma$ and of $s_n-\gamma_n$ for all $n$.  Then 
 	\begin{equation} \label{onesidedodomcomparison} S_{\alpha(n)} u_n \geq s_n - \gamma_n - C_0 r^2 \alpha(n)^{-1/2} \end{equation}
on all of $\delta_n \Z^d$, where $r$ is the radius of $\Omega$.
\end{lemma}

\begin{proof}
 Since
	 \[ \Delta \gamma_n = -1 + S_{\alpha(n)} \sigma_n, \]
by Lemma~\ref{laplacianofsmoothing} the function
	\[ f(x) = S_{\alpha(n)}u_n(x) + \gamma_n(x) + C_0 \alpha(n)^{-1/2}(r^2-|x|^2) \]
is superharmonic on $\delta_n \Z^d$.  Since $f\geq \gamma_n$ on $\Omega^\Points$, the function $\psi_n = \text{max}(f,\gamma_n)$ is superharmonic on $\Omega^\Points$, hence $\psi_n \geq s_n$ by Lemma~\ref{discretemajorantonacompactset}.  Thus $f \geq s_n$ on $\Omega^\Points$, so (\ref{onesidedodomcomparison}) holds on $\Omega^\Points$ and hence everywhere.  
\end{proof}

\begin{lemma}
\label{rotorbigballs}
Let $R_n$ be the set of occupied sites for rotor-router aggregation in $\delta_n \Z^d$ started from source density $\sigma_n$.  There is a ball $\Omega \subset \R^d$ with $\bigcup R_n^\Box \subset \Omega$.
\end{lemma}

\begin{proof}
By assumption there is a ball $B \subset \R^d$ containing the support of $\sigma_n$ for all $n$.
Let $A_n$ be the set of occupied sites for rotor-router aggregation in $\delta_n \Z^d$ started from source density $\tau(x) = M 1_{x \in B}$.  From the abelian property we have $R_n \subset A_n$.  By the inner bound of Theorem~\ref{rotorcircintro}, if we start with $\left\lfloor 2\delta_n^{-d} \Leb(B) \right\rfloor$ particles at the origin in $\delta_n \Z^d$, the resulting set of occupied sites contains $B^\Points$; by the abelian property it follows that if we start with $M \left\lfloor 2\delta_n^{-d} \Leb(B) \right\rfloor$ particles at the origin, the resulting set $\Omega_n$ of fully occupied sites contains $A_n$.  By the outer bound of Theorem~\ref{rotorcircintro}, $\Omega_n$ is contained in a ball $\Omega$ of volume $3M \Leb(B)$.
\end{proof}

\begin{proof}[Proof of Theorem~\ref{rotorodomconvergence}]
By Lemma~\ref{bigballs} there is a ball $\Omega$ containing the support of $s-\gamma$ and of $s_n-\gamma_n$ for all $n$.  By Lemma~\ref{rotorbigballs} we can enlarge $\Omega$ if necessary to contain the support of $S_{\alpha(n)} u_n$ for all $n$.
By Lemma~\ref{laplacianofsmoothing}, the function
	\[ \phi(x) = S_{\alpha(n)} u_n(x) - s_n(x) + \gamma_n(x) + C_0 \alpha(n)^{-1/2} |x|^2. \]
is subharmonic on the set
	\[ \widetilde{R}_n = \{x \in R_n \,:\, y \in R_n \text{ whenever } ||x-y||_1 \leq \alpha(n) \} \]
since $\PP_x\left(X_{\alpha(n)} \in R_n \right)=1$ for $x \in \widetilde{R}_n$.  From Lemma~\ref{laplacianofsmoothing} we have $\left|\Delta S_{\alpha(n)} u_n \right| \leq M+C_0$, so for $x \notin \widetilde{R}_n$ we have by Lemma~\ref{atmostquadratic}
	\[ S_{\alpha(n)} u_n(x) \leq c (M+C_0) \alpha(n)^2 \delta_n^2. \]
By the maximum principle in $\widetilde{R}_n$
	\[ \phi(x) \leq c(M+C_0)\alpha(n)^2 \delta_n^2 + C_0 \alpha(n)^{-1/2} r^2, \qquad x \in \widetilde{R}_n. \]
From Lemma~\ref{smoothedodomlowerbound} we obtain
	\begin{equation} \label{twosidedbound} -C_0 r^2\alpha(n)^{-1/2} \leq S_{\alpha(n)} u_n - s_n + \gamma_n \leq c(M+C_0)\alpha(n)^2\delta_n^2 + C_0 r^2\alpha(n)^{-1/2} \end{equation}
on all of $\delta_n \Z^d$, where $r$ is the radius of $\Omega$.

By Lemmas~\ref{greensintegralconvergencecont} and~\ref{majorantconvergence} we have $\gamma_n^\Box \to \gamma$ and $s_n^\Box \to s$ uniformly on $\Omega$.  Since $\alpha(n) \uparrow \infty$ and $\delta_n \alpha(n) \downarrow 0$, we conclude from (\ref{twosidedbound}) that $\left(S_{\alpha(n)} u_n\right)^\Box \to s-\gamma$ uniformly on $\Omega$.  Since both $S_{\alpha(n)} u_n$ and $s-\gamma$ vanish outside $\Omega$, this convergence is uniform on $\R^d$.  By Lemma~\ref{smoothingdist} we have 
	\[ \left|S_{\alpha(n)} u_n - u_n \right| \to 0 \]
uniformly on $\delta_n \Z^d$, and hence $u_n^\Box \to s-\gamma$ uniformly on $\R^d$.
\end{proof}

\subsection{Convergence of Domains}
In addition to the assumptions of the previous section, in this section we require that
	\begin{equation} \label{rotorboundedawayfrom1} 
	\text{For all }x\in \R^d \text{ either }\sigma(x) \geq 1 \text{ or } \sigma(x)=0.
	\end{equation}
We also assume that
	\begin{equation} \label{rotorclosureofinterior}
	\{\sigma \geq 1\} = \overline{\{\sigma \geq 1\}^o}. 
	\end{equation}
Moreover, we assume that for any $\epsilon>0$ there exists $N(\epsilon)$ such that
	\begin{equation} \label{rotorsillytechnicality}
	\text{If }x \in \{\sigma \geq 1\}_\epsilon, \text{ then } \sigma_n(x) \geq 1 \text{ for all }n \geq N(\epsilon);
	\end{equation}
and
	\begin{equation} \label{rotordiscretetechnicality}
	\text{If }x \notin \{\sigma \geq 1\}^\epsilon, \text{ then } \sigma_n(x) =0 \text{ for all }n \geq N(\epsilon).
	\end{equation}
	
\begin{theorem}
\label{rotordomainconvergence}
Let $\sigma$ and $\sigma_n$ satisfy (\ref{absolutebound})-(\ref{discreteabsolutebound}), (\ref{uniformcompactsupport}), (\ref{smootheddensityconvergence}) and (\ref{rotorboundedawayfrom1})-(\ref{rotordiscretetechnicality}).  For $n\geq 1$ let $R_n = \{u_n>0\}$ be the domain of sites in $\delta_n \Z^d$ that emit a particle in the rotor-router model started from source density $\sigma_n$.  For any $\epsilon>0$ we have for all sufficiently large $n$
	\begin{equation} \label{rotordomainbounds} \widetilde{D}_\epsilon^\Points \subset R_n \subset \widetilde{D}^{\epsilon\Points}, \end{equation}
where
	\[ \widetilde{D} = \{s>\gamma\} \cup \{\sigma\geq 1\}^o. \]
\end{theorem}

\begin{lemma}
\label{bootstrappinggrowth}
Fix $\epsilon>0$ and $n \geq N(\epsilon/2)$.  Given $x \not\in \widetilde{D}^\epsilon$ and $\delta_n \leq \rho \leq \epsilon/2$, let
	\[ \mathcal{N}_\rho(x) = \# B(x,\rho) \cap R_n. \]
If $u_n \leq \delta_n^2 m$ on $B(x,\rho)$, then
	\[ \mathcal{N}_\rho(x) \geq \frac{m}{m-1} \mathcal{N}_{\rho-\delta_n}(x). \]
\end{lemma}

\begin{proof}
Since at least $\mathcal{N}_\rho(x)$ particles must enter the ball $B(x,\rho)$, we have
	\[ \sum_{y \in \partial B(x,\rho-\delta_n)} u_n(y) \geq \delta_n^2 \mathcal{N}_\rho(x). \]
There are $\mathcal{N}_\rho(x)-\mathcal{N}_{\rho-\delta_n}(x)$ terms in the sum on the left side, and each term is at most $\delta_n^2 m$, hence
	\[ m\big( \mathcal{N}_\rho(x) - \mathcal{N}_{\rho-\delta_n}(x) \big) \geq \mathcal{N}_\rho(x). \qed \]
\renewcommand{\qedsymbol}{}
\end{proof}

The following lemma can be seen as a weak analogue of Lemma~\ref{quadraticgrowth} for the divisible sandpile.  In Lemma~\ref{rotorquadraticgrowth}, below, we obtain a more exact analogue under slightly stronger hypotheses.

\begin{lemma}
\label{slowgrowth}
For any $\epsilon>0$, if $n$ is sufficiently large and $x \in R_n$ with $x \not\in \widetilde{D}^\epsilon$, then there is a point $y \in R_n$ with $|x-y| \leq \epsilon/2$ and 
	\[ u_n(y) \geq \frac{\epsilon \delta_n}{4\log \big(2\omega_d (\epsilon/2\delta_n)^d \big)}. \]
\end{lemma}

\begin{proof}
Take $n \geq N(\epsilon/2)$ large enough so that $4 \delta_n <\epsilon$.  Let $m$ be the maximum value of $\delta_n^{-2} u_n$ on $B(x,\epsilon)$.  Since $x \in R_n$ we have $\mathcal{N}_0(x) = 1$.  Iteratively applying Lemma~\ref{bootstrappinggrowth}, we obtain
	\[ \mathcal{N}_{\epsilon/2}(x) \geq \left( \frac{m}{m-1} \right)^{\floor{\epsilon/2\delta_n}} \mathcal{N}_0(x) \geq \exp \left( \frac{\epsilon}{4m \delta_n} \right). \]
Since 
	\[ \mathcal{N}_{\epsilon/2}(x) \leq \# B(x,\epsilon/2) \leq 2\omega_d \left(\frac{\epsilon}{2\delta_n} \right)^d, \]
we conclude that
	\[ \frac{\epsilon}{4m\delta_n} \leq \log \big( 2\omega_d (\epsilon/\delta_n)^d \big). \]
Solving for $m$ yields the result.
\end{proof}

The following lemma shows that far away from $D$, the rotor-router odometer grows at most quadratically as we move away from the boundary of $R_n$.

\begin{lemma}
For any $\epsilon>0$ and $k\geq 1$, if $n$ is sufficiently large and $z\not\in R_n \cup \widetilde{D}^\epsilon$, then 
	\[ u_n(x) \leq 4cd k^2 \delta_n^2, \qquad x \in B(z,k\delta_n)^\Points, \]
where $c=c_{1/2}$ is the constant in Lemma~\ref{atmostquadratic}.
\end{lemma}

\begin{proof}
Taking the divergence in (\ref{scaledgradodom}), since $-2d\delta_n\,$div$\,\kappa(x)$ is the net number of particles entering $x$, we obtain
	\begin{equation} \label{laplacianofu_n} \Delta u_n = 1_{R_n} - \sigma_n + \delta_n\, \text{div}\, R. \end{equation}
By (\ref{rotordiscretetechnicality}), if $n \geq N(\epsilon/2)$ is large enough so that $k \delta_n < \epsilon/4$, we have $\sigma_n = 0$ on the ball $B = B(z,2k\delta_n)^\Points$.  Since $|R| \leq 4d-2$, we have $|\Delta u_n| \leq 4d$ on $B$, so by Lemma~\ref{atmostquadratic}
	\[ u_n(x) \leq 4cd |x-z|^2 \]
for $x \in B(z,k\delta_n)^\Points$.
\end{proof}

The following lemma is analogous to Lemma~\ref{quadraticgrowth} for the divisible sandpile.

\begin{lemma}
\label{rotorquadraticgrowth}
Fix $\epsilon>0$ and $k \geq 4C_0^2$, where $C_0$ is the constant in Lemma~\ref{laplacianofsmoothing}.  There is a constant $C_1$ such that if $x \in R_n$, $x \not\in \widetilde{D}^\epsilon$ satisfies
	\[ S_k u_n(x) > C_1 k^2 \delta_n^2 \]
and $n$ is sufficiently large, there exists $y \in \delta_n \Z^d$ with $|x-y| \leq \frac{\epsilon}{2} + \delta_n$ and
	\[ S_k u_n(y) \geq S_k u_n(x) + \frac{\epsilon^2}{8}. \]
\end{lemma}

\begin{proof}
Let
	\[ R_n^{(k)} = \{x \in R_n \,:\, y \in R_n \text{ whenever } ||x-y||_1 \leq k\delta_n \}. \]
By Lemma~\ref{laplacianofsmoothing}, we have
	\[ \Delta S_k u_n (y) \geq 1 - S_k \sigma_n (y) - C_0/\sqrt{k}, \qquad y \in R_n^{(k)}. \]
Note that $C_0/\sqrt{k}<\frac12$.  Take $n \geq N(\epsilon/4)$ large enough so that $k \delta_n < \epsilon/4$; then $S_k \sigma_n$ vanishes on the ball $B=B(x,\epsilon/2)^\Points$ by (\ref{rotordiscretetechnicality}).  Thus the function
	\[ f(y) = S_k u_n(y) - \frac12 |x-y|^2 \]
is subharmonic on $\widetilde{R}_n := R_n^{(k)} \cap B$, so it attains its maximum on the boundary.  

If $y \in \partial R_n^{(k)}$, there is a point $z \not\in R_n$ with $|y-z| \leq k \delta_n$.  By Lemma~\ref{laplacianofsmoothing} we have $|\Delta u_n| \leq M+C_0$, hence by Lemma~\ref{atmostquadratic} it follows that
	\[ u_n \leq 4c(M+C_0)k^2 \delta_n^2 \]
in the ball $B(z,2k\delta_n)$.  Taking $C_1 = 4c(M+C_0)$ we find that 
	\[ S_k u_n(y) \leq C_1 k^2 \delta_n^2 < S_k u_n(x). \]
Thus $f$ cannot attain its maximum in $\widetilde{R}_n$ on $\partial R_n^{(k)}$, and hence must attain its maximum at a point $y \in \partial B$.  Since $f(y) \geq f(x)$ we conclude that
	\begin{align*} S_k u_n(y) &= f(y) + \frac12 |x-y|^2 \\ 
		&\geq S_k u_n(x) + \frac12 \left( \frac{\epsilon}{2} \right)^2. \qed \end{align*}
\renewcommand{\qedsymbol}{}
\end{proof}

\begin{proof}[Proof of Theorem~\ref{rotordomainconvergence}]
Fix $\epsilon>0$.  By Lemma~\ref{pointset} we have 
	\[ \widetilde{D}_\epsilon \subset D_{\eta} \cup \{\sigma \geq 1\}_{\eta} \] 
for some $\eta>0$.  Since the closure of $D_{\eta}$ is compact and contained in $D$, we have $u \geq m_{\eta}$ on $D_{\eta}$ for some $m_{\eta}>0$.  By Theorem~\ref{rotorodomconvergence}, for sufficiently large $n$ we have $u_n > u - \frac12 m_{\eta} > 0$ on $D_{\eta}^\Points$, hence $D_{\eta}^\Points \subset R_n$.  Likewise, by (\ref{sillytechnicality}) we have $\{\sigma \geq 1\}_{\eta}^\Points \subset R_n$ for large enough $n$.  Thus $\widetilde{D}_\epsilon^\Points \subset R_n$.

For the other inclusion, fix $x \in \delta_n \Z^d$ with $x \notin \widetilde{D}^\epsilon$.  Since $u$ vanishes on the ball $B=B(x,\epsilon)$, by Theorem~\ref{rotorodomconvergence} we have $u_n < \frac{\epsilon^2}{8}$ on $B^\Points$ for all sufficiently large $n$.  Let $k=4C_0^2$, and take $n$ large enough so that $k\delta_n < \epsilon/4$.  Then $S_k u_n < \frac{\epsilon^2}{8}$ on the ball $B(x,3\epsilon/4)^\Points$.  By Lemma~\ref{rotorquadraticgrowth} it follows that $S_k u_n \leq C_1 k^2 \delta_n^2$ on the smaller ball $B' = B(x,\epsilon/4)^\Points$.  In particular, for $y \in B'$ we have
	\[ u_n(y) \leq \frac{S_k u_n(y)}{\PP_y(X_k=y)} < C_2 k^{d/2+2} \delta_n^2. \]
For sufficiently large $n$ the right side is at most $\epsilon \delta_n / 4\log \big(2\omega_d (\epsilon/2\delta_n)^d \big)$, so we conclude from Lemma~\ref{slowgrowth} that $x \not\in R_n$.
\end{proof}

\section{Internal DLA}
 \label{IDLAscalinglimit}
 
Our hypotheses for convergence of internal DLA domains are the same as those for the rotor-router model: $\sigma$ is a bounded, nonnegative, compactly supported function on $\R^d$ that is continuous almost everywhere, and $\{\sigma_n\}_{n\geq 1}$ is a sequence of uniformly bounded functions on $\delta_n \Z^d$ with uniformly bounded supports, whose ``smoothings'' converge pointwise to $\sigma$ at all continuity points of $\sigma$. 

\begin{theorem}
\label{IDLAconvergence}
Let $\sigma$ and $\sigma_n$ satisfy (\ref{absolutebound})-(\ref{discreteabsolutebound}), (\ref{uniformcompactsupport}), (\ref{smootheddensityconvergence}) and (\ref{rotorboundedawayfrom1})-(\ref{rotordiscretetechnicality}).  For $n\geq 1$ let $I_n$ be the random domain of occupied sites for internal DLA in $\delta_n \Z^d$ started from source density $\sigma_n$.  For all $\epsilon>0$ we have with probability one
	\begin{equation} \label{IDLAdomainbounds} \widetilde{D}_\epsilon^\Points \subset I_n \subset \widetilde{D}^{\epsilon\Points} \qquad \text{\em for all sufficiently large $n$,}  \end{equation}
where
	\[ \widetilde{D} = \{s>\gamma\} \cup \{\sigma\geq 1\}^o. \]
\end{theorem}
            
\subsection{Inner Estimate}

Fix $n \geq 1$, and label the particles in $\delta_n \Z^d$ by the integers $1, \ldots, m_n$, where $m_n = \sum_{x \in \delta_n \Z^d} \sigma_n(x)$.  Let $x_i$ be the starting location of the particle labeled $i$, so that
	\[ \# \{i|x_i=x\} = \sigma_n(x). \]
For each $i=1,\ldots,m_n$ let $(X^i_t)_{t\geq 0}$ be a simple random walk in $\delta_n \Z^d$ such that $X_0^i = x_i$, with $X^i$ and $X^j$ independent for $i\neq j$.
	
For $z \in \delta_n \Z^d$ and $\epsilon>0$, consider the stopping times
	\begin{align} &\tau_z^i = \inf \big\{t \geq 0 \,|\, X_t^i = z\big\}; 
	\nonumber \\
			        &\tau_\epsilon^i = \inf \big\{t \geq 0 \,|\, X_t^i \not \in D_\epsilon^\Points\big\}; 
			        \nonumber \\
			        &\nu^i = \inf \big\{t \geq 0 \,|\, X_t^i \not \in \{X_{\nu^j}^j\}_{j=1}^{i-1} \big\}. \label{stoppingtimefortheithparticle}	
	\end{align}
The stopping time $\nu^i$ is defined inductively in $i$ with $\nu^1 = 0$.  We think of building up the internal DLA cluster one site at a time, by letting the particle labeled $i$ walk until it exits the set of sites already occupied by particles with smaller labels.  Thus $\nu^i$ is the number of steps taken by the particle labeled $i$, and $X_{\nu^i}^i$ is the location where it stops.

Fix $z \in D_\epsilon^\Points$ and consider the random variables 
	\begin{align*} &\mathcal{M}_\epsilon = \sum_{i=1}^{m_n} 1_{\{\tau_z^i < \tau_\epsilon^i\}}; \index{$\mathcal{M}_\epsilon$}  \\
			&L_\epsilon = \sum_{i=1}^{m_n} 1_{\{\nu^i \leq \tau_z^i < \tau_\epsilon^i\}}. \index{$L_\epsilon$}
	\end{align*}
These sums can be interpreted in terms of the following two-stage procedure.  During the first stage, we form the occupied internal DLA cluster by adding on one site at a time, as described above.  In the second stage, we let each particle continue walking from where it stopped until it exits $D_\epsilon^\Points$.  Then $\mathcal{M}_\epsilon$ counts the total number of particles that visit $z$ during both stages, while $L_\epsilon$ counts the number of particles that visit $z$ during the second stage.  In particular, if $L_\epsilon < \mathcal{M}_\epsilon$, then $z$ was visited during the first stage and hence belongs to the occupied cluster.

The sum $L_\epsilon$ is difficult to estimate directly because the indicator random variables in the sum are not independent.  Following \cite{LBG}, we can bound $L_\epsilon$ by a sum of independent indicators as follows.  For each site $y \in D_\epsilon^\Points$, let $(Y_t^y)_{t \geq 0}$ be a simple random walk in $\delta_n \Z^d$ such that $Y_0^y=y$, with $Y^x$ and~$Y^{y}$ independent for $x \neq y$.  Let
	\[ \widetilde{L}_\epsilon = \sum_{y \in D_\epsilon^\Points} 1_{\{\tau_z^y < \tau_\epsilon^y\}} \index{$\widetilde{L}_\epsilon$} \]
where
	\begin{align*}
	&\tau_z^y = \inf \big\{t \geq 0 \,|\, Y_t^y = z\big\}; \index{$\tau_z^y$, first hitting time of $z$} \\
	&\tau_\epsilon^y = \inf \big\{t \geq 0 \,|\, Y_t^y \not \in D_\epsilon^\Points\big\}. \index{$\tau_\epsilon^y$, first exit time of $D_\epsilon^\Points$}
	\end{align*}
Thus $\widetilde{L}_\epsilon$ counts the number of walks $Y^y$ that hit $z$ before exiting $D_\epsilon^\Points$.  Since the sites $X_{\nu^i}^i$ are distinct, we can couple the walks $\{Y^y\}$ and $\{X^i\}$ so that $L_\epsilon \leq \widetilde{L}_\epsilon$. 

Define
	\[ f_{n,\epsilon}(z) = g_{n,\epsilon}(z,z) \EE(\mathcal{M}_\epsilon-\widetilde{L}_\epsilon), \index{$f_{n,\epsilon}$} \]
where 
	\[ g_{n,\epsilon}(y,z) = \EE \#\{t < \tau_\epsilon^y \,|\, Y_t^y=z\} \index{$g_{n,\epsilon}$, Green's function on $D_\epsilon^\Points$}\]
is the Green's function for simple random walk in $\delta_n \Z^d$ stopped on exiting $D_\epsilon^\Points$.  Then
	\begin{align} \label{fintermsofgreens} f_{n,\epsilon}(z) &= g_{n,\epsilon}(z,z) \left(\sum_{i=1}^{m_n} \PP\big(\tau_z^i<\tau_\epsilon^i\big) - \sum_{y \in D_\epsilon^\Points} \PP\big(\tau_z^y<\tau_\epsilon^y\big) \right)
	    \nonumber \\
	    &= g_{n,\epsilon}(z,z) \sum_{y \in D_\epsilon^\Points} (\sigma_n(y)-1) \PP\big(\tau_z^y<\tau_\epsilon^y\big) \nonumber \\
	&= \sum_{y \in D_\epsilon^\Points} (\sigma_n(y)-1)g_{n,\epsilon}(y,z) \end{align}
where in the last step we have used the identity
	\begin{equation} \label{greensquotient} \PP\big(\tau_z^y < \tau_\epsilon^y\big) = \frac{g_{n,\epsilon}(y,z)}{g_{n,\epsilon}(z,z)}. \end{equation}
Thus $f_{n,\epsilon}$ solves the Dirichlet problem
	\begin{align}
	\Delta f_{n,\epsilon} &= \delta_n^{-2}(1-\sigma_n), 
	\qquad &\text{on }D_\epsilon^\Points; \label{laplacianoff} \\
	f_{n,\epsilon}&=0, \qquad &\text{on }\partial D_\epsilon^\Points. \nonumber
	\end{align}

Note that the divisible sandpile odometer function $u_n$ for the source density $\sigma_n$ solves exactly the same Dirichlet problem, with the domain $D_\epsilon^\Points$ in (\ref{laplacianoff}) replaced by the domain $D_n$ of fully occupied sites.  Our strategy will be first to use Theorem~\ref{domainconvergence} to argue that since $D_n \to D$, the solutions to the Dirichlet problems in $D_n$ and $D_\epsilon^\Points$ should be close; next, by Theorem~\ref{odomconvergence}, since $u_n \to u$ it follows that the functions $f_{n,\epsilon}$ and $u$ are close.  Since $u$ is strictly positive in $D_\epsilon$, we obtain in this way a lower bound on $f_{n,\epsilon}$, and hence a lower bound on $\EE (\mathcal{M}_\epsilon-\widetilde{L}_\epsilon)$.  Finally, using large deviations for sums of independent indicators, we conclude that with high probability $\widetilde{L}_\epsilon<\mathcal{M}_\epsilon$, and hence with high probability every point $z \in D_\epsilon^\Points$ belongs to the occupied cluster $I_n$.  The core of the argument is Lemma~\ref{coreoftheargument}, below, which gives the desired lower bound on $f_{n,\epsilon}$.

In order to apply Theorems~\ref{odomconvergence} and~\ref{domainconvergence}, we must have discrete densities $\sigma_n$ which converge pointwise to $\sigma$ at all continuity points of $\sigma$.  Recall, however, that in order to allow $\sigma$ to assume non-integer values, we have chosen not to assume that $\sigma_n$ converges to $\sigma$, but rather only that the \emph{smoothed density} $S_{\alpha(n)} \sigma_n$ converges to $\sigma$; see (\ref{smootheddensityconvergence}).  For this reason, we will need to run the divisible sandpile on $\delta_n \Z^d$ starting from the smoothed density rather than from $\sigma_n$, and we will use the following smoothed version of equation (\ref{laplacianoff}).
	\begin{align}
	\Delta S_{k} f_{n,\epsilon} = \delta_n^{-2}(1-S_{k}\sigma_n), 
	\qquad \text{on }D_{\epsilon'}^\Points. \label{smoothedlaplacianoff}
	\end{align}
Here $k \geq 1$ is arbitrary, and $\epsilon' = \epsilon + k\delta_n$.


The following lemma is proved in the same way as Lemma~\ref{smoothingdist}.
\begin{lemma}
\label{IDLAsmoothingdist}
Let $\epsilon' = \epsilon + k\delta_n$.  Then for $z \in D_{\epsilon'}^\Points$ we have
	\[ |S_k f_{n,\epsilon}(z) - f_{n,\epsilon}(z)| \leq \frac12 M k. \]
\end{lemma}

\begin{proof}
From (\ref{deltaintermsofsmoothing}) and (\ref{smoothedlaplacianoff}) we have for $z \in D_{\epsilon'}^\Points$
	\begin{align*} |S_k f_{n,\epsilon}(z) - f_{n,\epsilon}(z)| &\leq \sum_{j=0}^{k-1} |S_{j+1}f_{n,\epsilon}(z)-S_j f_{n,\epsilon}(z)| \\
			&= \frac{\delta_n^2}{2} \sum_{j=0}^{k-1} |\Delta S_j f_{n,\epsilon}(z)| \\
			&\leq \frac12 M k. \qed
			\end{align*}
\renewcommand{\qedsymbol}{}
\end{proof}
        
\begin{lemma}
\label{coreoftheargument}
Fix $\epsilon>0$, and let $\beta>0$ be the minimum value of $u=s-\gamma$ on $\overline{D_\epsilon}$.  There exists $0<\eta < \epsilon$ such that for all sufficiently large $n$
	\[ f_{n,\eta}(z) \geq \frac12 \beta \delta_n^{-2}, \qquad z \in D_\epsilon^\Points. \]
\end{lemma}

\begin{proof}
Since $u$ is uniformly continuous on $\overline{D}$, we can choose $\eta>0$ small enough so that $u \leq \beta/8$ outside $D_{2\eta}$.  Let $\eta' = \eta + \delta_n \alpha(n)$.  Since $\delta_n \alpha(n) \downarrow 0$, we have $u \leq \beta/8$ outside $D_{\eta'}$ if $n$ is sufficiently large.  Since $u\geq \beta$ on $\overline{D_\epsilon}$, for large enough $n$ we have $\partial D_{\eta'} \not\subset \overline{D_\epsilon}$ and hence $\eta'<\epsilon$.  Let $u_n$ be the odometer function for the divisible sandpile on $\delta_n \Z^d$ started from source density $S_{\alpha(n)} \sigma_n$.  We have $u_n^\Box \rightarrow u$ uniformly by Theorem~\ref{odomconvergence}, so for $n$ sufficiently large we have $|u_n-u^\Points| \leq \beta/8$, hence $u_n \leq \beta/4$ on $\partial D_{\eta'}^\Points$.  

Let $D_n$ be the domain of fully occupied sites for the divisible sandpile in $\delta_n \Z^d$ started from source density $S_{\alpha(n)} \sigma_n$.  By Theorem~\ref{domainconvergence} we have $D_\eta^\Points \subset D_n$ for all sufficiently large $n$.  Now (\ref{laplacianoff}) implies that $S_{\alpha(n)} f_{n,\eta}-\delta_n^{-2} u_n$ is harmonic on $D_{\eta'}^\Points$.  By Lemmas~\ref{IDLAsmoothingdist} and \ref{atmostquadratic}, for $z \in \partial D_{\eta'}^\Points$ we have
	\begin{align*} |S_{\alpha(n)} f_{n,\eta} (z) | &\leq |f_{n,\eta}(z)| + \frac12 M \alpha(n) \\
					&\leq c M \alpha(n)^2 + \frac12 M \alpha(n). \end{align*}
Since $\delta_n \alpha(n) \downarrow 0$, for sufficiently large $n$ we have
	\[ S_{\alpha(n)} f_{n,\eta}(z) \geq - \frac{\beta}{16 \delta_n^2}, \qquad z \in \partial D_{\eta'}^\Points. \]
Since $S_{\alpha(n)} f_{n,\eta}-\delta_n^{-2} u_n$ attains its minimum in $D_{\eta'}^\Points$ on the boundary, for $z \in D_{\eta'}^\Points$ we have
	\begin{align*} S_{\alpha(n)} f_{n,\eta}(z) &\geq \delta_n^{-2} \left(u_n(z) - \frac{5\beta}{16} \right) \\
					     &\geq \delta_n^{-2} \left(u(z) - \frac{7\beta}{16} \right). 
		 \end{align*}
Since $u \geq \beta$ in $D_\epsilon^\Points$, taking $n$ large enough so that $\alpha(n) \delta_n^2 \leq \beta/8M$, we conclude from Lemma~\ref{IDLAsmoothingdist} that 
	\[ f_{n,\eta} \geq \frac{9\beta}{16 \delta_n^2} - \frac12 M \alpha(n) \geq \frac12 \beta \delta_n^{-2}. \]
on $D_\epsilon^\Points$.	
\end{proof}

\begin{lemma}
\label{exittime}
We have
\[ \EE \widetilde{L}_\eta = \frac{\EE \tau_\eta^z}{g_{n,\eta}(z,z)}. \]
Moreover
\[ \EE \mathcal{M}_\eta \leq M \frac{\EE \tau_\eta^z}{g_{n,\eta}(z,z)}. \]
\end{lemma}

\begin{proof}
Let
	\[ T_z^y = \inf \{t \geq 0 | Y^y_t = z \} \] 
be the first hitting time of $z$ for the walk $Y^y$.  By (\ref{greensquotient}) and the symmetry of $g_{n,\eta}$, we have
	\begin{align} \label{exactforL}
	\EE \widetilde{L}_\eta &= \sum_{y \in D_\eta^\Points} \PP \big(T_z^y < \tau_\eta^y \big) \\
			&= \sum_{y \in D_\eta^\Points} \frac{g_{n,\eta}(y,z)}{g_{n,\eta}(z,z)} \nonumber \\
			&= \frac{1}{g_{n,\eta}(z,z)} \sum_{y \in D_\eta^\Points} g_{n,\eta}(z,y). \nonumber
	\end{align}
The sum in the last line is $\EE_z T_{\partial D_\eta^\Points}$.  

To prove the inequality for $\EE \mathcal{M}_\eta$, note that $\EE \mathcal{M}_\eta$ is bounded above by $M$ times the sum on the right side of (\ref{exactforL}).
\end{proof}

The next lemma, using a martingale argument to compute the expected time for simple random walk to exit a ball, is well known.

\begin{lemma}
\label{ballexittime}
Fix $r>0$ and let $B=B(o,r)^\Points \subset \delta_n \Z^d$, and let $T$ be the first hitting time of $\partial B$.  Then 
	\[ \EE_o T = \left( \frac{r}{\delta_n} \right)^2 + O \left( \frac{r}{\delta_n} \right). \]
\end{lemma}

\begin{proof}
Since $\delta_n^{-2} |X_t|^2 - t$ is a martingale with bounded increments, and $\EE_o T < \infty$, by optional stopping we have
	\[ \EE_o T = \delta_n^{-2} \EE_o |X_T|^2 = \delta_n^{-2} (r + O(\delta_n))^2. \qed \]
\renewcommand{\qedsymbol}{}
\end{proof}

The next lemma, which bounds the expected number of times simple random walk returns to the origin before reaching distance $r$, is also well known.

\begin{lemma}
\label{ballgreens}
Fix $r>0$ and let $B=B(o,r)^\Points \subset \delta_n \Z^d$, and let $G_B$ be the Green's function for simple random walk stopped on exiting $B$.  If $n$ is sufficiently large, then
	\[ G_B(o,o) \leq \log \frac{r}{\delta_n}. \]
\end{lemma}

\begin{proof}
In dimension $d\geq 3$ the result is trivial since simple random walk on $\delta_n \Z^d$ is transient.  In dimension two, consider the function
	\[ f(x) = G_B(o,x) - g_n(o,x) \]
where $g_n$ is the rescaled potential kernel defined in (\ref{definitionofg_ndimension2}).  
Since $f$ is harmonic in $B$ it attains its maximum on the boundary, hence by Lemma~\ref{greensfunctionconvergence} we have for $x \in B$
	\[ f(x) \leq \frac{2}{\pi} \log r + O \left ( \frac{\delta_n^2}{r^2} \right). \]
Since $g_n(o,o) = \frac{2}{\pi} \log \delta_n$, the result follows on setting $x=o$.
\end{proof}

We will use the following large deviation bound; for a proof, see \cite[Cor.\ A.14]{AS}. 

\begin{lemma}
\label{alonspencer}
If $N$ is a sum of finitely many independent indicator random variables, then for all $\lambda>0$
	\[ \PP(|N-\EE N|>\lambda \EE N) < 2e^{-c_\lambda \EE N} \]
where $c_\lambda>0$ is a constant depending only on $\lambda$.
\end{lemma}

Let
	\[ \tilde{I}_n = \big\{X_{\nu^i}^i| \nu^i<\tilde{\tau}^i \big\} \subset I_n \index{$\tilde{I}_n$} \]
where $\nu^i$ is given by (\ref{stoppingtimefortheithparticle}), and
	\[ \tilde{\tau}^i = \inf \big\{t \geq 0| X_t^i \notin \widetilde{D}^\Points \big\}. \]
The inner estimate of Theorem~\ref{IDLAconvergence} follows immediately from the lemma below.  Although for the inner estimate it suffices to prove Lemma~\ref{stronginnerestimate} with $I_n$ in place of $\tilde{I}_n$, we will make use of the stronger statement with $\tilde{I}_n$ in the proof of the outer estimate in the next section.

\begin{lemma}
\label{stronginnerestimate}
For any $\epsilon>0$,
	\[ \PP \big(\widetilde{D}_\epsilon^\Points \subset \tilde{I}_n \text{\em~for all but finitely many $n$}\big) = 1. \]
\end{lemma}

\begin{proof}
For $z \in \widetilde{D}_\epsilon^\Points$, let $\mathcal{E}_z(n)$ be the event that $z \notin \tilde{I}_n$.  By Borel-Cantelli it suffices to show that
	\begin{equation} \label{borelcantelliwiths} \sum_{n \geq 1} \sum_{z \in \widetilde{D}_\epsilon^\Points} \PP(\mathcal{E}_z(n)) < \infty. \end{equation}
By Lemma~\ref{pointset}, since $\widetilde{D} = D \cup \{\sigma \geq 1\}^o$ we have
	\[ \widetilde{D}_\epsilon \subset D_{\epsilon'} \cup \{\sigma \geq 1\}_{\epsilon'} \]
for some $\epsilon'>0$.  By (\ref{rotorsillytechnicality}), for $n \geq N(\epsilon')$ the terms in (\ref{borelcantelliwiths}) with $z \in \{\sigma \geq 1\}_{\epsilon'}$ vanish, so it suffices to show
	\begin{equation} \label{borelcantelliwithd} \sum_{n \geq 1} \sum_{z \in D_{\epsilon'}^\Points} \PP(\mathcal{E}_z(n)) < \infty. \end{equation}

By Lemma~\ref{coreoftheargument} there exists $0<\eta < \epsilon'$ such that 
	\begin{equation} \label{fromthecore} f_{n,\eta}(z) \geq \frac12 \beta \delta_n^{-2}, \qquad z \in D_{\epsilon'}^\Points \end{equation}
for all sufficiently large $n$, where $\beta>0$ is the minimum value of $u$ on $\overline{D_{\epsilon'}}$.  Fixing $z \in D_{\epsilon'}^\Points$, since $L_\eta \leq \widetilde{L}_\eta$ we have
	\begin{align} \PP(\mathcal{E}_z(n)) &\leq \PP(\mathcal{M}_\eta = L_\eta) \nonumber \\
						 &\leq \PP(\mathcal{M}_\eta \leq \widetilde{L}_\eta) \nonumber \\
						 &\leq \PP(\mathcal{M}_\eta \leq a) + \PP(\widetilde{L}_\eta \geq a) \label{choiceofa} \end{align}
for a real number $a$ to be chosen below.  By Lemma~\ref{alonspencer}, since $\widetilde{L}_\eta$ and $\mathcal{M}_\eta$ are sums of independent indicators, we have
\begin{align} & \PP( \widetilde{L}_\eta \geq (1+\lambda) \EE \widetilde{L}_\eta) < 2e^{-c_\lambda \EE \widetilde{L}_\eta} \label{largedevbounds} \\
			&\PP(\mathcal{M}_\eta \leq (1-\lambda) \EE \mathcal{M}_\eta) < 2e^{-c_\lambda \EE \mathcal{M}_\eta} \nonumber
			 \end{align}
where $c_\lambda$ depends only on $\lambda$, chosen below.  Now since $z\in D_{\epsilon'}$ and $D$ is bounded by Lemma~\ref{occupieddomainisbounded}, we have
	\[ B(z,\epsilon'-\eta) \subset D_\eta \subset B(z,R) \]
hence by Lemma~\ref{ballexittime}
	\begin{equation} \label{quadraticexittimebound} \frac12 \left( \frac{\epsilon'-\eta}{\delta_n} \right)^2 \leq \EE \tau_\eta^z \leq \left( \frac{R}{\delta_n} \right)^2. \end{equation}
By Lemma~\ref{exittime} it follows that $\EE \mathcal{M}_\eta \leq MR^2/\delta_n^2 g_{n,\eta}(z,z)$.  Taking 
	\[ a = \EE \widetilde{L}_\eta + \frac{\beta}{4\delta_n^2 g_{n,\eta}(z,z)} \] 
in (\ref{choiceofa}), and letting $\lambda = \beta/4MR^2$ in (\ref{largedevbounds}), we have
	\begin{equation} \label{thisfollowsfromthedefinitions} \lambda \EE \widetilde{L}_\eta \leq \lambda \EE \mathcal{M}_\eta \leq \frac{\beta}{4 \delta_n^2 g_{n,\eta}(z,z)}, \end{equation}
hence
		\[ a \geq (1+\lambda) \EE \widetilde{L}_\eta. \]
Moreover, from (\ref{fromthecore}) we have
	\[ \EE \mathcal{M}_\eta - E\widetilde{L}_\eta = \frac{f_{n,\eta}(z)}{g_{n,\eta}(z,z)} 
						\geq \frac{\beta}{2\delta_n^2 g_{n,\eta}(z,z)}, \]
hence by (\ref{thisfollowsfromthedefinitions})
	\[ a \leq \EE \mathcal{M}_\eta - \frac{\beta}{4\delta_n^2 g_{n,\eta}(z,z)} \leq (1-\lambda) \EE \mathcal{M}_\eta. \]

Thus we obtain from (\ref{choiceofa}) and (\ref{largedevbounds})
	\begin{equation} \label{largedev} \PP(\mathcal{E}_z(n)) \leq 4 e^{-c_\lambda \EE \widetilde{L}_\eta}. \end{equation}
By Lemmas~\ref{exittime} and~\ref{ballgreens} along with (\ref{quadraticexittimebound})
	\begin{align*} \EE \widetilde{L}_\eta &= \frac{\EE_z T_{\partial D_\eta^\Points}}{g_{n,\eta}(z,z)} \\
		&\geq \frac12 \left( \frac{\epsilon'-\eta}{\delta_n} \right)^2 \frac{1}{\log(R/\delta_n)}. \end{align*}
Using (\ref{largedev}), the sum in (\ref{borelcantelliwithd}) is thus bounded by
	\[ \sum_{n \geq 1} \sum_{z \in D_{\epsilon'}^\Points} \PP (\mathcal{E}_z(n)) \leq
	   \sum_{n \geq 1} \delta_n^{-d} \omega_d R^d \cdot 4 \exp \left(-\frac{c_\lambda(\epsilon'-\eta)^2}{2\delta_n^2 \log(R/\delta_n)} \right) < \infty. \qed \]
\renewcommand{\qedsymbol}{}	
\end{proof}

\subsection{Outer Estimate}

For $x \in \Z^d$ write
	\[ Q(x,h) = \{y \in \Z^d \,:\, ||x-y||_{\infty} \leq h \} \]
for the cube of side length $2h+1$ centered at $x$.  According to the next lemma, if we start a simple random walk at distance $h$ from a hyperplane $H \subset \Z^d$, then the walk is farily ``spread out'' by the time it hits $H$, in the sense that its chance of first hitting $H$ at any particular point $z$ has the same order of magnitude as $z$ ranges over a $(d-1)$-dimensional cube of side length order $h$.

\begin{lemma}
\label{boxshells}
Fix $y \in \Z^d$ with $y_1 = h$, and let $T$ be the first hitting time of the hyperplane $H=\{x\in \Z^d|x_1=0\}$.  Let $F = H \cap Q(y,h)$.   For any $z \in F$ we have  
	\[ \PP_y(X_T = z) \geq ah^{1-d} \]
for a constant $a$ depending only on $d$.
\end{lemma}

\begin{proof}
For fixed $w \in H$, the function
	\[ f(x) = \PP_x (X_T=w) \]
is harmonic in the ball $B=B(y,h)$.  By the Harnack inequality \cite[Theorem 1.7.2]{Lawler} we have
	\[ f(x) \geq c f(y), \qquad x \in B(y,h/2) \]
for a constant $c$ depending only on $d$.  By translation invariance, it follows that for $w' \in H \cap B(w,h/2)$, letting $x = y+w-w'$ we have
	\[ \PP_y (X_T = w') = f(x) \geq c f(y) = c \PP_y (X_T = w). \]
Iterating, we obtain for any $w' \in H \cap Q(w,h)$
	\begin{equation} \label{harnackincube} \PP_y (X_T = w') \geq c^{\sqrt{d}} \PP_y (X_T = w). \end{equation}
Let $T'$ be the first exit time of the cube $Q(y,h-1)$.  Since $F$ is a boundary face of this cube, we have $\{X_{T'} \in F\} \subset \{X_T \in F\}$.  Let $z_0$ be the closest point to $y$ in $H$.  Taking $w'=z_0$, we have $w' \in H \cap Q(w,h)$ whenever $w \in F$.  Summing (\ref{harnackincube}) over $w \in F$, we obtain
	\[ (2h+1)^{d-1} c^{-\sqrt{d}} \PP_y(X_T=z_0)  \geq \PP_y(X_T \in F) \geq \PP(X_{T'} \in F) = \frac{1}{2d}. \]
Now for any $z\in F$, taking $w=z_0$ and $w'=z$ in (\ref{harnackincube}), we conclude that
	\[ \PP_y(X_T = z) \geq c^{\sqrt{d}} \PP_y(X_T=z_0) \geq \frac{c^{2\sqrt{d}}}{2d}(2h+1)^{1-d}. \qed \]
\renewcommand{\qedsymbol}{}
\end{proof}

\begin{lemma}
\label{thetrenches}
Let $h,\rho$ be positive integers.  Let $N$ be the number of particles that ever visit the cube $Q(o,\rho)$ during the internal DLA process, if one particle starts at each site $y \in \Z^d - Q(o,\rho+h)$, and $k$ additional particles start at sites $y_1,\ldots, y_k \notin Q(o,\rho+h)$.  If $k \leq \frac14 h^d$, then 
	\[ N \leq \text{\em Binom}(k,p), \]
where $p<1$ is a constant depending only on $d$.
\end{lemma}

\begin{proof}
Let $F_j$ be a half-space defining a face of the cube $Q=Q(o,\rho)$, such that dist$(y_j,F_j) \geq h$.  Let $z_j\in F_j$ be the closest point to $y_j$ in $F_j$, and let $B_j = Q(z_j,h/2) \cap F_j^c$.  Let $Z_j$ be the random set of sites where the particles starting at $y_1, \ldots, y_j$ stop.  Since $\#B_j \geq \frac12 h^d \geq 2k$, there is a hyperplane $H_j$ parallel to $F_j$ and intersecting $B_j$, such that
	\begin{equation} \label{sparseslice} \# B_j \cap H_j \cap Z_{j-1} \leq \frac12 h^{d-1}. \end{equation}
Denote by $A_j$ the event that the particle starting at $y_j$ ever visits $Q$.  On this event, the particle must pass through $H_j$ at a site which was already occupied by an earlier particle.  Thus if $\{X_t\}_{t \geq 0}$ is the random walk performed by the particle starting at $y_j$, then
	\[ A_j \subset \{ X_T \in Z_{j-1} \}, \]
where $T$ is the first hitting time of $H_j$.  By Lemma~\ref{boxshells}, every site $z \in B_j\cap H_j$ satisfies
	\[ \PP (X_T = z) \geq ah^{1-d}. \]
From (\ref{sparseslice}), since $\# B_j \cap H_j \geq h^{d-1}$ we obtain $\PP(X_T \notin Z_{j-1}) \geq a/2$, hence
	\[ \PP(A_j | \mathcal{F}_{j-1}) \leq p \]
where $p=1-a/2$, and $\mathcal{F}_i$ is the $\sigma$-algebra generated by the walks performed by the particles starting at $y_1, \ldots, y_i$.  Thus we can couple the indicators $1_{A_j}$ with i.i.d. indicators $I_j \geq 1_{A_j}$ of mean $p$ to obtain
	\[ N = \sum_{j=1}^k 1_{A_j} \leq \sum_{j=1}^k I_j = \text{Binom}(k,p). \qed \]
\renewcommand{\qedsymbol}{}
\end{proof}

The next lemma shows that if few enough particles start outside a cube $Q(o,3\rho)$, it is highly unlikely that any of them will reach the smaller cube $Q(o,\rho)$.

\begin{lemma}
\label{cubedeath}
Let $\rho,k$ be positive integers with
	\[ k \leq \frac14 \big(1-p^{1/2d}\big)^d \rho^d \]
where $p<1$ is the constant in Lemma~\ref{thetrenches}.  Let $N$ be the number of particles that ever visit the cube $Q(o,\rho)$ during the internal DLA process, if one particle starts at each site $y \in \Z^d - Q(o,3\rho)$, and $k$ additional particles start at sites $y_1,\ldots, y_k \notin Q(o,3\rho)$.  Then
	\[ \PP(N>0) \leq c_0 e^{-c_1\rho} \]
where $c_0,c_1>0$ are constants depending only on $d$.
\end{lemma}

\begin{proof}
Let $N_j$ be the number of particles that ever visit the cube $Q_j = Q(o, \rho_j)$, where
	\[ \rho_j = \big(2+p^{j/2d}\big)\rho. \]
Let $k_j = p^{j/2}k$, and let $A_j$ be the event that $N_j \leq k_j$.  Taking $h = \rho_j - \rho_{j+1}$ in Lemma~\ref{thetrenches}, since
	\[ k_j \leq \frac14 p^{j/2} \big(1-p^{1/2d}\big)^d \rho^d = \frac14 h^d \]
we obtain
	\[ N_{j+1}1_{A_j} \leq \text{Binom}(N_j,p). \]
Hence
	\begin{align} \label{conditionalprobintermsofbinomial} \PP(A_{j+1}|A_j) &\geq \PP \big(\text{Binom}(k_j, p) \leq k_{j+1} \big) \nonumber \\		
		&\geq 1-2e^{-ck_j} \end{align}
where in the second line we have used Lemma~\ref{alonspencer} with $\lambda=\sqrt{p}-p$.

Now let 
	\[ j = \left\lfloor 2 \frac{\log k - \log \rho}{\log(1/p)} \right\rfloor \]
so that $p^{1/2} \rho \leq k_j \leq \rho$.  On the event $A_j$, at most $\rho$ particles visit the cube $Q(o,2\rho)$.  Since the first particle to visit each cube $Q(o,2\rho-i)$ stops there,  at most $\rho-i$ particles visit $Q(o,2\rho-i)$.  Taking $i=\rho$ we obtain $\PP(N=0) \geq \PP(A_j)$.  
 From (\ref{conditionalprobintermsofbinomial}) we conclude that
	\[ \PP(N=0) \geq \PP(A_1) \PP(A_2 | A_1) \cdots \PP(A_j|A_{j-1}) \geq 1 - 2je^{-c\rho}. \]
The right side is at least $1-c_0e^{-c_1\rho}$ for suitable constants $c_0,c_1$.
\end{proof}

\begin{proof}[Proof of Theorem~\ref{IDLAconvergence}] The inner estimate is immediate from Lemma~\ref{stronginnerestimate}.  For the outer estimate, let
	\[ N_n = \# \big\{1 \leq i \leq m_n \big| \nu^i \geq \tilde{\tau}^i\big\} = m_n - \# \tilde{I}_n \]
be the number of particles that leave $\widetilde{D}^\Points$ before aggregating to the cluster.
Let $K_0 = \frac14 \left(\frac{2(1-p^{1/2d})}{3\sqrt{d}}\right)^d$, where $p$ is the constant in Lemma~\ref{thetrenches}.  For $\epsilon>0$ and $n_0 \geq 1$, consider the event 
	\[ F_{n_0} = \big\{ N_n \leq K_0 (\epsilon/\delta_n)^d \text{ for all } n \geq n_0 \big\}. \]
By (\ref{rotorboundedawayfrom1}) and (\ref{rotorclosureofinterior}), the closure of $\widetilde{D}$ contains the support of $\sigma$, so by Proposition~\ref{boundaryregularity},
	\begin{equation} \label{fromconservationofmass} \delta_n^d m_n = \delta_n^d \sum_{x \in \delta_n \Z^d} \sigma_n(x) \rightarrow \int_{\R^d} \sigma(x) dx = \Leb(\widetilde{D}). \end{equation}
Moreover, by Proposition~\ref{boundaryregularity}(i), for sufficiently small $\eta$ we have 
	\[ \Leb\big(\widetilde{D}-\widetilde{D}_{2\eta}\big) \leq \frac12 K_0 \epsilon^d. \]
Taking $n$ large enough so that $\delta_n<\eta$, we obtain
	\[ \delta_n^d \#\widetilde{D}_\eta^\Points = \Leb \big(\widetilde{D}_\eta^{\Points\Box}\big) \geq \Leb\big(\widetilde{D}_{2\eta}\big) \geq \Leb\big(\widetilde{D}\big) - \frac12 K_0 \epsilon^d. \]
Thus for sufficiently large $n$, on the event $\widetilde{D}_\eta^\Points \subset \tilde{I}_n$ we have
	\begin{align*} N_n &\leq m_n - \# \widetilde{D}_\eta^\Points \\
	 	&\leq m_n - \delta_n^{-d} \Leb\big(\widetilde{D}\big) + \frac12 K_0 \epsilon^d \delta_n^{-d} \\
		& \leq K_0 \epsilon^d \delta_n^{-d}, \end{align*}
where in the last line we have used (\ref{fromconservationofmass}).  From Lemma~\ref{stronginnerestimate} we obtain
	\begin{equation} \label{probtendingto1} \PP\big(F_{n_0}\big) \geq \PP\big( \widetilde{D}_\eta^\Points \subseteq \tilde{I}_n \text{ for all } n \geq n_0 \big) \uparrow 1 \end{equation}
as $n_0 \uparrow \infty$.

By compactness, we can find finitely many cubes $Q_1, \ldots, Q_m$ of side length $\rho = 2\epsilon/3\sqrt{d}$ centered at points in $\partial (\widetilde{D}^\epsilon)$, with $\partial (\widetilde{D}^\epsilon) \subset \bigcup Q_i$.  Taking $k= \floor{K_0(\epsilon/\delta_n)^d}$ in Lemma~\ref{cubedeath}, since $3Q_i$ is disjoint from $\widetilde{D}$, we obtain for $n \geq n_0$
	\[ \PP\big(\{Q_i \cap I_n \neq \emptyset\} \cap F_{n_0}\big) \leq c_0e^{-c_1\rho}. \]
Summing over $i$ yields
	\[ \PP \big(\big\{\widetilde{D}^\epsilon \cap I_n \neq \emptyset \big\} \cap F_{n_0}\big) \leq c_0me^{-c_1\rho}. \]
By Borel-Cantelli, if $G$ is the event that $I_n \not\subset \widetilde{D}^{\epsilon\Points}$ for infinitely many $n$, then $\PP\big(F_{n_0} \cap G\big) = 0$.  From (\ref{probtendingto1}) we conclude that $\PP(G)=0$.
\end{proof}

\section{Multiple Point Sources}
\label{multiplesources}

This section is devoted to proving Theorems~\ref{multiplepointsources} and~\ref{algebraicboundary}.

\subsection{Associativity and Hausdorff Continuity of the Smash Sum}


In this section we establish two basic properties of the smash sum (\ref{smashsumdef}) that are needed to prove Theorems~\ref{multiplepointsources} and~\ref{algebraicboundary}.  In Lemma~\ref{associativity} we show that the smash sum is associative, and in Lemma~\ref{sumofneighborhoods} we show that it is continuous in the Hausdorff metric.

\begin{lemma}
\label{associativity}
Let $A,B,C \subset \R^d$ be bounded open sets whose boundaries have measure zero.  Then
	\begin{align*} (A \oplus B) \oplus C &= A \oplus (B \oplus C) \\
							        &= A \cup B \cup C \cup D \end{align*}
where $D$ is given by (\ref{thenoncoincidenceset}) with $\sigma = 1_A + 1_B + 1_C$.
\end{lemma}

\begin{proof}
Let
	\[ \gamma(x) = -|x|^2 - G(1_A + 1_B)(x) \]
and
	\[ \hat{\gamma}(x) = -|x|^2 - G(1_{A\oplus B} + 1_C)(x). \]
Let $u = s-\gamma$ and $\hat{u} = \hat{s}-\hat{\gamma}$, where $s,\hat{s}$ are the least superharmonic majorants of $\gamma,\hat{\gamma}$.  Then
	\begin{align} (A \oplus B) \oplus C &= (A \oplus B) \cup C \cup \{\hat{u}>0\} \nonumber \\
							&= A \cup B \cup \{u>0\} \cup C \cup \{\hat{u}>0\} \nonumber \\
\label{sumofodometers}			&= A \cup B \cup C \cup \{u+\hat{u}>0\}. \end{align}
Let $\nu_n$ be the final mass density for the divisible sandpile in $\delta_n \Z^d$ started from source density $1_{A^\Points} + 1_{B^\Points}$, and let $u_n$ be the corresponding odometer function.  By Theorem~\ref{odomconvergence} we have $u_n \to u$ as $n \to \infty$.  Moreover, by Theorem~\ref{domainconvergence} we have
	\begin{equation} \label{restatementofdomainconvergence} \nu_n^\Box(x) \to 1_{A\oplus B}(x) \end{equation}
for all $x \notin \partial(A \oplus B)$.  Let $\hat{u}_n$ be the odometer function for the divisible sandpile on $\delta_n \Z^d$ started from source density $\nu_n + 1_{C^\Points}$.  By Proposition~\ref{boundaryregularity}(i) the right side of (\ref{restatementofdomainconvergence}) is continuous almost everywhere, so by Theorem~\ref{odomconvergence} we have $\hat{u}_n \to \hat{u}$ as $n \to \infty$.  On the other hand, by the abelian property, Lemma~\ref{abelianproperty}, the sum $u_n + \hat{u}_n$ is the odometer function for the divisible sandpile on $\delta_n \Z^d$ started from source density $1_{A^\Points} + 1_{B^\Points} + 1_{C^\Points}$, so by Theorem~\ref{odomconvergence} we have $u_n + \hat{u}_n \to \tilde{u} := \tilde{s} - \tilde{\gamma}$, where
	\[ \tilde{\gamma}(x) = -|x|^2 - G(1_A + 1_B + 1_C)(x) \]
and $\tilde{s}$ is the least superharmonic majorant of $\tilde{\gamma}$.  In particular, 
	\[ u+\hat{u} = \lim_{n \to \infty} (u_n + \hat{u}_n) = \tilde{u}, \] 
so the right side of (\ref{sumofodometers}) 
is equal to $A \cup B \cup C \cup D$.
\end{proof}

The following lemma shows that the smash sum of two sets with a small intersection cannot extend very far beyond their union.  As usual, $\mathcal{L}$ denotes Lebesgue measure in $\R^d$, and $A^\epsilon$ denotes the outer $\epsilon$-neighborhood of a set $A \subset \R^d$.

\begin{lemma}
\label{limitedexpansion}
Let $A,B \subset \R^d$ be bounded open sets whose boundaries have measure zero, and let $\rho = (\Leb (A \cap B))^{1/d}$.  There is a constant $c$, independent of $A$ and $B$, such that
	\[ A\oplus B \subset (A \cup B)^{c\rho}. \]
\end{lemma}

\begin{proof}
Let $\gamma,s$ be given by (\ref{theobstacle}) and (\ref{themajorant}) with $\sigma = 1_A+1_B$, and write $u=s-\gamma$.  Fix $x \in (A\oplus B)-(A\cup B)$ and let $r =$ dist$(x,A \cup B)$.  Let $\mathcal{B} = B(x,r/2)$.  By Lemma~\ref{majorantbasicprops}(iii), $s$ is harmonic in $A\oplus B - (A \cup B)$, so by Lemma~\ref{superharmonicpotential} the function
	\begin{align*} w(y) &= u(y) - |x-y|^2 \\  
				&= s(y) + |y|^2 + G\big(1_{A}+1_{B}\big)(y) - |x-y|^2 \end{align*}
is harmonic on the intersection $(A\oplus B) \cap \mathcal{B}$; hence it attains its maximum on the boundary.  Since $w(x)>0$ the maximum cannot be attained on $\partial (A\oplus B)$, so it is attained at some point $y \in \partial \mathcal{B}$, and
	\begin{equation} \label{largeodom} u(y) \geq w(y) + \frac{r^2}{4} > \frac{r^2}{4}. \end{equation}

If $z$ is any point outside $A\oplus B$, then by Lemma~\ref{laplacianofodometer} and Lemma~\ref{continuumatmostquadratic} with $\lambda=4d$, there is a constant $c' \geq 1$ such that $u \leq c' h^2$ on $B(z,h)$ for all $h$.  Taking $h=r/2\sqrt{c'}$, we conclude from (\ref{largeodom}) that $B(y,h) \subset A\oplus B$.
Since $B(y,h)$ is disjoint from $A\cup B$, we have by Corollary~\ref{volumesadd}
	\[ \Leb(A\cup B) + \omega_d h^d \leq \Leb(A\oplus B) = \Leb(A) + \Leb(B), \]
hence $\omega_d h^d \leq \Leb(A \cap B)$.  Taking $c=2 \sqrt{c'}/\omega_d^{1/d}$ yields $r = c\, \omega_d^{1/d} h \leq c\rho$.
\end{proof}

The following lemma, together with Lemma~\ref{monotonicity}, shows that the smash sum operation 
is continuous in the Hausdorff metric.

\begin{lemma}
\label{sumofneighborhoods}
Let $A,B \subset \R^d$ be bounded open sets whose boundaries have measure zero.  For any $\epsilon>0$ there exists $\eta>0$ such that
	\begin{equation} \label{threeinclusions} (A \oplus B)_\epsilon \subset A_\eta \oplus B_\eta \subset A^\eta \oplus B^\eta \subset (A \oplus B)^\epsilon. \end{equation}
\end{lemma}

\begin{proof}
By Lemma~\ref{pointset}, since $A\oplus B = A\cup B\cup D$, we have
	\[ (A\oplus B)_\epsilon \subset A_{\epsilon'} \cup B_{\epsilon'} \cup D_{\epsilon'} \]
for some $\epsilon'>0$.  By Lemma~\ref{threesteps}(iii) with $\sigma = 1_A+1_B$ and $\sigma_n = 1_{A_{1/n}} + 1_{B_{1/n}}$, for sufficiently small $\eta$ we have $D_{\epsilon'} \subset A_\eta \oplus B_\eta$.  This proves the first inclusion in (\ref{threeinclusions}).

The second inclusion is immediate from Lemma~\ref{monotonicity}. 

 For the final inclusion, write $A' = A^\eta - A$ and $B' = B^\eta - B$.  Since  $\Leb(A') \downarrow \Leb(\partial A)=0$ and $\Leb(B') \downarrow \Leb(\partial B) = 0$ as $\eta \downarrow 0$, for small enough $\eta$ we have by Lemmas~\ref{associativity} and~\ref{limitedexpansion}
	 \begin{align*} A^\eta \oplus B^\eta &= A \oplus B \oplus A' \oplus B' \\
	 		&\subset ((A \oplus B) \cup A^\eta \cup B^\eta)^{\epsilon-\eta} \\
			&\subset (A \oplus B)^\epsilon. \qed \end{align*}
\renewcommand{\qedsymbol}{}
\end{proof}

\subsection{Smash Sums of Balls}

In this section we deduce Theorem~\ref{multiplepointsources} from our other results.  The following result on internal DLA with a single point source is a restatement of the main result of \cite{LBG}.

\begin{theorem}
\label{singlepointsource}
Fix $\lambda>0$, and let $I_n$ be the random set of occupied sites for internal DLA in $\delta_n \Z^d$, starting with $m = \big\lfloor\lambda \delta_n^{-d}\big\rfloor$ particles at the origin.  If $\delta_n \leq 1/n$ for all $n$, then for any $\epsilon>0$, we have with probability one
	\[ B_\epsilon^\Points \subset I_n \subset B^{\epsilon\Points} \qquad \text{\em for all sufficiently large $n$,} \] 
where $B$ is the ball of volume $\lambda$ centered at the origin in $\R^d$.
\end{theorem}

\begin{proof}[Proof of Theorem~\ref{multiplepointsources}]
For $i=1,\ldots,k$ let $D_n^i, R_n^i, I_n^i$ be the domain of occupied sites in $\delta_n \Z^d$ 
starting from a single point source of $\big\lfloor \delta_n^{-d} \lambda_i\big\rfloor$ particles at $x_i^\Points$.
By Theorems~\ref{rotorcircintro} and~\ref{divsandcircintro}, for any $\eta>0$ we have
	\begin{equation} \label{singlepointsources} (B_i)_\eta^\Points \subset D_n^i, R_n^i \subset (B_i)^{\eta\Points} \qquad \text{for all $i$ and all sufficiently large $n$.} \end{equation}
Moreover if $\delta_n \leq 1/n$, then by Theorem~\ref{singlepointsource} we have with probability one
	\begin{equation} \label{singlepointsourcesIDLA} (B_i)_\eta^\Points \subset I_n^i \subset (B_i)^{\eta\Points} \qquad \text{for all $i$ and all sufficiently large $n$.} \end{equation}

Next we argue that the domains $D_n,R_n,I_n$ can be understood as smash sums of $D_n^i, R_n^i, I_n^i$ as $i$ ranges over the integers $1, \ldots, k$.
By the abelian property \cite{DF}, the domain $I_n$ is the Diaconis-Fulton smash sum in $\delta_n \Z^d$ of the domains $I_n^i$.  Likewise, if $r_1$ is an arbitrary rotor configuration on $\delta_n \Z^d$, let $S_n^2$ be the smash sum of $R_n^1$ and $R_n^2$ formed using rotor-router dynamics with initial rotor configuration $r_1$, and let $r_2$ be the resulting final rotor configuration.  For $i \geq 3$ define $S_n^i$ inductively as the smash sum of $S_n^{i-1}$ and $R_n^i$ formed using rotor-router dynamics with initial rotor configuration $r_{i-1}$, and let $r_i$ be the resulting final rotor configuration.  Then $R_n=S_n^{k}$.  Finally, by Lemma~\ref{abelianproperty}, the domain $D_n$ contains the smash sum of domains $D_n^i$ formed using divisible sandpile dynamics, and $D_n$ is contained in the smash sum of the domains $D_n^i \cup \partial D_n^i$.  

Fixing $\epsilon>0$, by Theorem~\ref{DFsum} and Lemma~\ref{monotonicity}, it follows from (\ref{singlepointsources}) that for all sufficiently large $n$
	\begin{equation} \label{squeezed} A_{\epsilon/2}^\Points \subset D_n, R_n, I_n \subset \widetilde{A}^{\epsilon/2 \Points} \end{equation}
where $A$ is the smash sum of the balls $(B_i)_\eta$, and $\widetilde{A}$ is the smash sum of the balls $(B_i)^\eta$.  By Lemma~\ref{sumofneighborhoods} we can take $\eta$ sufficiently small so that
	\[ D_{\epsilon/2} \subset A \subset \widetilde{A} \subset D^{\epsilon/2} \]
where $D = B_1 \oplus \ldots \oplus B_k$.  Together with (\ref{squeezed}), this completes the proof.
\end{proof}

\subsection{Algebraic Boundary}

In this section we prove Theorem~\ref{algebraicboundary} showing that the boundary of a smash sum of disks in $\R^2$ is an algebraic curve.  The first step is to establish the classical quadrature identity (\ref{classicalquadrature}), which holds in any dimension.  We will prove the following quadrature inequality generalizing (\ref{classicalquadrature}) for superharmonic functions on a smash sum of balls in $\R^d$.

\begin{prop}
\label{ballquadrature}
Fix $x_1, \ldots, x_k \in \R^d$ and $\lambda_1, \ldots, \lambda_k>0$.  Let $B_i$ be the ball of volume $\lambda_i$ centered at $x_i$.  Then
	\[ \int_{B_1 \oplus \ldots \oplus B_k} h(x) dx \leq \sum_{i=1}^k \lambda_i h(x_i) \]
for all integrable superharmonic functions $h$ on $B_1 \oplus \ldots \oplus B_k$.
\end{prop}

We could deduce Proposition~\ref{ballquadrature} from Proposition~\ref{boundaryregularitysmooth} by integrating smooth approximations to $h$ against smooth densities $\sigma_n$ converging to the sum of the indicators of the $B_i$.  However, it more convenient to use the following result of M. Sakai, which is proved by a similar type of approximation argument. 

\begin{theorem} \cite[Thm.\ 7.5]{Sakai}
\label{fromsakai}
Let $\Omega \subset \R^d$ be a bounded open set, and let $\sigma$ be a bounded function on $\R^d$ supported on $\bar{\Omega}$ satisfying $\sigma>1$ on $\Omega$.  Let $\gamma, s, D$ be given by (\ref{theobstacle})-(\ref{thenoncoincidenceset}).  
	\[ \int_D h(x) dx \leq \int_D h(x) \sigma(x) dx \]
for all integrable superharmonic functions $h$ on $D$.
\end{theorem}

We remark that the form of the obstacle in \cite{Sakai} is superficially different from ours: taking $w=1$ and $\omega=\sigma$ in section~7 of \cite{Sakai}, the obstacle is given by
	\[ \psi(x) = G1_B - G\sigma \]
for a large ball $B$, rather than our choice (\ref{theobstacle}) of
	\[ \gamma(x) = -|x|^2 - G\sigma. \]
Note, however that $\gamma-\psi$ is constant on $B$ by (\ref{ballpotential}) and (\ref{ballpotentialdim2}).  Thus if $B$ is sufficiently large, the two obstacle problems have the same noncoincidence set $D$ by Lemmas~\ref{majorantonacompactset} and~\ref{occupieddomainisbounded}.

To prove Proposition~\ref{ballquadrature} using Theorem~\ref{fromsakai}, we must produce an appropriate density $\sigma$ on $\R^d$ strictly exceeding $1$ on its support.  We will take $\sigma$ to be twice the sum of indicators of balls of half the volume of the $B_i$.

\begin{proof}[Proof of Proposition~\ref{ballquadrature}]
Let $B'_i$ be the ball of volume $\lambda_i/2$ centered at $x_i$, and consider the sum of indicators
	\[ \sigma = 2 \sum_{i=1}^k 1_{B'_i}. \]
Let $\gamma, s, D$ be given by (\ref{theobstacle})-(\ref{thenoncoincidenceset}).  By Theorem~\ref{fromsakai} and the mean value proprerty for superharmonic functions (\ref{meanvalueproperty}), we have
	\[ \int_D h(x) dx \leq 2 \sum_{i=1}^k \int_{B'_i} h(x) dx \leq \sum_{i=1}^k \lambda_i h(x_i) \]
for all integrable superharmonic functions $h$ on $D$.

It remains to show that $D = B_1 \oplus \ldots \oplus B_k$.  By Lemma~\ref{startingdensitygreaterthan1} we have $B'_i \subset D$ for all $i$, hence by Lemma~\ref{associativity}
	\[ D = B'_1 \oplus B'_1 \oplus \ldots \oplus B'_k \oplus B'_k. \]
By Lemma~\ref{relaxingaball} we have $B_i = B'_i \oplus B'_i$, completing the proof. 
\end{proof}

Theorem~\ref{algebraicboundary} now follows from Proposition~\ref{ballquadrature} by a result of Gustafsson \cite[sec.\ 6]{Gustafsson83}; see also \cite[Lemma 1.1(a)]{Gustafsson88}.  The methods of \cite{Gustafsson83} rely heavily on complex analysis, which is why they apply only in dimension two.  Indeed, in higher dimensions it is not known whether classical quadrature domains have boundaries given by algebraic surfaces \cite[Ch.\ 2]{Shapiro}.

\chapter{Aggregation on Trees}
\label{trees}

\section{The Sandpile Group of a Tree}
\label{sandpiletree}

In this section we describe a short exact sequence relating the sandpile group of a tree to those of its principal subtrees.  In the case of a regular tree this sequence splits, enabling us to compute the full decomposition of the sandpile group into cyclic subgroups. This resolves in the affirmative a conjecture of E. Toumpakari concerning the ranks of the Sylow $p$-subgroups.

We begin with a simple combinatorial problem.  Let $T_n$ be the $d$-regular tree of height $n$.  Collapse all the leaves of $T_n$ to a single vertex $s$, the {\it sink}, and add an edge connecting the root to the sink.  

\begin{lemma} \label{teaser}
Let $t_n$ be the number of oriented spanning trees of $T_n$ rooted at the sink.  Then
	\[ t_n = t_{n-1}^{d-2} (dt_{n-1}-(d-1)t_{n-2}^{d-1}). \]
\end{lemma}

\begin{proof}
If the edge $(r,s)$ from the root to the sink is included in the spanning tree, then each of the principal branches of $T_n$ may be assigned an oriented spanning tree independently, so there are $t_{n-1}^{d-1}$ such spanning trees.  On the other hand, if $(r,s)$ is not included in the spanning tree, there is a directed path $r \rightarrow x_1 \rightarrow \ldots \rightarrow x_{n-1} \rightarrow s$ in the spanning tree from the root to the sink.  In this case, every principal branch except the one rooted at $x_1$ may be assigned an oriented spanning tree independently; within the branch rooted at $x_1$, every subbranch except the one rooted at $x_2$ may be assigned an oriented spanning tree independently; and so on (see Figure~\ref{spanningtreerecurrence}).  Since there are $(d-1)^{n-1}$ possible paths $x_1 \rightarrow \ldots \rightarrow x_{n-1}$, we conclude that
	\begin{equation} \label{productformofrecurrence} t_n = t_{n-1}^{d-1} + (d-1)^{n-1} \prod_{k=1}^{n-1} t_k^{d-2}. \end{equation}
Substituting $n-1$ for $n$ we find that
	\[ (d-1)^{n-2} \prod_{k=1}^{n-2} t_k^{d-2} = t_{n-1} - t_{n-2}^{d-1} \]
hence from (\ref{productformofrecurrence})
	\[ t_n = t_{n-1}^{d-2}(t_{n-1} + (d-1)(t_{n-1} - t_{n-2}^{d-1})). \qed \]
\renewcommand{\qedsymbol}{}
\end{proof}

\begin{figure}
\centering
\includegraphics[scale=1.1]{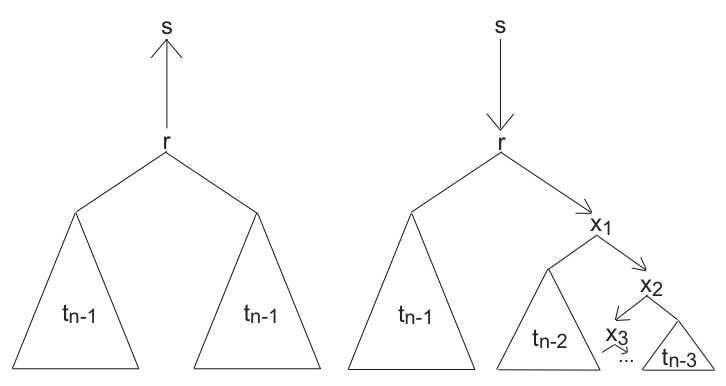}
\caption{The two cases in the proof of Lemma~\ref{teaser}}
\label{spanningtreerecurrence}
\end{figure}

From Lemma~\ref{teaser} one can readily show by induction that
	\[ t_n = (1+a+\ldots+a^{n-1}) \prod_{k=1}^{n-2} (1 + a + \ldots + a^k)^{a^{n-2-k}(a-1)}. \]
where $a=d-1$.  A variant of this formula was found by Toumpakari \cite{Toumpakari}, who gives an algebraic proof.  

For any graph $G$ there is an abelian group, the {\it sandpile group}, whose order is the number of oriented spanning trees of $G$ rooted at a fixed vertex; its definition and properties are reviewed in section~\ref{generaltrees}.  A product formula such as the one above immediately raises the question of an analogous factorization of the sandpile group.  Our main goal in this section is to prove Theorem~\ref{sandpilegroupdecomp}, which establishes precisely such a factorization.  In section~\ref{rotortree} we will apply this result to study the rotor-router model on regular trees.

The next three sections are organized as follows.  In section~\ref{generaltrees} we review the definition and basic properties of the sandpile group, and characterize the recurrent states on a tree explicitly in terms of what we call {\it critical vertices}.  We prove a general result, Theorem~\ref{quotientisom}, relating the sandpile group of an arbitrary tree to the sandpile groups of its principal branches (i.e.\ the subtrees rooted at the children of the root).  This result takes the form of an isomorphism between two quotients.  In section~\ref{regulartrees} we take advantage of the symmetry of regular trees to define a projection map back onto the subgroup being quotiented.  This allows us to express $SP(T_n)$ as the direct sum of a cyclic group and a quotient of the direct sum of $d-1$ copies of $SP(T_{n-1})$, which enables us to prove Theorem~\ref{sandpilegroupdecomp} by induction.  Finally, in section~\ref{toumpakari}, we prove Toumpakari's conjecture~\cite{Toumpakari}.
 
\subsection{General Trees}
\label{generaltrees}

Let $T$ be a finite rooted tree.  Collapse all the leaves to a single vertex $s$, the {\it sink}, and add an edge connecting the root to the sink.  If the vertices of $T$ are $x_1, \ldots, x_n=s$, the {\it sandpile group} of $T$ is defined by
	\[ SP(T) = \Z^n / \Delta \]
where $\Delta \subset \Z^n$ is the lattice
	\[ \Delta = \langle s,\Delta_{x_1}, \ldots, \Delta_{x_{n-1}} \rangle. \]
Here $\Delta_{x_i}$ is the vector in $\Z^n$ taking value $1$ at each neighbor of $x_i$, value $-$deg$(x)$ at $x$, and value $0$ elsewhere.  The sandpile group of a graph was originally defined in \cite{Biggs,BLS,Dhar}.

A nonnegative vector $u \in \Z^n$ may be thought of as a {\it chip configuration} with $u_i$ chips at vertex $x_i$.  A vertex $x \neq s$ is {\it unstable} if $u(x)\geq \text{deg}(x)$.  An unstable vertex may {\it topple}, sending one chip to each neighbor.  Note that the operation of toppling the vertex $x$ corresponds to adding the vector $\Delta_x$ to $u$.  We say that $u$ is {\it stable} if no vertex is unstable.  Given chip configurations $u$ and $v$, we define $u+v$ as the stable configuration resulting from starting with $u(x)+v(x)$ chips at each vertex $x$ and toppling any unstable vertices; the order in which topplings are performed does not affect the final configuration, as observed by Dhar \cite{Dhar} and (in a more general setting) Diaconis and Fulton \cite{DF}.

A chip configuration $u$ is {\it recurrent} \cite{CR} if $u+v = u$ for a nonnegative configuration $v$.  In \cite{CR} it is proved that every equivalence class mod $\Delta$ has a unique recurrent representative.  Thus the sandpile group $SP(T)$ may be thought of as the set of recurrent configurations under the operation of addition followed by toppling.  Note that if $v$ is a nonnegative configuration, its recurrent representative is given by
	\[ \hat{v} := v+e \]
where $e$ is the identity element of $SP(T)$ (the recurrent representative of $0$); indeed, $\hat{v}$ is recurrent since $e$ is recurrent, and $\hat{v} \equiv v$ (mod $\Delta$) since $e \in \Delta$.  Note that if $u$ is a recurrent configuration and $v$ is a nonnegative configuration, then
	\begin{equation} \label{recurrentpluspositive} u+\hat{v} = u+(v+e) = (u+e)+v = u+v. \end{equation}

The following result is a simple variant of the ``burning algorithm'' \cite{Dhar}; see also \cite[Cor.\ 2.6]{CR}.

\begin{lemma}
\label{burning}
Let $\beta(x)$ be the number of edges from $x$ to the sink.
A chip configuration on $T$ is recurrent if and only if adding $\beta(x)$ chips at each vertex $x$ causes every vertex to topple exactly once.
\end{lemma}

\begin{proof}
Note that 
	\begin{equation} \label{sumofalldeltas} \beta = \text{deg}(s) s+\Delta_s 
		=  \text{deg}(s) s - \sum_{x	 \neq s} \Delta_x. \end{equation}  
If every vertex topples exactly once, then
	\[ u+\beta = u + \beta + \sum_{x \neq s} \Delta_x = u, \]
so $u$ is recurrent.  Conversely, suppose $u$ is recurrent.  Since $\beta \in \Delta$ we have $\hat{\beta}=e$, hence from (\ref{recurrentpluspositive})
	\[ u+\beta = u+e = u. \]
By (\ref{sumofalldeltas}), since $\{\Delta_x\}_{x\neq s}$ are linearly independent, every vertex topples exactly once.
\end{proof}

\begin{figure}
\centering
\includegraphics[scale=1.15]{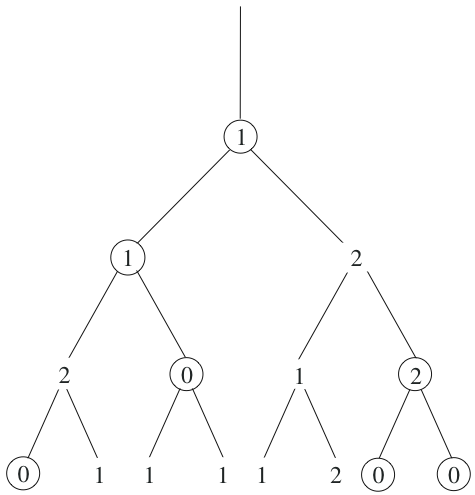}
\caption{A recurrent configuration on the ternary tree of height $5$.  Critical vertices are circled; if any of the circled vertices had fewer chips, the configuration would not be recurrent.}
\end{figure}

We first characterize the recurrent configurations of a tree explicitly.  The characterization uses the following inductive definition.  Denote by $C(x)$ the set of children of a vertex $x \in T$. 

\begin{definition}
A vertex $x \in T$ is {\it critical} for a chip configuration $u$ if $x\neq s$ and
	\begin{equation} \label{criticality} u(x) \leq \# \{y \in C(x) ~|~ \text{$y$ is critical}\}. \end{equation}
\end{definition}

\begin{prop}
\label{recurrentcharacterization}
A configuration $u \in SP(T)$ is recurrent if and only if equality holds in (\ref{criticality}) for every critical vertex $x$.
\end{prop}

\begin{proof}

If $x$ is critical, then
	\begin{equation} \label{restatementofcriticality} u(x) + \# \{y \in C(x) ~|~ y \text{ is not critical}\} \leq \text{deg}(x)-1. \end{equation}
Thus after chips are added as prescribed by Lemma~\ref{burning}, inducting upward in decreasing distance to the root, if $x \neq r$ is critical, its parent must topple before it does.  In particular, if strict inequality holds in (\ref{criticality}), and hence in (\ref{restatementofcriticality}), for some vertex $x$, that vertex will never topple, so $u$ is not recurrent.

Conversely, suppose equality holds in (\ref{criticality}), hence in (\ref{restatementofcriticality}), for every critical $x$.  Begin toppling vertices in order of decreasing distance from the root.  Note that a non-critical vertex $x$ satisfies
	\begin{equation} u(x) + \# \{y \in C(x) ~|~ y \text{ is not critical}\} \geq \text{deg}(x), \end{equation}
Inducting upward, every non-critical vertex topples once.  Hence by equality in (\ref{restatementofcriticality}), once all vertices other than the root are stable, every critical vertex $x$ has either toppled (if its parent toppled) or is left with exactly deg$(x)-1$ chips (if its parent did not topple).  In particular, the root now topples, as it was given an extra chip in the beginning.  Now if $x$ is a critical vertex that has not yet toppled, its parent is also such a vertex.  Inducting downward from the root, since all of these vertices are primed with deg$(x)-1$ chips, they each topple once, and $u$ is recurrent.
\end{proof}

Write $T_1, \ldots, T_k$ for the principal branches of $T$ (i.e.\ the subtrees rooted at the children of the root).  If $u_i$ is a chip configuration on $T_i$, and $a$ is an integer, we will use the notation $\left(\begin{array}{c} a \\ u_1,\ldots,u_k \end{array} \right)$ for the configuration on $T$ which has $a$ chips at the root and coincides with $u_i$ on $T_i$.  The following result is an immediate consequence of Proposition~\ref{recurrentcharacterization}.

\begin{lemma}
\label{recurrentsubconfigs}
Let $u = \left(\begin{array}{c} a \\ u_1,\ldots,u_k \end{array} \right)$.
	\begin{enumerate}
	\item[(i)]  If $u$ is recurrent, each $u_i$ is recurrent. 
	\item[(ii)]  If $u_1, \ldots, u_k$ are recurrent and $a=k$, then $u$ is recurrent.
	\end{enumerate}
\end{lemma}

Write $\delta_x$ for a single chip at a vertex $x$, and denote by $\hat{x} = e+\delta_x$ the recurrent form of $\delta_x$.  Note that by (\ref{recurrentpluspositive}), if $u$ is recurrent then
	\begin{equation} \label{youcanaddhoweveryouwant} u + \hat{x} = u+ \delta_x. \end{equation}

\begin{theorem}
\label{quotientisom}
Let $T_1, \ldots, T_k$ be the principal branches of $T$.  Then
	\[ SP(T)/(\hat{r}) \simeq \bigoplus_{i=1}^k SP(T_i) / ((\hat{r}_1,\ldots,\hat{r}_k)) \]
where $r$, $r_i$ are the roots of $T$, $T_i$ respectively.
\end{theorem}

\begin{proof}
Define $\phi : SP(T) \rightarrow \bigoplus_{i=1}^k SP(T_i)$ by 
	\begin{align} 
	\left(\begin{array}{c} a \\ u_1,\ldots,u_k \end{array} \right) &\mapsto (u_1,\ldots,u_k). 
	\end{align}
Lemma~\ref{recurrentsubconfigs}(i) ensures this map is well-defined.  Note that if 
	\[ \left(\begin{array}{c} a \\ u_1,\ldots,u_k \end{array} \right) = \left(\begin{array}{c} b \\ v_1,\ldots,v_k \end{array} \right) + \hat{r}, \]
by (\ref{youcanaddhoweveryouwant}) either $b<k$ and $u_i=v_i$ for all $i$; or $b=k$ and the root topples, in which case $u_i = v_i + \hat{r}_i$.  Thus $\phi$ descends to a map of quotients $\bar{\phi} : SP(T)/(\hat{r}) \rightarrow \bigoplus_{i=1}^k SP(T_i) / ((\hat{r}_1,\ldots,\hat{r}_k))$.

By adding two configurations without allowing the root to topple, the configurations on each branch add independently, hence by (\ref{youcanaddhoweveryouwant})
	\[ \left(\begin{array}{c} a \\ u_1,\ldots,u_k \end{array} \right) + \left(\begin{array}{c} b \\ v_1,\ldots,v_k \end{array} \right) = \left(\begin{array}{c} c \\ u_1+v_1,\ldots,u_k+v_k \end{array} \right)  + d\hat{r} \]
for some nonnegative integers $c,d$.  Thus $\bar{\phi}$ is a group homomorphism.  Moreover, $\bar{\phi}$ is surjective by Lemma~\ref{recurrentsubconfigs}(ii).  Finally, if
	\[ (u_1,\ldots,u_k) = (v_1,\ldots,v_k)+c(\hat{r}_1,\ldots,\hat{r}_k), \]
then by (\ref{youcanaddhoweveryouwant}), allowing the root to topple exactly $c$ times, we obtain
	\[ \left(\begin{array}{c} a \\ u_1,\ldots,u_k \end{array} \right) = \left(\begin{array}{c} b \\ v_1,\ldots,v_k \end{array} \right) + (c(k+1)+d)\hat{r}, \]
for a suitable integer $d$.  Thus $\bar{\phi}$ is injective.
\end{proof}

\subsection{Regular Trees}
\label{regulartrees}



In this section we show that for regular trees, Theorem~\ref{quotientisom} can be strengthened to express $SP(T)$ as a direct sum.

Let $T_n$ be the regular tree of degree $d$ and height $n$, with leaves collapsed to the sink vertex and an edge added from the root to the sink as in section~\ref{generaltrees}.  The chip configurations which are constant on the levels of $T_n$ form a subgroup of $SP(T_n)$.  If each vertex at height $k$ has $a_k$ chips, we can represent the configuration as a vector $(a_1,\ldots,a_{n-1})$.  If such a recurrent configuration is zero on a level, all vertices above that level are critical, so by Proposition~\ref{recurrentcharacterization} they must have $d-1$ chips each.  The recurrent configurations constant on levels are thus in bijection with integer vectors $(a_1,\ldots, a_{n-1})$ with $0\leq a_i \leq d-1$ subject to the constraint that if $a_i=0$ then $a_1=\ldots = a_{i-1}=d-1$.  

\begin{figure}
\[ \begin{array}{ccccccccccccccc}
 \hat{r} & 2\hat{r} & 3\hat{r} & 4\hat{r} & 5\hat{r} & 6\hat{r} &7\hat{r} & 8\hat{r} & 9\hat{r} & 10\hat{r} & 11\hat{r} & 12\hat{r} & 13\hat{r} & 14\hat{r} & 15\hat{r}=e \\ ~ \\
 
     2 & 0 & 1 & 2 & 0 & 1 & 2 & 2 & 2 & 0 & 1 & 2 & 0 & 1 & 2  \\
     0 & 1 & 1 & 1 & 2 & 2 & 2 & 2 & 0 & 1 & 1 & 1 & 2 & 2 & 2  \\
     2 & 2 & 2 & 2 & 2 & 2 & 2 & 0 & 1 & 1 & 1 & 1 & 1 & 1 & 1
\end{array} \]
\caption{Multiples of the root $\hat{r}$ in the ternary tree of height $4$.  Each column vector represents a chip configuration which is constant on levels of the tree.}
\label{lexorderfigure}
\end{figure}

The following lemma uses the lexicographic order given by $a<b$ if for some $k$ we have $a_{n-1}=b_{n-1}, \ldots, a_{k+1}=b_{k+1}$ and $a_k<b_k$.  In the {\it cyclic lexicographic order} on recurrent vectors we have also $(d-1,\ldots,d-1)<(d-1,\ldots,d-1,0)$.

\begin{lemma}
\label{lexorder}
If $u,v$ are recurrent configurations on $T_n$ that are constant on levels, write $u \leadsto v$ if $v$ follows $u$ in the cyclic lexicographic order on the set of recurrent vectors.     Then for every integer $k \geq 0$, we have
	\[ k\hat{r} \leadsto (k+1)\hat{r}. \]
\end{lemma}

Figure~\ref{lexorderfigure} demonstrates the lemma for a ternary tree of height $4$.

\begin{proof}
By (\ref{youcanaddhoweveryouwant}) we have
	\[ (k+1)\hat{r} = k\hat{r} + \delta_r. \]
Thus if $k\hat{r} = (a_1, \ldots, a_{n-1})$ with $a_1<d-1$, then $(k+1)\hat{r} = (a_1+1,a_2,\ldots,a_{n-1})$ as desired.  Otherwise, if not all $a_i$ equal $d-1$, let $j>1$ be such that $a_1=\ldots=a_{j-1}=d-1$ and $a_j<d-1$.  Adding a chip at the root initiates the toppling cascade
	\[ \left( \begin{array}{c} d \\ d-1 \\ d-1 \\ \vdots \\ d-1 \\ d-1 \\ a_j \\ a_{j+1} \\ \vdots \\ a_{n-1} \end{array} \right) \rightarrow
	\left( \begin{array}{c} 0 \\ d \\ d-1 \\ \vdots \\ d-1 \\ d-1 \\ a_j \\ a_{j+1} \\ \vdots \\ a_{n-1} \end{array} \right) \rightarrow
	\left( \begin{array}{c} d-1 \\ 0 \\ d \\ \vdots \\ d-1 \\ d-1 \\ a_j \\ a_{j+1} \\ \vdots \\ a_{n-1} \end{array} \right) \rightarrow \ldots \rightarrow
	\left( \begin{array}{c} d-1 \\ d-1 \\ d-1 \\ \vdots \\ 0 \\ d \\ a_j \\ a_{j+1} \\ \vdots \\ a_{n-1} \end{array} \right) \rightarrow
	\left( \begin{array}{c} d-1 \\ d-1 \\ d-1 \\ \vdots \\ d-1 \\ 0 \\ a_j+1 \\ a_{j+1} \\ \vdots \\ a_{n-1} \end{array} \right), \]
as desired.  If all $a_i=d-1$ the cascade will travel all the way down, ending in $(d-1,\ldots,d-1,0)$ as desired.
\end{proof}

\begin{prop}
\label{rootsubgroup}
Let $T_n$ be the regular tree of degree $d$ and height $n$, and let $R(T_n)$ be the subgroup of $SP(T_n)$ generated by $\hat{r}$.  Then $R(T_n)$ consists of all recurrent configurations that are constant on levels, and its order is
	\begin{equation} \label{orderofrootsubgroup} \# R(T_n) = \frac{(d-1)^n - 1}{d-2}. \end{equation}
\end{prop}

\begin{proof}
Since the toppling rule is symmetric, all configurations in $R(T_n)$ are constant on levels.  The number of such recurrent configurations is the number of vectors of the form $(d-1,\ldots,d-1,0,a_j,\ldots,a_{n-1})$, with $a_i \in [d-1]$, which is
	\[  \sum_{j=0}^{n-1} (d-1)^j = \frac{(d-1)^n - 1}{d-2}. \]
Moreover, by Lemma~\ref{lexorder}, any such vector can be expressed as a multiple of $\hat{r}$, so $R(T_n)$ contains all the recurrent configurations that are constant on levels.
\end{proof}

Index the vertices of the $d$-regular tree of height $n$ by words of length $\leq n-2$ in the alphabet $\{1,\ldots,d-1\}$.  Let $\sigma_i$ be the automorphism of the tree given by
	\[ \sigma_i(w_1\ldots w_k) = w_1 \ldots (w_i+1) \ldots w_k \]
with the sum taken mod $d-1$; if $k<i$ then $\sigma_i(w)=w$.  Given a map $\alpha : [n-2] \rightarrow [d-1]$ let $\sigma_\alpha$ be the composition $\prod_{i=1}^{n-2} \sigma_i^{\alpha(i)}$.

If $\sigma$ is an automorphism of the form $\sigma_\alpha$, write $\sigma u$ for the chip configuration $\sigma u(x) = u(\sigma x)$.  Writing $u\oplus v$ for addition in the sandpile group and $u+v$ for the ordinary vector sum, we have
	\[ u\oplus v = u+v+\sum_{j=1}^m \Delta_{x_j} \]
Since $\sigma \Delta_x = \Delta_{\sigma x}$ we obtain
	\[ \sigma (u\oplus v) = \sigma(u)+\sigma(v) +  \sum_{j=1}^m \Delta_{\sigma x_j}. \]
The configuration on the right side is stable, recurrent, and $\equiv \sigma(u)+\sigma(v)$ $($mod $\Delta)$, so it is equal to $\sigma(u)\oplus \sigma(v)$.  Thus $\sigma$ is an automorphism of the sandpile group.

\begin{prop}
\label{regulardirectsum}
Let $T_n$ be the regular tree of degree $d$ and height $n$, and let $R(T_n)=(\hat{r})$ be the subgroup of $SP(T_n)$ generated by the root.  Then
	\[ SP(T_n) \simeq R(T_n) \oplus \frac{SP(T_{n-1}) \oplus \ldots \oplus SP(T_{n-1})}{(R(T_{n-1}),\ldots,R(T_{n-1}))} \]
with $d-1$ summands of $SP(T_{n-1})$ on the right side.
\end{prop}

\begin{proof}
Define $p : SP(T_n) \rightarrow SP(T_n)$ by
	\begin{equation} \label{symmetrization} p(u) = (d-1)^2 \sum_{\alpha: [n-2]\rightarrow [d-1]} \sigma_\alpha u. \end{equation}
By construction $p(u)$ is constant on levels, so the image of $p$ lies in $R(T_n)$ by Proposition~\ref{rootsubgroup}.  Given $u \in R(T_n)$, since $u$ is constant on levels we have $\sigma_\alpha(u)=u$ for all $\alpha$.  Since there are $(d-1)^{n-2}$ terms in the sum (\ref{symmetrization}), we obtain
	\[ p(u) = (d-1)^n u = u \]
where the second inequality follows from (\ref{orderofrootsubgroup}).  Thus $R(T_n)$ is a summand of $SP(T_n)$, and the result follows from Theorem~\ref{quotientisom}.
\end{proof}

\begin{figure}
\centering
\includegraphics[scale=1.2]{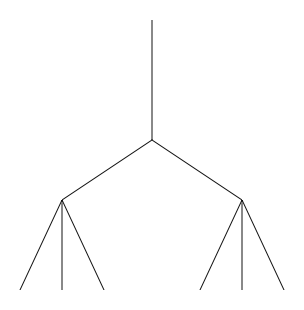}
\caption{A non-regular tree for which Proposition~\ref{regulardirectsum} fails.}
\label{counterexample}
\end{figure}

Proposition~\ref{regulardirectsum} fails for general trees.  For example, if $T$ is the tree consisting of a root with 2 children each of which have 3 children (Figure~\ref{counterexample}), then $\hat{r}=\left(\begin{array}{c} 2 \\ 3,3 \end{array} \right)$ has order $10$ and the element $x=\left(\begin{array}{c} 2 \\ 0,3 \end{array} \right)$ satisfies $4x=\hat{r}$, so $x$ has order $40$.  The total number of recurrent configurations is $4\cdot 4\cdot 3 - 8 = 40$, so $SP(T) \simeq \Z/40\Z$, and $R(T) \simeq \Z/10\Z$ is not a summand.  

Write $\Z_p^q$ as a shorthand for $(\Z/p\Z) \oplus \ldots \oplus (\Z/p\Z)$ with $q$ summands.

\begin{theorem}
\label{main}
Let $T_n$ be the regular tree of degree $d=a+1$ and height $n$, with leaves collapsed to the sink vertex and an edge joining the root to the sink.  Then
	\[ SP(T_n) \simeq \Z_{1+a}^{a^{n-3}(a-1)} \oplus \Z_{1+a+a^2}^{a^{n-4}(a-1)} \oplus \ldots \oplus \Z_{1+a+\ldots+a^{n-2}}^{a-1} \oplus \Z_{1+a+\ldots+a^{n-1}}. \]
\end{theorem}

\begin{proof}
Induct on $n$.  For the base case $n=2$ we have 
	\[ SP(T_2) = R(T_2) \simeq \Z_d = \Z_{1+a}. \]
Write $q_n = 1+a+\ldots + a^{n-1}$.  By Proposition~\ref{rootsubgroup}, the root subgroup $R(T_n)$ is cyclic of order $q_n$.  By Proposition~\ref{regulardirectsum} and the inductive hypothesis, it follows that
	\begin{align*} SP(T_n) &\simeq \Z_{q_n} \oplus \frac{SP(T_{n-1})^{\oplus a}}{\Z_{q_{n-1}}} \\
					&\simeq \Z_{q_n} \oplus \Z_{q_{n-1}}^{a-1} \oplus \Z_{q_{n-2}}^{a(a-1)} \oplus \ldots \oplus \Z_{q_2}^{a^{n-3}(a-1)}. \qed \end{align*}
\renewcommand{\qedsymbol}{}
\end{proof}

\subsection{Proof of Toumpakari's Conjecture}
\label{toumpakari}

As before write $a=d-1$ and
	\[ q_n = 1 + a + \ldots + a^{n-1}. \]
If $p$ is a prime not dividing $d(d-1)$, let $t_p$ be the least positive $n$ for which $p|q_n$.  Then
	\[ t_p = \begin{cases} p & \text{if } a \equiv 1 ~(\text{mod }p) \\
					\text{ord}_p(a), & \text{else}. \end{cases} \]
Here $\text{ord}_p(a)$ is the least positive $k$ for which $p|a^k-1$.  Note that $p|q_n$ if and only if $t_p|n$.  The following result was conjectured by E. Toumpakari in \cite{Toumpakari} (where the factor of $d-2$ was left out, presumably an oversight).

\begin{theorem} 
Let $B_n$ be the ball of radius $n$ in the $d$-regular tree, with each leaf connected by $d-1$ edges to the sink, but with  no edge connecting the root to the sink.  Let $p$ be a prime not dividing $d(d-1)$, and let $S_p(n)$ be the Sylow-$p$ subgroup of the sandpile group $SP(B_n)$.  Then
\[ \text{\em rank}(S_p(n)) = \begin{cases} d(d-2) \sum_{\substack{m<n \\ m \equiv n \,(\text{\em mod } t_p)}} (d-1)^m, 
					& \text{\em if } n \not\equiv -1 ~(\text{\em mod }  t_p); \\
			 d(d-2) \sum_{\substack{m<n \\ m \equiv n \,(\text{\em mod } t_p)}} (d-1)^m + d-1, 
					& \text{\em if } n \equiv -1 (\text{\em mod } t_p). \end{cases} \]
\end{theorem}

\begin{proof}
By Theorem~\ref{quotientisom} we have 
	\[ SP(B_n)/(\hat{r}) \simeq \frac{SP(T_{n+1}) \oplus \ldots \oplus SP(T_{n+1})}{(R(T_{n+1}), \ldots, R(T_{n+1}))} \]
with $d$ summands.  By Proposition~\ref{rootsubgroup} we have $R(T_{n+1}) \simeq \Z_{q_{n+1}}$, so from Theorem~\ref{main}
	\begin{equation} \label{balldecomp} SP(B_n)/(\hat{r}) \simeq \Z_{q_{n+1}}^a \oplus \Z_{q_n}^{(a-1)(a+1)} \oplus \Z_{q_{n-1}}^{(a-1)a(a+1)} \oplus \ldots \oplus \Z_{q_2}^{(a-1)a^{n-2}(a+1)}. \end{equation}
By Proposition~7.2 of \cite{Toumpakari}, the root subgroup $(\hat{r})$ of $SP(B_n)$ has order $d(d-1)^n$.  Thus for $p$ not dividing $d(d-1)$ the Sylow $p$-subgroup of $SP(B_n)$ is the same as that of the quotient $SP(B_n)/(\hat{r})$.  Each summand $\Z_{q_k}$ in (\ref{balldecomp}) contributes $1$ to the rank of $S_p(n)$ if $t_p|k$ and $0$ otherwise.  If $n \not\equiv -1$ $($mod $t_p)$, the total rank is therefore
	\begin{align*} \text{rank}(S_p(n)) &= \sum_{\substack{ 2 \leq k \leq n \\ t_p|k}} (a-1)a^{n-k}(a+1) \\
						&= d(d-2) \sum_{\substack{0 \leq m \leq n-2 \\ m \equiv n \,(\text{mod }t_p)}} (d-1)^m. \end{align*}
In the case that $n \equiv -1$ $($mod $t_p)$, the first summand $\Z_{q_{n+1}}^a$ in (\ref{balldecomp}) contributes an additional rank $a=d-1$ to $S_p(n)$.
\end{proof}

\section{The Rotor-Router Model on Trees}
\label{rotortree}

This section is devoted to proving Theorems~\ref{aggregintro} and~\ref{escapeseqs}.  

\subsection{The Rotor-Router Group}
\label{rotorgroup}

In this section we define the {\it rotor-router group} of a graph and show it is isomorphic to the sandpile group.  
This isomorphism, Theorem~\ref{groupisom}, is mentioned in the physics literature; see \cite{PDDK, PPS}.  To our knowledge the details of the proof are not written down anywhere.  While our main focus is on the tree, the isomorphism is just as easily proved for general graphs, and it seems to us worthwhile to record the general proof here.

Let $G$ be a strongly connected finite directed graph without loops.
Fix a sink vertex $s$ in $G$, and write $Rec(G)$ for the set of oriented spanning trees of $G$ rooted at the sink.
Given a configuration of rotors $T$, write $e_x(T)$ for the configuration resulting from starting a chip at $x$ and letting it walk according to the rotor-router rule until it reaches the sink.  (Note that if the chip visits a vertex infinitely often, it visits all of its neighbors infinitely often; since $G$ is strongly connected, the chip eventually reaches the sink.)  We view $T$ as a subgraph of $G$ in which every vertex except the sink has out-degree one.  Note that such a subgraph is an oriented spanning tree rooted at the sink if and only if it has no oriented cycles.

\begin{lemma}
If $T \in Rec(G)$, then $e_x(T) \in Rec(G)$.
\end{lemma}

\begin{proof}
Let $Y$ be any collection of vertices of $G$.  If the chip started at $x$ reaches the sink without ever visiting $Y$, then the rotors at vertices in $Y$ point the same way in $e_x(T)$ as they do in $T$, so they do not form an oriented cycle.  If the chip does visit $Y$, let $y\in Y$ be the last vertex it visits.  Then either $y=s$, or the rotor at $y$ points to a vertex not in $Y$; in either case, the rotors at vertices in $Y$ do not form an oriented cycle.
\end{proof}

We will need slightly more refined information about the intermediate states that occur before the chip falls into the sink. These states may contain oriented cycles, but only of a very restricted form.  For a vertex $x$ we write $Cyc_x(G)$ for the set of rotor configurations $U$ such that
	\begin{itemize}
	\item[(i)] $U$ contains an oriented cycle; and
	\item[(ii)] If the rotor $U(x)$ is deleted, the resulting configuration contains no oriented cycles.
	\end{itemize}

\begin{lemma}
\label{cyclecriterion}
Starting from a rotor configuration $T_0 \in Rec(G)$ with a chip at $x_0$, let $T_k$ and $x_k$ be the rotor configuration and chip location after $k$ steps.  Then
\begin{itemize}
\item[(i)] If $T_k \notin Rec(G)$, then $T_k \in Cyc_{x_k}(G)$.
\item[(ii)] If $T_k \in Rec(G)$, then $x_k \notin \{x_0,\ldots,x_{k-1}\}$.
\end{itemize}
\end{lemma}

\begin{proof}
(i) It suffices to show that any oriented cycle in $T_k$ contains $x_k$.  Let $Y$ be any set of vertices of $G$ not containing $x_k$.  If $Y$ is disjoint from $\{x_0, \ldots, x_{k-1}\}$, then the rotors at vertices in $Y$ point the same way in $T_k$ as they do in $T_0$, so they do not form an oriented cycle.  Otherwise, let $y\in Y$ be the vertex visited latest before time $k$.  The rotor at $y$ points to a vertex not in $Y$, so the rotors at vertices in $Y$ do not form an oriented cycle.

(ii) Suppose $x_k \in \{x_0,\ldots,x_{k-1}\}$.  Let $y_0=x_k$, and for $i=0,1,\ldots$ let $y_{i+1} = T_k(y_i)$.  Then the last exit from $x_k$ before time $k$ was to $y_1$, and by induction if $y_1, \ldots, y_{i-1}$ are disjoint from $x_k$, then $y_{i-1}$ was visited before time $k$, and the last exit from $y_{i-1}$ before time $k$ was to $y_i$.  It follows that $y_i=x_k$ for some $i$, and hence $T_k$ contains an oriented cycle.
\end{proof}

\begin{lemma}
If $T_1, T_2 \in Rec(G)$ and $e_x(T_1)=e_x(T_2)$, then $T_1=T_2$.
\end{lemma}

\begin{proof}



We will show that $T$ can be recovered from $e_x(T)$ by reversing one rotor step at a time.
Given rotor configurations $U,U'$ and vertices $y,y'$, we say that $(U',y')$ is a predecessor of $(U,y)$ if a chip at $y'$ with rotors configured according to $U'$ would move to $y$ in a single step with resulting rotors configured according to $U$.  For each neighbor $z \rightarrow y$ with $U(z)=y$ there is a unique predecessor of the form $(U',z)$, which we will denote $P_z(U,y)$.

Suppose $(U,y)$ is an intermediate state in the evolution from $T$ to $e_x(T)$.  If $U \notin Rec(G)$, then by case (i) of Lemma~\ref{cyclecriterion} there is a cycle of rotors $U(y)=y_1, U(y_1)=y_2, \ldots, U(y_n)=y$.  If $U(z)= y$ and $z \neq y_n$, then $z$ is not in this cycle, so the predecessor $P_z(U,y)$ has a cycle disjoint from its chip location.  Thus $P_z(U,y)$ does not belong to $Rec(G)$ or to $Cyc_z(G)$, so by Lemma~\ref{cyclecriterion} it cannot be an intermediate state in the evolution from $T$ to $e_x(T)$.  The state immediately preceding $(U,y)$ in the evolution from $T$ to $e_x(T)$ must therefore be $P_{y_n}(U,y)$.

Now suppose $U \in Rec(G)$.  By case (ii) of Lemma~\ref{cyclecriterion}, $U$ is the rotor configuration when $y$ is first visited.  If $y=x$, then $U=T$.  Otherwise, let $x=x_0 \rightarrow x_1 \rightarrow \ldots \rightarrow x_k=s$ be the path in $U$ from $x$ to the sink.  Then the last exit from $x$ before visiting $y$ was to $x_1$.  By induction, if $x_1, \ldots, x_{j-1}$ are different from $y$, then $x_{j-1}$ was visited before $y$ and the last exit from $x_{j-1}$ before visiting $y$ was to $x_j$.  It follows that $x_j=y$ for some $j$, and the state immediately preceding $(U,y)$ must be $P_{x_{j-1}}(U,y)$.
\end{proof}

Thus for any vertex $x$ of $G$, the operation $e_x$ of adding a chip at $x$ and routing it to the sink acts invertibly on the set of states $Rec(G)$ whose rotors form oriented spanning trees rooted at the sink.  It is for this reason that we call these states recurrent.  We define the {\it rotor-router group} $RR(G)$ as the subgroup of the permutation group of $Rec(G)$ generated by $\{e_x\}_{x\in G}$.  Note that if there are two (indistinguishable) chips on $G$ and each takes a single step according to the rotor-router rule, the resulting rotor configuration does not depend on the order of the two steps.  Thus the operators $e_x$ commute, and the group $RR(G)$ is abelian; for a general discussion of this property, which is shared by a number of models including the abelian sandpile and the rotor-router, see \cite{DF}.

\begin{lemma}
\label{transitivity}
$RR(G)$ acts transitively on $Rec(G)$.
\end{lemma}

\begin{proof}
Given $T_1, T_2 \in Rec(G)$, let $u(x)$ be the number of rotor turns needed to get from $T_1(x)$ to $T_2(x)$.  Let $v(x)$ be the number of chips ending up at $x$ if $u(y)$ chips start at each vertex $y$, with rotors starting in configuration $T_1$, and each chip takes a single step.  After each chip has taken a step, the rotors are in configuration $T_2$, hence
		\[ \left( \sum_{x \in V(G)} u(x)e_x \right) T_1 =  \left( \sum_{x \in V(G)} v(x)e_x \right) T_2. \]
Letting $g = \sum_{x \in V(G)} (u(x)-v(x))e_x$ we obtain $T_2 = gT_1$.
\end{proof}

\begin{theorem}
\label{groupisom}
Let $G$ be a strongly connected finite directed graph without loops, let $RR(G)$ be its rotor-router group, and $SP(G)$ its sandpile group.  Then $RR(G) \simeq SP(G)$.
\end{theorem}

\begin{proof}
Let $V$ be the vertex set of $G$.  Recall \cite{CR} 
that the sandpile group can be expressed as the quotient
	\[ SP(G) \simeq \Z^V / (s,\Delta_x)_{x\in V} \]
where $s \in V$ is the sink and
	\[ \Delta_x = \sum_{y\leftarrow x} y -  \text{outdeg}(x) x. \]
Define $\phi : \Z^{G} \rightarrow RR(G)$ by
	\[ \phi(u) = \sum_{x \in V(G)} u(x)e_x. \]
Since adding outdeg$(x)$ chips at $x$ and letting each chip take one step results in one chip at each neighbor $y \leftarrow x$ with the rotors unchanged, we have
	\[ \text{outdeg}(x) e_x = \sum_{y \leftarrow x} e_y. \]
Thus $\phi(\Delta_x)=0$ and $\phi$ descends to a map $SP(G) \rightarrow RR(G)$.  This map is trivially surjective; to show it is injective, by Lemma~\ref{transitivity} we have
	\[ \# RR(G) \geq \# Rec(G) = \# SP(G), \]
where the equality on the right is the matrix-tree theorem \cite{Stanley}. 
\end{proof}

\subsection{Aggregation on Regular Trees}
\label{rotoraggtree}

Let $T$ be the infinite $d$-regular tree.  Fix an origin vertex $o$ in $T$.  In {\it rotor-router aggregation}, we grow a cluster of points in $T$ by repeatedly starting chips at the origin and letting them walk until they exit the cluster.  Beginning with $A_1 = \{o\}$, define the cluster $A_n$ inductively by
	\[ A_n = A_{n-1} \cup \{x_n\}, \qquad n >1. \]
where $x_n \in T$ is the endpoint of a rotor-router walk started at $o$ and stopped on first exiting $A_{n-1}$.  We do not change the positions of the rotors when adding a new chip. 

In this section we use the group isomorphism proved in the last section to show that $A_n$ is a perfect ball for suitable values of $n$ (Theorem~\ref{aggregintro}).  The proof makes crucial use of the calculation of the order of the subgroup of the sandpile group of $T_n$ generated by the root, Proposition~\ref{rootsubgroup}.

A function $H$ on the vertices of a directed graph $G$ is {\it harmonic} if
	\[ H(x) = \frac{1}{\text{outdeg}(x)} \sum_{y \leftarrow x} H(y) \]
for all vertices $x$.

\begin{lemma}
\label{HPinvariant}
Let $H$ be a harmonic function on the vertices of $G$.  Suppose chips on $G$ can be routed, starting with $u(x)$ chips at each vertex $x$ and ending with $v(x)$ chips at each vertex $x$, in such a way that the initial and final rotor configurations are the same.  Then
	\[ \sum_{x \in V(G)} H(x) u(x) = \sum_{x \in V(G)} H(x) v(x). \]
\end{lemma}

\begin{proof}
Let $u=u_0, u_1, \ldots, u_k=v$ be the intermediate configurations.  If $u_{i+1}$ is obtained from $u_i$ by routing a chip from $x_i$ to $y_i$, then
	\begin{equation}\label{stepbystep} \sum_{x \in V(G)} H(x) (u(x)-v(x)) = \sum_i H(x_i)-H(y_i). \end{equation}
If the initial and final rotor configurations are the same, then each rotor makes an integer number of full turns, so the sum in (\ref{stepbystep}) can be written 
	\[ \sum_i H(x_i)-H(y_i) = \sum_{x \in V(G)} N(x) \sum_{y \leftarrow x} (H(x)-H(y)) \]
where $N(x)$ is the number of full turns made by the rotor at $x$.  Since $H$ is harmonic, the inner sum on the right vanishes.
\end{proof}

Let $T_n$ be the regular tree of degree $d$ and height $n$, with an edge added from the root $r$ to the sink $o$.  Denote by $(X_t)_{t\geq 0}$ the simple random walk on $T_n$, and let $\tau$ be the first hitting time of the set consisting of the leaves and the sink.  Fix a leaf $z$ of $T_n$, and let
	\begin{equation} \label{ourharmonicfunction} H(x) = \PP_x(X_\tau = z) \end{equation}
be the probability that random walk started at $x$ and stopped at time $\tau$ stops at $z$.  

The quantity $H(r)$ can be computed by a standard martingale argument.  Recall that the process
	\[ M_t = a^{-|X_t|} \]
is a martingale, where $a=d-1$ and $|x|$ denotes the distance from $x$ to the sink.  Since $M_t$ has bounded increments and $\EE_r \tau < \infty$, we obtain from optional stopping
	\[ a^{-1} = \EE_r M_0 = \EE_r M_\tau = p + (1-p) a^{-n} \]
where $p = \PP_r(X_\tau = o)$.  Solving for $p$ we obtain
	\begin{equation} \label{gamblersruin} \PP_r(X_\tau=o) = \frac{a^{n-1}-1}{a^n-1}. \end{equation}
In the event that the walk stops at a leaf, by symmetry it is equally likely to stop at any leaf.  Since there are $a^{n-1}$ leaves, we obtain from (\ref{gamblersruin})
	\begin{equation} \label{hittingprobofleaf} H(r) = \frac{1 - \PP_r(X_\tau = o)}{a^{n-1}}
										= \frac{a-1}{a^n-1}. \end{equation}

\begin{lemma}
\label{exitmeasure}
Let $a=d-1$.  If the initial rotor configuration on $T_n$ is acyclic, then starting with $\frac{a^n-1}{a-1}$ chips at the root, and stopping each chip when it reaches a leaf or the sink, exactly one chip stops at each leaf, and the remaining $\frac{a^{n-1}-1}{a-1}$ chips stop at the sink.  Moreover the starting and ending rotor configurations are identical.
\end{lemma}

\begin{proof}
By Theorem~\ref{groupisom} and Proposition~\ref{rootsubgroup}, the element $e_r \in RR(T_n)$ has order $\frac{a^n-1}{a-1}$, so the starting and ending rotor configurations are identical.  Fix a leaf $z$ and let $H$ be the harmonic function given by (\ref{ourharmonicfunction}).  By Lemma~\ref{HPinvariant} and (\ref{hittingprobofleaf}), the number of chips stopping at $z$ is 
	\[  \sum H(x) v(x) = \sum H(x) u(x) = \frac{a^n-1}{a-1} H(r) = 1. \]
Since there are $a^{n-1}$ leaves, the remaining $\frac{a^n-1}{a-1}-a^{n-1} = \frac{a^{n-1}-1}{a-1}$ chips stop at the sink.
\end{proof}

Let $T$ be the infinite $d$-regular tree.  The ball of radius $\rho$ centered at the origin in $o \in T$ is
	\[ B_\rho = \{x \in T ~:~ |x| \leq \rho \} \]
where $|x|$ is the length of the shortest path from $o$ to $x$.  Write
	\[ b_\rho = \# B_\rho = 1 + (a+1) \frac{a^\rho-1}{a-1}. \]
As the following result shows, provided we start with an acyclic configuration of rotors, the rotor-router aggregation cluster $A_n$ is a perfect ball at those times when an appropriate number of chips have aggregated.  It follows that at all other times, the cluster is as close as possible to a ball: if $b_\rho<n<b_{\rho+1}$ then $B_\rho \subset A_n \subset B_{\rho+1}$.

\begin{theorem}
\label{treecirc}
Let $A_n$ be the region formed by rotor-router aggregation on the infinite $d$-regular tree, starting from $n$ chips at the origin.  If the initial rotor configuration is acyclic, then $A_{b_\rho} = B_\rho$ for all $\rho \geq 0$.
\end{theorem}

\begin{proof}
Define a modified aggregation process $A'_n$ as follows.  Stop the $n$-th chip when it either exits the occupied cluster $A'_{n-1}$ or returns to $o$, and let
	\[ A'_n = A'_{n-1} \cup \{x'_n\} \]
where $x'_n$ is the point where the $n$-th chip stops.  By relabeling the chips, this yields a time change of the original process, i.e.\ $A'_n = A_{f(n)}$ for some sequence $f(1), f(2), \ldots$.  Thus it suffices to show $A'_{c_\rho}=B_\rho$ for some sequence $c_1, c_2, \ldots$.  We will show by induction on $\rho$ that this is the case for
	\[ c_\rho = 1 + (a+1) \sum_{t=1}^{\rho} \frac{a^t-1}{a-1}, \]
and that after $c_\rho$ chips have stopped the rotors are in their initial state.  For the base case $\rho=1$, we have $c_1 = a+2=d+1$.  The first chip stops at $o$, and the next $d$ stop at each of the neighbors of $o$, so $A'_{d+1}=B_1$.  Since the rotor at $o$ has performed one full turn, it is back in its initial state.

Assume now that $A'_{c_{\rho-1}} = B_{\rho-1}$ and that the rotors are in their initial acyclic state.  Starting with $c_\rho - c_{\rho-1}$ chips at $o$ and letting each take one step, there are $\frac{a^\rho-1}{a-1}$ chips at each neighbor of $o$.  Since the rotor at $o$ has performed an integer number of full turns, the rotors remain in their initial acyclic state.  By Lemma~\ref{exitmeasure}, exactly one chip will stop at each leaf $z \in B_{\rho}-B_{\rho-1}$, and the remainder will stop at $o$.  Thus $A'_{c_\rho} = B_\rho$.  Moreover, by Lemma~\ref{exitmeasure}, once all chips have stopped, the rotors are once again in their initial state, completing the inductive step.
\end{proof}

\subsection{Recurrence and Transience}
\label{recurrenceandtransience}

\begin{figure}
	\centering
		\includegraphics[scale=.8]{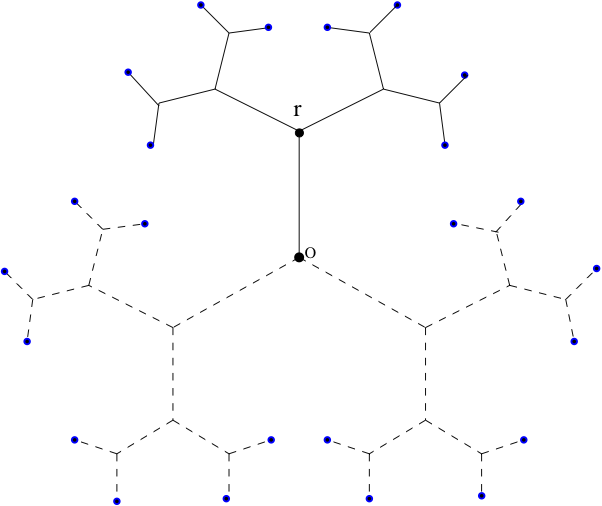}
		\\ ~ \\ ~ \\
		\includegraphics[scale=.75]{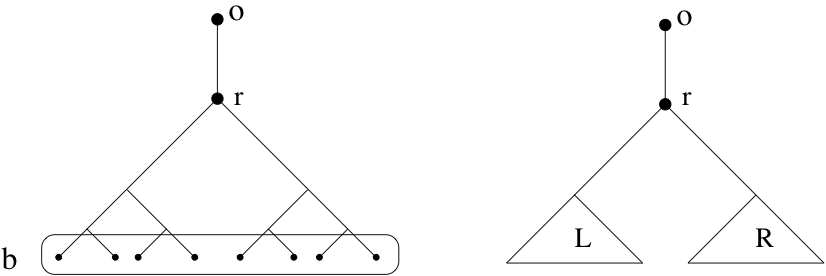}
	\caption{A branch of the regular ternary tree.}
	\label{fig:Tree1}
\end{figure}

In this section we explore questions of recurrence and transience for the rotor-router walk on regular trees. We aim to study to what extent the rotor-router walk behaves as a deterministic analogue of random walk.  We find that the behavior depends quite dramatically on the initial configuration of rotors.

Let $T_n$ be the $d$-regular tree of height $n$.  We collapse the leaves to a sink labeled $b$ for boundary, label the origin $o$ and treat it as both the source and a sink.  We will examine the hitting rates of the two sinks for various configurations of the model, and compare to the expected result for random walk.  



Since the origin is a sink, the recurrence and transience of chips on one principal branch of the tree are independent of the rotor configurations on the other branches. Therefore, as shown in Figure~\ref{fig:Tree1}, we focus on a single branch $Y_n$ of $T_n$.  By the root $r$ of $Y_n$ we will mean the unique node with $|r|=1$.  To each rotor direction we associate an index from $\{1,\ldots,d\}$, with direction $d$ corresponding to a rotor pointing to the parent vertex.  Rotors cycle through the indices in order.

\begin{lemma}
\label{ternarycase}
Let $Y_n$ be a principal branch of $T_n$, the regular ternary tree of height $n$.  If all rotors initially point in direction $1$, then the first $2^{n}-1$ chips started at $r$ alternate, the first stopping at $b$, the next stopping at $o$, the next at $b$, and so on until all the rotors again point in direction $1$.
\end{lemma}

\begin{proof}
Induct on $n$. For $n=2$ the result is obvious as there is only one rotor, which sends the first chip in direction $2$ to $b$, the next chip up in direction $3$ to $o$, and the third chip in direction $1$ to $b$, at which point the rotor is again in its initial state.

Now suppose that the lemma holds for $Y_{n-1}$.  Let $L$ and $R$ be the two principal branches of $Y_n$.  We think of $L$ and $R$ as each having a rotor that points either to $b$ or back up to $r$. The initial state of these rotors is pointing to $r$. The first
chip is sent from the origin to $R$, which by induction sends it to $b$. Note that the root rotor is now pointing towards $R$, the $R$-rotor is pointing to $b$, and the $L$-rotor is pointing to $r$ (Figure~\ref{fig:FourSteps}a). We now observe that the next four chips form a
pattern that will be repeated. The second chip is sent directly to $o$ (Figure~\ref{fig:FourSteps}b) and
the third chip is sent to $L$ which sends it to $b$ (Figure~\ref{fig:FourSteps}c). The fourth chip
is sent to $R$, but by induction this chip is returned and then it is sent
to $o$ (Figure~\ref{fig:FourSteps}d). Finally, the fifth chip is sent to $L$, returned, sent to $R$, and
through to $b$ (Figure~\ref{fig:FourSteps}e). Note that the root rotor is now again pointing towards $R$,
the $R$-rotor is again pointing to $b$, and the $L$-rotor is again pointing to $r$. In this cycle of four chips, the two branches $R$ and $L$ see two chips apiece. This cycle repeats $2^{ n - 2 } - 1$
times, and each subtree sees $2^{ n - 1 } - 2$ chips.

\begin{figure}
	\centering
		\includegraphics[scale=1]{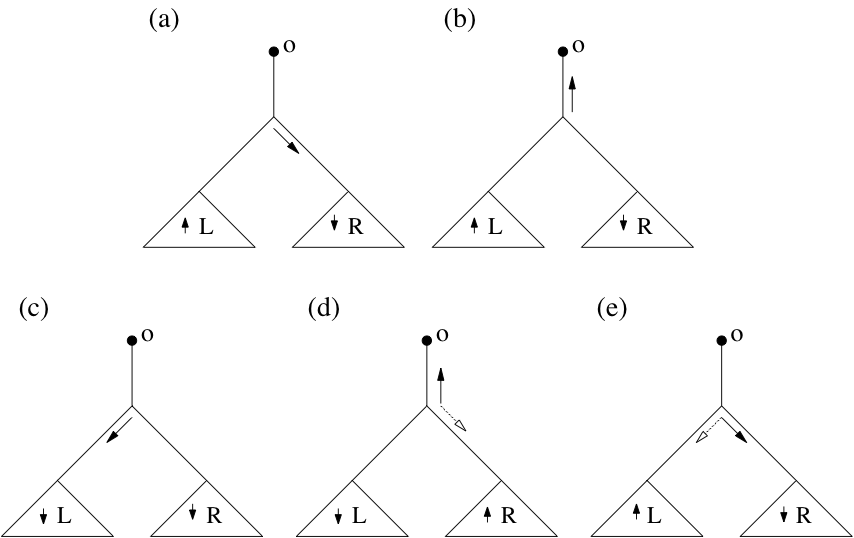}
	\caption{The four-chip cycle, which begins after the first chip has been routed to $b$.}
	\label{fig:FourSteps}
\end{figure}

Recall that we sent the first chip to $R$, so that it has seen a total of $2^{ n - 1 }
- 1$ chips. By induction, all the rotors in $R$ are in their initial
configuration. We have sent a total of $2^n - 3$ chips. The next chip
is sent to $o$, and the last to $L$ which sends it to $b$. Now $L$ has seen
$2^{ n - 1 } - 1$ chips so by induction all of its rotors are in their
initial configuration. The root rotor is pointing towards $L$, its initial
configuration. We have sent a total of $2^n - 1$ chips, alternating
between $b$ and $o$, and all of the rotors of $Y_n$ are in the initial
configuration, so the inductive step is complete.
\end{proof}

We remark that the obvious generalization of Lemma~\ref{ternarycase} to trees of degree $d > 3$ fails; indeed, we do not know of a starting rotor configuration on trees of higher degree which results in a single chip stopping at $o$ alternating with a string of $d-1$ chips stopping at $b$.

Consider now the case of the infinite ternary tree $T$.  A chip performing rotor-router walk on $T$ must either return to $o$ or escape to infinity visiting each vertex only finitely many times. 
Thus the state of the rotors after a chip has escaped to infinity is well-defined.  Therefore, we can properly define $R(m)$ as the number of chips that stop at $o$ after $m$ chips have executed the rotor-router walk beginning at $r$.  The following result shows that there is a configuration of the tree for which the rotor-router walk behaves as an exact quasirandom analogue to the random walk, in which chips exactly alternate returning to the origin with escaping to infinity.

\begin{prop}
\label{ternarythm}
Let $T$ be the infinite ternary tree, with principal branches labeled $Y^{(1)}$, $Y^{(2)}$, and $Y^{(3)}$ in correspondence with the direction indexing of the rotor at the origin. Set the rotors along the rightmost branch of $Y^{(3)}$ initially pointing in direction $2$, and all remaining rotors initially pointing in direction $1$. Let $E(m)$ be the expected number of chips that return to the origin if $m$ chips perform independent random walks on $T$. Let $R(m)$ be the number of chips that return to the origin if $m$ chips perform rotor-router walks on $T$. Then $|E(m) - R(m)| \leq \frac{1}{2}$ for all $m$.
\end{prop}

\begin{proof}
Lemma~\ref{ternarycase} implies that for the branches $Y^{(1)}$ and $Y^{(2)}$, the chips sent to a given branch alternate indefinitely with the first escaping to infinity, the next returning to $o$, and so on.  Likewise, chips sent to $Y^{(3)}$ will alternate indefinitely with the first returning to $o$, the next escaping to infinity, and so on. Since chips on the full tree $T$ are routed cyclically through the branches beginning with $Y^{(2)}$, we see that the chips too will alternate indefinitely between escaping to infinity and returning to the origin, with the first escaping to infinity. Thus $R(m) = \left\lfloor \frac{m}{2} \right\rfloor$. Taking $n \rightarrow \infty$ in (\ref{gamblersruin}) we obtain $E(m) = \frac{m}{2}$, and the result follows.
\end{proof}


\begin{lemma}
\label{finitetreesreturn}
Let $Y_n$ be a principal branch of $T_n$, the $d$-regular tree of height $n$.  If all rotors initially point in direction $d-1$, then the first $n-1$ chips return to $o$ before hitting height $n$ of $Y_n$.
\end{lemma}

\begin{proof}
Induct on $n$. For $n=2$, there is a single rotor which sends the first chip from $r$ to $o$.
Now suppose the lemma holds for $Y_{n-1}$. Let $Z_1, \ldots, Z_{d-1}$ be the principal branches of $Y_n$.  The first chip placed at $r$ is sent directly to $o$. By the inductive hypothesis, the first $n-2$ chips that are sent to each branch $Z_i$ are returned to $r$ before hitting height $n$ of $Y_n$.  Thus each of the next $n-2$ chips is sent to $Z_1$, returned to $r$, sent to $Z_2$, and so on until it is sent to $Z_{d-1}$, returned to $r$ and then routed to $o$.
\end{proof}

Our next result shows that, perhaps surprisingly, the rotors can be set up so as to make rotor-router walk on the $d$-regular tree recurrent.

\begin{prop}
\label{infinitetreesreturn}
Let $T$ be the infinite $d$-regular tree.  If all rotors initially point in direction $d-1$, then every chip in an infinite succession of chips started at the origin eventually returns to the origin.
\end{prop}

\begin{proof}
By Lemma~\ref{finitetreesreturn}, for each $n$, the $n$-th chip sent to each principal branch $Y$ returns to the origin before hitting height $n+1$ of $T$.
\end{proof}

We continue our exploration of recurrence and transience on the infinite ternary tree $T$, but now we allow for arbitrary rotor configurations. We again focus on a single branch $Y$. For a given rotor configuration we define the {\it escape sequence} for the first $n$ chips to be the binary word $a = a_{1}\ldots a_{n}$, where for each $j$, 
\[a_{j} = \begin{cases} 0, & \text{if the $j^{th}$ chip returns to the origin;} \\
				    1, & \text{if the $j^{th}$ chip escapes to infinity.} \end{cases} \]
As noted above, a chip cannot stay within a finite height indefinitely without returning to the origin, so $a$ is well-defined.\\											

We define a map $\psi$ on an escape sequence $a = a_{1} \ldots a_{n}$. First we rewrite $a$ as the concatenation of subwords $b_{1} \cdots b_{m}$ where each $b_{j} \in \{0,10,110\}$. Since at least one of any three consecutive chips entering $Y$ is routed directly upwards by the root rotor, at most two of any three consecutive letters in an escape sequence $a$ can be ones.  Therefore, any escape sequence can be factored in this way up to the possible concatenation of an extra $0$. Now we define $\psi(a) = (c,d)$ by
\[(c_{j},d_{j}) = \begin{cases} (0,0), & \text{if} \; b_{j} = 0\\
																(1,1), & \text{if} \; b_{j} = 110\\
																(0,1), & \text{if} \; b_{j} = 10 \; \text{and} \; \#\{i<j|b_{i}=10\} \; \text{is odd}\\
																(1,0), & \text{if} \; b_{j} = 10 \; \text{and} \; \#\{i<j|b_{i}=10\} \; \text{is even} \end{cases} 
\]

Now we define the map $\phi$ on a pair of binary words $c$ and $d$ each of length $m$ by $\phi(c,d) = b_{1} \cdots b_{m}$, where
\[b_{j} = \begin{cases} 0,& \text{if}\; (c_{j},d_{j}) = (0,0)\\
												10, & \text{if}\; (c_{j},d_{j}) = (1,0) \;\text{or}\; (0,1)\\
												110, & \text{if} \;(c_{j},d_{j}) = (1,1) \end{cases} \]
Note that $\phi$ is a left inverse of $\psi$, i.e.\ $\phi \circ \psi(a) = a$, up to possible concatenation of an extra $0$.											
												
\begin{lemma}
\label{phi}
Let $Y$ be a principal branch of the infinite ternary tree.  Fix a rotor configuration on $Y$ with the root rotor pointing up.  Let $c$ and $d$ be the escape sequences for the configurations on the left and right sub-branches of $Y$, respectively.  Then $\phi(c,d)$ is the escape sequence for the full branch $Y$.
\end{lemma}

\begin{proof}
We claim that each word $b_j$ is the escape sequence for the $j^{th}$ full rotation of the root rotor. 
Suppose first that $(c_j,d_j) = (0,0)$ so that $b_j = 0$. In this case, since each sub-branch will return the next chip it sees back to $r$, the next chip that enters the full tree will be sent up to the origin, and the root rotor will once again be pointing up to the origin. If $(c_j,d_j) = (1,0)$, the next chip entering the full tree will be routed to the left sub-branch where it escapes to infinity. The following chip will be routed to the right sub-branch and then back up to the origin, and the root rotor will once again be pointing up.  In this case we have $b_j = 10$.  If $(c_j,d_j) = (0,1)$, the next chip entering the full tree will be routed to the left sub-branch, back up to $r$, and then to the right sub-branch where it escapes to infinity. The following chip will be routed directly up to the origin leaving the root rotor pointed up once again.  Again, in this case $b_j = 10$.  Finally, if $(c_j,d_j) = (1,1)$, the next two chips entering the full branch will escape to infinity, one through each sub-branch. The following chip will be routed directly up to the origin, once again leaving the root rotor pointing up.  In this case we have $b_j = 110$. 
\end{proof}

We observe that if the root rotor is not pointing up and $c$ and $d$ are the escape sequences of the left and right sub-branches respectively, then $a = \phi(c,d)$ is \textit{not} the escape sequence of the full branch. In order to calculate the escape sequence of the full branch we define \textit{extended escape sequences} on the sub-branches, $c'$ and $d'$. Suppose the root rotor initially points to the left sub-branch. Then $c' = 0c$ and $d' = d$. Suppose the root rotor initially points to the right sub-branch. Then $c' = 0c$ and $d' = 0d$. Now $a = \phi(c',d')$ is the escape sequence of the full branch. 

We now introduce the condition that is central to characterizing which words can be escape sequences:
\renewcommand{\theequation}{$P_k$}
\begin{equation} 
 \text{any subword of length $2^{k}-1$ contains at most $2^{k-1}$ ones}
\end{equation}
\renewcommand{\theequation}{\arabic{equation}}
We next show that the map $\psi$ preserves this requirement. 

\begin{lemma}
\label{psi}
Let $a$ be a binary word satisfying $(P_{k})$ and let $\psi(a) = (c,d)$ as defined above. Then $c$ and $d$ each satisfy $(P_{k-1})$.
\end{lemma}

\begin{proof}
Let $c'$ be a subword of $c$ of length $2^{k-1}-1$ and let $d'$ be the corresponding subword of $d$. Let $a' = \phi(c',d')$, which is a subword of $a0$. The formula for $\phi$ guarantees that $a'$ has one zero for each letter of $c'$, so $a'$ has $2^{k-1}-1$ zeros.  Since the last letter of $a'$ is zero, and $a$ satisfies $(P_k)$, it follows that $a'$ has at most $2^{k-1}$ ones (else after truncating the final zero, the suffix of $a'$ of length $2^k-1$ has at most $2^{k-1}-2$ zeros, hence at least $2^{k-1}+1$ ones).

Let $m$ be the number of ones in $c'$.  Since the instances of $(0,1)$ and $(1,0)$ alternate in the formula for $\psi(a)=(c,d)$, it follows that $d'$ must have at least $m-1$ ones. Since the number of ones in $c'$ and $d'$ combined equals the number of ones in $a'$, we obtain $2m-1 \leq 2^{k-1}$, hence $m \leq 2^{k-2}$.  The same argument with the roles of $c$ and $d$ reversed shows that $d$ has at most $2^{k-2}$ ones.
\end{proof}

\begin{lemma}
\label{branch-escape}
Let $a=a_1\ldots a_n$ be a binary word of length $n$. Then $a$ is an escape sequence for some rotor configuration on the infinite branch $Y$ if and only if $a$ satisfies $(P_{k})$ for all $k$.
\end{lemma}

\begin{proof}
Suppose $a$ is an escape sequence. We prove that $a$ satisfies $(P_{k})$ for each $k$ by induction on $k$. That $a$ satisfies $(P_{1})$ is trivial. Now suppose that every escape sequence satisfies $(P_{k-1})$ and let $c$ and $d$ be the extended escape sequences of the left and right sub-branches respectively. As we saw above, $a = \phi(c,d)$ up to the possible concatenation of an extra zero.  Let $a'$ be a subword of $a$ of length $2^{k}-1$, and let $\psi(a')=(c',d')$.  Then there are words $c''$ and $d''$ each of which is a subword of $c$ or $d$, and which are equal to $c'$ and $d'$, respectively, except possibly in the first letter; moreover the first letters satisfy $c'_1 \leq c''_1$ and $d'_1 \leq d''_1$.

By the formula for $\psi$, the number of ones in $a'$ is the sum of the number of ones in $c'$ and $d'$. If $c'$ has length at most $2^{k-1}-1$, then since $c$ and $d$ satisfy $(P_{k-1})$, each of $c'$ and $d'$ has at most $2^{k-2}$ ones, and therefore $a'$ has at most $2^{k-1}$ ones.  On the other hand, if $c'$ has length at least $2^{k-1}$, then the number of zeros in $a'$ is at least $2^{k-1}-1$. Thus $a'$ has at most $2^{k-1}$ ones, so $a$ satisfies $(P_{k})$.

The proof of the converse is by induction on $n$. For $n=1$ the statement is trivial. Suppose that every binary word of length $n-1$ satisfying $(P_{k})$ for each $k$ is an escape sequence. Then by Lemma~\ref{psi}, $\psi(a) = (c,d)$ gives a pair of binary words each satisfying $(P_{k})$ for all $k$. If $c$ and $d$ have length $n-1$ or less, then they are escape sequences by induction, hence $a$ is an escape sequence by Lemma~\ref{phi}. If $c$ and $d$ are of length $n$ then the definition of $\psi$ implies that $a_j=0$ for all $j$, in which case $a$ is an escape sequence by Proposition~\ref{infinitetreesreturn}.
\end{proof}

We can now establish our main result characterizing all possible escape sequences on the infinite ternary tree.

\begin{theorem}
\label{escape}
Let $a=a_1\ldots a_n$ be a binary word. For $j \in \{1,2,3\}$ write $a^{(j)} = a_{j} a_{j+3} a_{j+6} \ldots$. Then $a$ is an escape sequence for some rotor configuration on the infinite ternary tree $T$ if and only if each $a^{(j)}$ satisfies $(P_{k})$ for all $k$.
\end{theorem}

\begin{proof}
Let $Y^{(1)}$, $Y^{(2)}$, and $Y^{(3)}$ be the three principal branches of $T$ assigned so that the rotor at the origin initially points to $Y^{(3)}$. Then $a$ is the escape sequence for $T$ if and only if $a^{(j)} = a_{j} a_{j+3} a_{j+6} \ldots $ is the escape sequence for $Y^{(j)}$. The result now follows from Lemma~\ref{branch-escape}.
\end{proof}


\chapter{Conjectures and Open Problems}

\section{Rotor-Router Aggregation}

\begin{figure}
\centering
\includegraphics{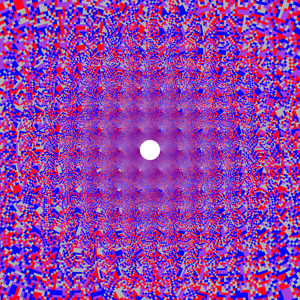}
\caption{Image of the rotor-router aggregate of one million particles under the map $z \mapsto 1/z^2$.  The colors represent the rotor directions.  The white disc in the center is the image of the complement of the occupied region.}
\label{invertedrotor}
\end{figure}

A number of intriguing questions remain unanswered about rotor-router aggregation in $\Z^d$.  While we have shown in Theorem~\ref{rotorcircintro} that the asymptotic shape of the rotor-router model with a single point source of particles is a ball, the near perfect circularity found in Figure~\ref{rotor1m} remains a mystery.  In particular, we do not know whether an analogue of Theorem~\ref{divsandcircintro} holds for the rotor-router model, with constant error in the radius as the number of particles grows.   

Equally mysterious are the patterns in the rotor directions evident in Figure~\ref{rotor1m}.  The rotor directions can be viewed as values of the odometer function mod $2d$, but our control of the odometer is not fine enough to provide useful information about the patterns.  If the rescaled occupied region $\sqrt{\pi/n} \,A_n$ is viewed as a subset of the complex plane, it appears that the monochromatic regions visible in Figure~\ref{rotor1m}, in which all rotors point in the same direction, occur near points of the form $(1+2z)^{-1/2}$, where $z = a+bi$ is a Gaussian integer (i.e.\ $a,b \in \Z$).  We do not even have a heuristic explanation for this phenomenon.  Figure~\ref{invertedrotor} shows the image of $A_{1,000,000}$ under the map $z \mapsto 1/z^2$; the monochromatic patches in the transformed region occur at lattice points.

L\'{a}szl\'{o} Lov\'{a}sz (personal communication) has asked whether the occupied region $A_n$ is simply connected, i.e.\ whether its complement is connected.  While Theorem~\ref{rotorcircintro} shows that $A_n$ cannot have any holes far from the boundary, we cannot answer his question at present.

Another question is whether our methods could be adapted to internal DLA to show that if $n=\omega_d r^d$, then with high probability $B_{r-c\log r} \subset I_n$, where $I_n$ is the internal DLA cluster of $n$ particles.  The current best bound is due to Lawler \cite{Lawler95}, who proves that with high probability $B_{r-r^{1/3}(\log r)^2} \subset I_n$.

\begin{figure}
\centering
\includegraphics[scale=.22]{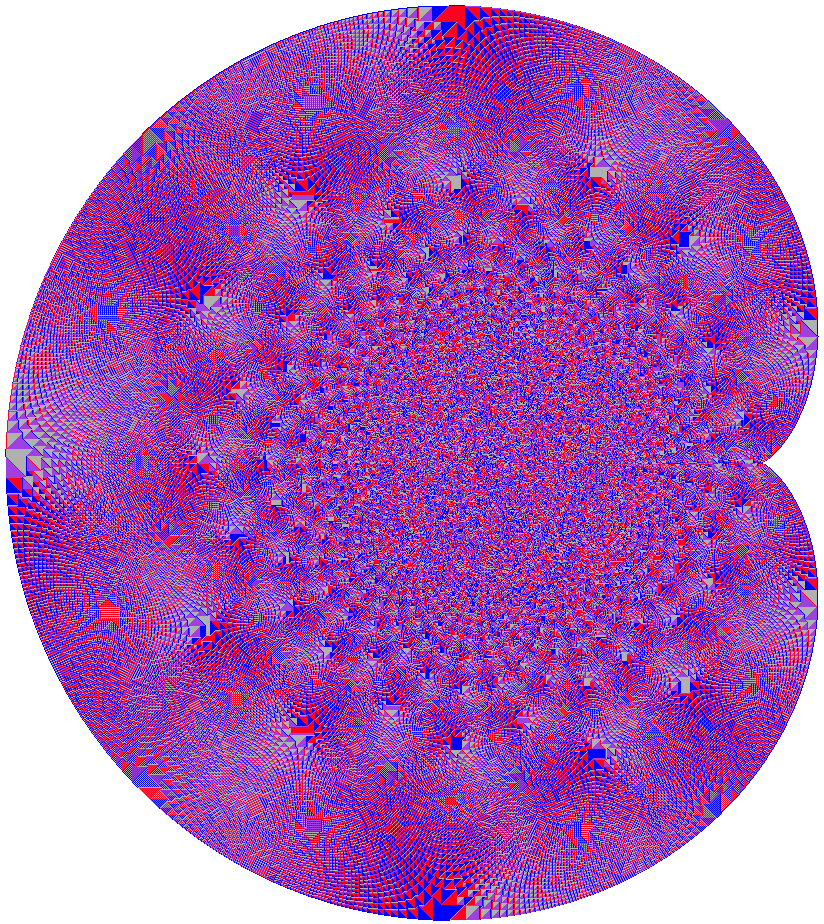} ~
\includegraphics[scale=.44]{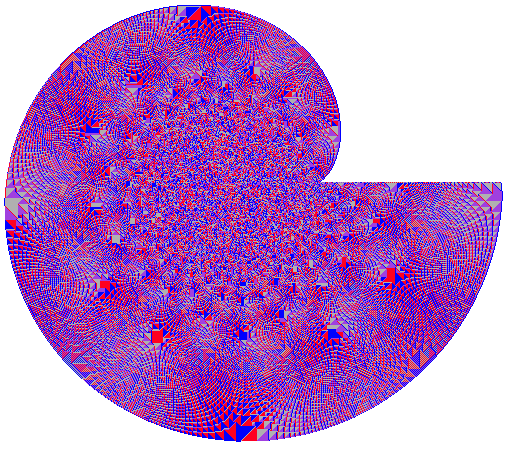}
\caption{Rotor-router aggregation started from a point source at the origin in $\Z^2$, modified so that sites on the positive $x$-axis (a) reflect particles back in the direction they came from; or, (b) send all particles in the downward direction.}
\label{cardioidandspiral}
\end{figure}

Finally, let us mention two simple variants of aggregation models which appear to have interesting limit shapes, but about which we cannot prove anything at present.  The following conjecture addresses a variant of the rotor-router model.  Analogous variants can also be defined for internal DLA and the divisible sandpile, and by analogy with Theorem~\ref{intromain} we expect all three models to yield the same limiting shape.

\begin{conjecture}
Let $A_n$ be the region formed from rotor-router aggregation starting with $n$ particles at the origin in $\Z^2$, if rotors on the positive $x$-axis reflect particles back in the direction they came from; all other rotors operate normally.  Then as $n \to \infty$, the asymptotic shape of the rescaled region $n^{-1/2} A_n$ is the cardioid in $\R^2$ whose boundary is given by the curve
	\[ \pi^2 |x|^4 - \pi |x|^2 + \frac{2\sqrt{2\pi}}{3\sqrt{3}} x_1 - \frac{1}{12} = 0. \]
\end{conjecture}

This apparent cardioid is pictured in Figure~\ref{cardioidandspiral} on the left.  If instead we alter the rotors on the positive $x$-axis so that they always point downward, we obtain the spiral-shaped region on the right.  Our simulations indicate that a limiting shape exists for this variant as well, but we do not have a conjectured formula for the boundary curve.

\section{Sandpile Aggregation}

While Theorem~\ref{sandpilecircintro} improves on the best-known bounds of \cite{LBR} and \cite{FR} for the shape of sandpile aggregation in $\Z^d$, it does not settle the question of existence of a limiting shape.  There is also an intriguing pattern in the shapes of the sandpile aggregates $S_{n,H}$ of $n$ particles in which every site starts with a hole of depth $H$.  These shapes are pictured in Figure~\ref{polygons} for $n=250,000$ and three different values of $H$.

\begin{figure}
\begin{center}
\begin{tabular}{ccc}
\includegraphics[scale=.172]{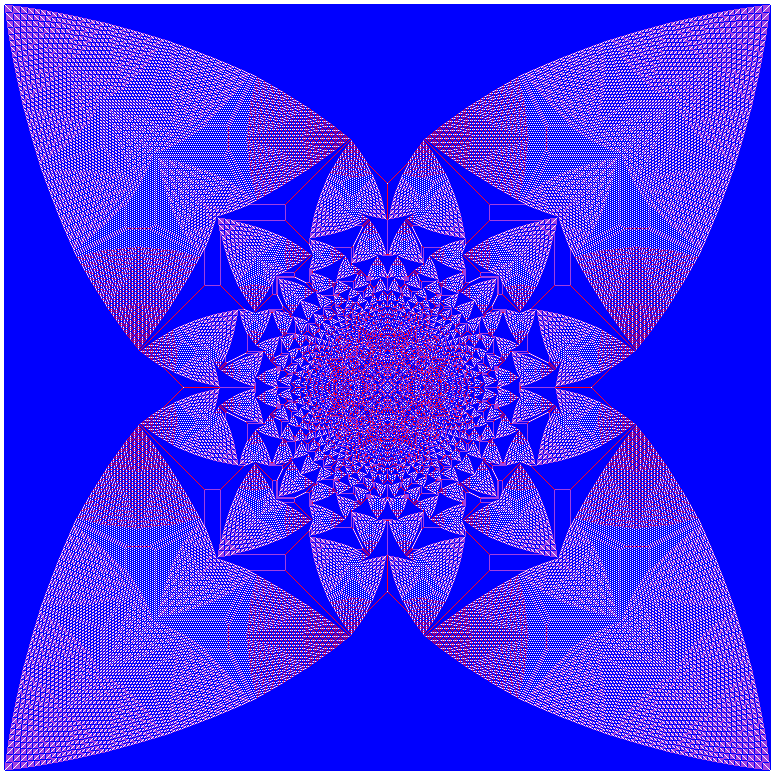}
&\includegraphics[scale=.28]{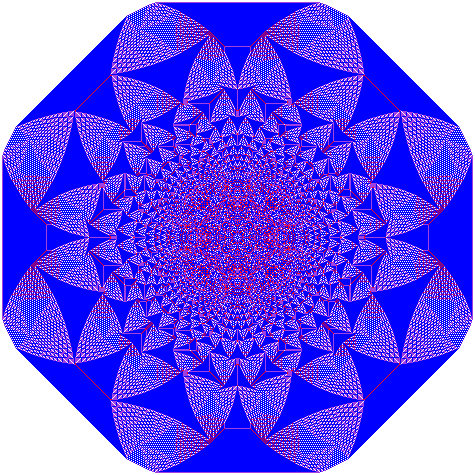}
&\includegraphics[scale=.359]{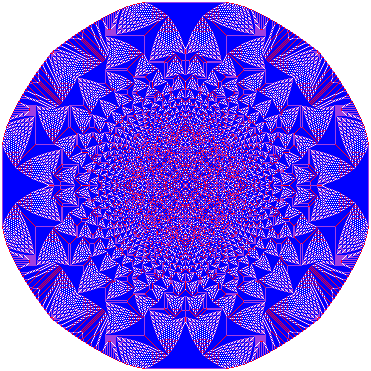} \\
$H=-2$ &$H=-1$ &$H=0$
\end{tabular}
\end{center}
\caption{Abelian sandpile shape $S_{n,H}$ started from a point source of $n=250,000$ particles in $\Z^2$, for three different values of the hole depth $H$.
}
\label{polygons}
\end{figure}

\begin{question}
\label{polygon}
In dimension two, is the limiting shape $n^{-1/2} S_{n,H}$ a regular polygon with $4H+12$ sides?
\end{question}

Simulations indicate a regular polygon with some rounding at the corners; it remains unclear if the rounded portions of the boundary become negligible in the limit.   If the limiting shape is not a polygon, it would still be very interesting to establish the weaker statement that it has the dihedral symmetry $D_{4H+12}$.

The only case that is even partly solved is $H=-2$: Fey and Redig \cite{FR} prove that $S_{n,-2}$ is a cube in $\Z^d$, but the growth rate of the cube is still not known.  It would be of interest to show that the volume of the cube $S_{n,-2}$ is of order $n$.










\end{document}